\documentclass[reqno]{amsart}
\usepackage{amsmath,amssymb,amsthm,amsfonts}
\usepackage{mathrsfs}
\usepackage{hyperref}
\usepackage{appendix}
\usepackage{graphicx}
\usepackage{setspace}
\usepackage{color}
\usepackage{multirow}

\newtheorem{lemma}{Lemma}[section]
\newtheorem{theorem}{Theorem}[section]

\newtheorem{proposition}{Proposition}[section]
\newtheorem{remark}{Remark}[section]
\newtheorem{corollary}{Corollary}[section]
\numberwithin{equation}{section}
\allowdisplaybreaks    
\newcommand{\p}{\partial}
\newcommand{\eps}{\varepsilon}

\begin{document}
\title[Boundary layer theory for steady MHD]{Global-in-$x$ stability of Prandtl layer expansions for steady magnetohydrodynamics flows over a moving plate}
\author[Shijin Ding]{Shijin Ding}
\address[S. Ding]{School of Mathematical Sciences, South China Normal University, Guangzhou, 510631, China}
\email{dingsj@scnu.edu.cn}
\author[Zhijun Ji]{Zhijun Ji}
\address[Z. Ji]{School of Mathematical Sciences, South China Normal University, Guangzhou, 510631, China}
\email{zhijunji@m.scnu.edu.cn}
\author[Zhilin Lin]{Zhilin Lin}
\address[Z. Lin]{School of Mathematical Science, Peking University, Beijing, 100871, China}
\email{zllin@m.scnu.edu.cn}

\date{\today}

\begin{abstract}
In this paper, we obtain the global-in-$x$ Sobolev stability of Prandtl layer expansions for 2-D steady incompressible MHD flows with shear outer ideal MHD flows $(1,0,\sigma,0)$ ($\sigma\geq 0$) on a moving plate. It is worth noticing that in the Sobolev sense, the degeneracy of the tangential magnetic field in our result is allowed, which is much different from both the unsteady case in \cite{LXYwell,LXYjus,LXY_ill} and the steady case in \cite{DLJ,DLXmhd}.
\end{abstract}
\maketitle
\tableofcontents

\vspace{-5mm}

\section{Introduction}
\subsection{Formulation of the problem}
This paper aims to verify the validity of Prandtl layer expansions for the following two dimensional steady incompressible magnetohydrodynamics(MHD) system in a domain $\Omega=\{(X,Y)\in [1,\infty)\times \mathbb{R}_+\}$
\begin{equation}\label{MHDflows_system}
\begin{cases}
  \mathbf{U}\cdot\nabla_{X,Y}\mathbf{U}-\mathbf{B}\cdot\nabla_{X,Y}\mathbf{B}+\nabla_{X,Y} P=\nu_1\eps\Delta_{X,Y}\mathbf{U}, \\
  \mathbf{U}\cdot\nabla_{X,Y}\mathbf{B}-\mathbf{B}\cdot\nabla_{X,Y}\mathbf{U}=\nu_2\eps\Delta_{X,Y}\mathbf{B}, \\
  \nabla_{X,Y}\cdot\mathbf{U}=0,\quad\nabla_{X,Y}\cdot\mathbf{B}=0,
\end{cases}
\end{equation}
in which $\mathbf{U}=(U,V)$ is the velocity field and $\mathbf{B}=(H,G)$ represents the magnetic field. For simplicity, the viscosity and the resistivity coefficients are taken to be of the same order of $\varepsilon$ with $\nu_1=\nu_2=1$.
We consider Dirichlet boundary condition on $\{Y=0\}$ (with moving speed $u_b>0$) for the velocity field
\begin{equation}\label{MHDflowsbc1}
  (U,V)|_{Y=0}=(u_b,0)=(1-\delta,0),~{\rm for~} 0<\delta\ll 1,
\end{equation}
and insulating boundary condition for the magnetic field
\begin{equation}\label{MHDflowsbc2}
  (H,G)|_{Y=0}=(0,0).
\end{equation}
 More about the insulating boundary condition can be found in \cite{BLMHD1,Gilbert1} and references therein.

Formally, letting $\varepsilon\rightarrow 0$, then \eqref{MHDflows_system} is reduced to the following ideal MHD equations
\begin{equation}\label{idealMHD_system}
\begin{cases}
  \mathbf{U}^0\cdot\nabla_{X,Y}\mathbf{U}^0-\mathbf{B}^0\cdot\nabla_{X,Y}\mathbf{B}^0+\nabla_{X,Y} P^0=0, \\
  \mathbf{U}^0\cdot\nabla_{X,Y}\mathbf{B}^0-\mathbf{B}^0\cdot\nabla_{X,Y}\mathbf{U}^0=0, \\
  \nabla_{X,Y}\cdot\mathbf{U}^0=0,\quad\nabla_{X,Y}\cdot\mathbf{B}^0=0.
\end{cases}
\end{equation}
To solve the ideal MHD system \eqref{idealMHD_system}, it suffices to impose the non-penetration conditions for both velocity and magnetic field on the boundary $\{Y=0\}$
 \begin{equation}\label{idealMHDbc}
  (\mathbf{U}^0\cdot\mathbf{n})|_{Y=0}= ( \mathbf{B}^0\cdot\mathbf{n})|_{Y=0}=0.
\end{equation}

In this paper, we focus on the following shear outer ideal MHD flows
\begin{equation}\label{idealMHD}
  (\mathbf{U}^0,\mathbf{B}^0,P^0)=(1,0,\sigma,0,0),
\end{equation}
where the constant $\sigma$ is chosen so that $0\leq\sigma\ll1$.

We will rewrite the MHD flows in the scaled boundary layer variable $(x,y)$
\begin{equation*}
  x=X,\quad y=\frac{Y}{\sqrt\eps}.
\end{equation*}
with the scaled unknowns as follows
\begin{equation}\label{scale}
  (U^{\eps},V^{\eps},H^{\eps},G^{\eps},P^{\eps})(x,y)= \left(U,\frac{V}{\sqrt\eps}, H,\frac{G}{\sqrt\eps},P\right)(X,Y),
\end{equation}
which preserve the divergence-free conditions. Then problem \eqref{MHDflows_system}-\eqref{MHDflowsbc2} is rewritten as
\begin{equation}\label{scalesystem}
\left\{
\begin{aligned}
  &(U^{\eps}\p_x+V^{\eps}\p_y)U^{\eps}-(H^{\eps}\p_x+G^{\eps}\p_y)H^{\eps}+\p_x P^{\eps}
  =\eps\p_{xx}U^{\eps}+\p_{yy}U^{\eps}, \\
  &(U^{\eps}\p_x+V^{\eps}\p_y)V^{\eps}-(H^{\eps}\p_x+G^{\eps}\p_y)G^{\eps}+\frac{\p_y P^{\eps}}{\eps}
  =\eps\p_{xx}V^{\eps}+\p_{yy}V^{\eps}, \\
  &(U^{\eps}\p_x+V^{\eps}\p_y)H^{\eps}-(H^{\eps}\p_x+G^{\eps}\p_y)U^{\eps}
  =\eps\p_{xx}H^{\eps}+\p_{yy}H^{\eps},\\
  &(U^{\eps}\p_x+V^{\eps}\p_y)G^{\eps}-(H^{\eps}\p_x+G^{\eps}\p_y)V^{\eps}
  =\eps\p_{xx}G^{\eps}+\p_{yy}G^{\eps},\\
  &\p_x U^\eps+\p_y V^\eps=0,\quad \p_x H^\eps+\p_y G^\eps=0,\\
  &(U^\eps,V^\eps)|_{y=0}=(u_b,0)=(1-\delta,0),\\
  &(H^\eps,G^\eps)|_{y=0}=(0,0).
\end{aligned}
\right.
\end{equation}

To formulate our problem, let us introduce the following asymptotic expansions
\begin{equation}\label{expansion}
  (U^{\eps},V^{\eps},H^{\eps},G^{\eps},P^{\eps})
  =(u_{app},v_{app},h_{app},g_{app},p_{app})
   +\eps^{\frac{n}{2}+\gamma}(u,v,h,g,p)
\end{equation}
for some constants $\gamma>0$. In the above expressions, $(u,v,h,g,p)$ are the error profiles of the exact solutions to problem (\ref{scalesystem}), and $(u_{app},v_{app},h_{app},g_{app},p_{app})$ are the approximate solutions defined by
\begin{equation}\label{app}
\left\{
\begin{aligned}
  &u_{app}=1+u_p^0(x,y)+\sum_{i=1}^{n}\eps^{\frac{i}{2}}[u_e^i(x,\sqrt\eps y)+u_p^i(x,y)],\\
  &v_{app}=v_p^0(x,y)+\sum_{i=1}^{n}\eps^{\frac{i-1}{2}}[v_e^i(x,\sqrt\eps y)+\sqrt{\eps}v_p^i(x,y)],\\
  &h_{app}=\sigma+h_p^0(x,y)+\sum_{i=1}^{n}\eps^{\frac{i}{2}}[h_e^i(x,\sqrt\eps y)+h_p^i(x,y)],\\
  &g_{app}=g_p^0(x,y)+\sum_{i=1}^{n}\eps^{\frac{i-1}{2}}[g_e^i(x,\sqrt\eps y)+\sqrt{\eps}g_p^i(x,y)],\\
  &p_{app}=p_p^0(x,y)+\sum_{i=1}^{n}\eps^{\frac{i}{2}}[p_e^i(x,\sqrt\eps y)+p_p^i(x,y)]\\
  &\qquad\quad +\sum_{i=1}^{n}\eps^{\frac{i}{2}}[\eps^{\frac{i}{2}}p_{e}^{i,a}(x,\sqrt\eps y)+\sqrt{\eps}p_{p}^{i,a}(x,y)],
\end{aligned}
\right.
\end{equation}
in which the leading order ideal MHD flows have been taken as the shear flows $(u_e^0,v_e^0,h_e^0,g_e^0)=(1,0,\sigma,0)$.

To analyze the problem, it is important to figure out the boundary conditions for each profile to match the exact boundary conditions \eqref{MHDflowsbc1} and \eqref{MHDflowsbc2}.

(i) For the leading order boundary layers $(u_p^0,v_p^0,h_p^0,g_p^0)$, the boundary conditions are prescribed as
\begin{align}\label{bc_0_prandtl}
  &(u_p^0,h_p^0)(x,0)=(-\delta,-\sigma),\quad (u_p^0,h_p^0)(x,\infty)=(0,0),\\
  &(u_p^0,h_p^0)(1,y)=(u_0^0,h_0^0)(y).
\end{align}
We suppose that the initial data $(u_0^0,h_0^0)(y)$ decay rapidly
\begin{equation}\label{u_p^0_initialdecay}
  \Vert \langle y\rangle ^m\p_y^j(u_0^0,h_0^0)(y)\Vert_{L^\infty}\leq C(m,j),~{\rm for~ any~ } m,j\in\mathbb{N},
\end{equation}
and enjoy the following small conditions
\begin{equation}\label{u_p^0_initialsmall}
  \Vert \langle y\rangle ^m\p_y^j(u_0^0,h_0^0)(y)\Vert_{L^\infty}\leq\mathcal{O}(\delta,\sigma;m,j),~{\rm for~ any~ } m\in\mathbb{N},~{\rm when~ }j=0,1,2,3.
\end{equation}
In addition, we assume that
\begin{equation}\label{u_p^0_initial_uh}
  1+u_0^0(y)\gg \sigma+h_0^0(y)\geq 0
\end{equation}
and
\begin{equation}\label{u_p^0_initial_u}
  1+u_0^0(y)\geq\vartheta_0 ~{\rm~uniform~in~} y ~{\rm~for~some~}\vartheta_0>0.
\end{equation}

(ii) Since the ideal MHD flows $(u_e^i,v_e^i,h_e^i,g_e^i) (1\leq i\leq n)$ satisfy elliptic system, we consider the following boundary conditions
\begin{align}\label{bc_i_idealmhd}
  &(v_e^i,g_e^i)(x,0)=-(v_p^{i-1},g_p^{i-1})(x,0),\\
  &(v_e^i,g_e^i)\rightarrow (0,0){\rm ~as~} Y\rightarrow\infty.
\end{align}

Note that we will derive the explicit expression for the ideal MHD profiles via Poisson extension to the boundary conditions on $\{Y=0\}$ without extra in-flow information on $\{x=1\}$.

(iii) For the boundary layer correctors $(u_p^i,v_p^i,h_p^i,g_p^i),~ 1\leq i\leq n-1$, which solve the parabolic system, the boundary conditions are taken as
\begin{align}\label{bc_i_prandtl}
  &(u_p^i,h_p^i)(x,0)=(-\overline u_e^i(x),-\overline h_e^i(x)),\\
  &(u_p^i,v_p^i,h_p^i,g_p^i)(x,\infty)=(0,0,0,0).
\end{align}
Meanwhile, the in-flow conditions are also needed
\begin{equation}\label{bc_i_prandtl_inflow}
  (u_p^i,h_p^i)(1,y)=(u_0^i,h_0^i)(y), \ 1\leq i\leq n-1.
\end{equation}

(iv) The final boundary layer profiles are considered to satisfy the following boundary conditions
\begin{align}\label{bc_n_prandtl}
  &(u_p^n,v_p^n,h_p^n,g_p^n)(x,0)=(-\overline u_e^n(x),0,-\overline h_e^n(x),0),\\
  &(u_p^n,h_p^n)(x,\infty)=(0,0), \ (u_p^n,h_p^n)(1,y)=(u_0^n,h_0^n)(y).
\end{align}
Note that the final boundary layer correctors $(\tilde v_p^n,\tilde g_p^n)$ are constructed by cutting off $v_p^n,g_p^n$(see \eqref{u_p^n_def} and \eqref{h_p^n_def} for details), and the in-flow data for $(\tilde u_p^n,\tilde h_p^n)$ will be given through analysis, which are bounded profiles with decay as $y\rightarrow \infty$.

(v) For the remainder profiles $(u,v,h,g)$, we impose the matching conditions
\begin{align}\label{bc_remainder}
  &(u,v,h,g)|_{y=0}=(u,v,h,g)|_{y\rightarrow\infty}=(0,0,0,0),\\
  &(u,v,h,g)|_{x=1}=(u,v,h,g)|_{x\rightarrow\infty}=(0,0,0,0).
\end{align}

Since we will construct the approximate solutions with large $n$, we need to impose the following compatibility conditions for the in-flow data
\begin{align}\label{cp_1}
(u_0^i,h_0^i)(0)=-(u_e^i,h_e^i)(1,0),\ \p_y^2(u_0^0,h_0^0)(0)=(0,0),.
\end{align}
The in-flow data also satisfy the rapid decay
\begin{align}\label{cp_2}
  \Vert \langle y\rangle^m (u_0^i,h_0^i)(y)\Vert_{L^\infty}\leq C(i,m), ~ 1\leq i\leq n.
\end{align}

Throughout this paper, we always assume that conditions \eqref{bc_0_prandtl}-\eqref{cp_2} hold and we may also need some higher-order compatibility conditions similar to \eqref{cp_1} for $\partial_y^j(u^j_0,h^j_0)(j\geq 2)$. However, we do not specify them everywhere for simplicity.

\subsection{Literature reviews}\label{sec1.3}
The singularity behavior of the inviscid limit for viscous fluids is a longstanding problem in fluid mechanics. Specifically, in our consideration, there is a contradiction to the convergence in the vanishing viscosity and resistivity limit process, since the tangential velocity and magnetic field of the viscous and diffusive MHD flows mismatch that of the ideal MHD flows on the boundary $\{y=0\}$. To correct this mismatch, the boundary layer corrector functions should be introduced in a thin layer with width $O(\sqrt\eps)$ near the boundary. This idea was initially proposed by Prandtl \cite{Prandtl} in 1904 for investigating the inviscid limit of the Navier-Stokes equations.

From the physical point of view, the inviscid limit problem is very important and has been extensively studied in the fluid mechanics literatures. To understand the problem, two topics are needed to be considered: the well-posedness theory of each profile (including the ideal flows and the boundary layer correctors) and the convergence from the viscous solutions to the approximate solutions.

For the Navier-Stokes equations, the known results of the inviscid limit problem are concluded in some special framework of analytic functions \cite{Maekawa,Sammartino1,Sammartino2,Zhang} and Gevrey class \cite{Gerard1,LWX,LY}. However, the validity for the Prandtl layer theory in finite regularity function spaces remains open for the unsteady cases. Roughly speaking, due to strong boundary layer effect (resulted from the no-slip boundary condition), it is very difficult to control the vorticitiy near the boundaries in Sobolev sense, see also the review papers \cite{review2,review1}. However, when the boundary layer is absent or very weak, there are extensive literatures on the vanishing viscosity limit, see among many other references \cite{weakbdary1,Masmoudi,xiaoxin}.

Compared with the Navier-Stokes flows, due to the presence of magnetic field, the problem in magnetohydrodynamics may be much more challenging. There are few literatures on the inviscid limit problem for the MHD equations. For the well-posednees theory of the MHD boundary layer equations, see for instance \cite{analyticMHD1,analyticMHD2} in the framework of analytic functions, and \cite{GevreyMHD1} in the Gevrey space for 2-D and 3-D problems. In Sobolev space, Liu, Xie and Yang \cite{LXYwell} established local well-posedness theory, under the assumption of nondegenerate tangential magnetic field. Gerard-Varet and Prestipino \cite{BLMHD1} also investigated the stabilizing effect of nondegenerate magnetic field. Recently, Chen, Ren, Wang and Zhang \cite{BLMHD2} proved long time existence for two dimensional MHD boundary layer equations.

Due to strong coupling effect between the velocity filed and magnetic field, it seems very difficult to verify the validity of the boundary layer expansion for MHD flows in Sobolev space. Wang and Xin \cite{WangXin} established uniform stability of the boundary layers for a class of initial data considering Dirichlet boundary conditions for both velocity and magnetic fields. Zero viscosity and diffusion vanishing limit of the three dimensional MHD system with perfectly conducting wall was investigated by Wu and Wang \cite{WangWu1}, they established the convergence in $L^2$-norm. Under the assumption of nondegenerate tangential magnetic field, Liu, Xie and Yang \cite{LXYjus} justified the Prandtl ansatz for MHD boundary layer and gave the $L^\infty$ convergence estimates in Sobolev sense. Wang and Wang \cite{WangWang1} studied the boundary layer problem for viscous MHD system with non-characteristic boundary conditions.

The first result about the validity of the boundary layer theory in Sobolev space for steady Navier-Stokes equations was obtained by Guo and Nguyen \cite{GN17}, in which the local-in-$x$ validity for the Prandtl layer theory was established for the shear outer flows $(u_e^0(Y),0)$ in the domain $[0,L]\times\mathbb{R}_+$. Similar results were extended to the rotating disk \cite{Iyerrotating}, the bounded domain \cite{Iyerbounded,LD} and the case with nonshear outer flows \cite{Iyernonshear}. The global in $x$ expansion over a moving plate was verified by Iyer in a series of works \cite{Iyerglobal1,Iyerglobal2,Iyerglobal3}. The moving boundary assumption was removed by Guo and Iyer \cite{GuoIyer}, see also Gao and Zhang \cite{GaoZhang}. Recently, Iyer and Masmoudi \cite{IyerMasmoudi} verified global-in-$x$ Prandtl boundary layer theory for a large class of boundary layers.

On the other hand, for steady MHD flows, the MHD boundary layer equations without magnetic diffusion was studied by Wang and Ma \cite{WangMa}, in which well-posedness and non-existence theory of the boundary layer system were obtained for different ratios of the magnetic field and the velocity field. Ding, Lin and Xie \cite{DLXmhd} verified the Prandtl boundary layer ansatz of the steady MHD flows with a moving boundary on the domain $[0,L]\times\mathbb{R}_+$ for the case with shear outer ideal MHD flows $(u_e^0(Y),0,h_e^0(Y),0)$. And we generalized the result to nonshear outer flows $(u_e^0,v_e^0,h_e^0,g_e^0)(X,Y)$ in \cite{DLJ}. Liu, Yang and Zhang \cite{LYZ} proved stability of shear flows of Prandtl type and manage the degeneracy on the boundary caused by no-slip boundary condition, which relies the spectral analysis of the linearized operator around the Prandtl-type shear flow.

It is interesting to note that the results of \cite{DLJ,DLXmhd} not only rely heavily on the nondegeneracy of the tangential magnetic field as in the unsteady cases, but also on the smallness of $L$, but the main purpose of this paper is to remove these two restrictions for a simple outer flow $(1, 0, \sigma, 0)$.

More clearly, we summarize the results for the boundary layer theory of MHD related to this paper by the following Table \ref{tab1}.
\begin{table}[!h]
\caption{Some results related to this paper}
\begin{tabular}{|c|c|cc|c|}
\hline
\multirow{2}{*}{\textbf{Problem}}                                                 & \multirow{2}{*}{\textbf{\begin{tabular}[c]{@{}c@{}}Tangential \\ magnetic field\end{tabular}}} & \multicolumn{2}{c|}{\textbf{\begin{tabular}[c]{@{}c@{}}Boundary \\ conditions\end{tabular}}}                                                      & \multirow{2}{*}{\textbf{Result}}                                                \\ \cline{3-4}
                                                                                   &                                                                                                 & \multicolumn{1}{c|}{\textbf{\begin{tabular}[c]{@{}c@{}} Velocity \\ field\end{tabular}} }& \textbf{\begin{tabular}[c]{@{}c@{}}Magnetic \\ field\end{tabular}}      &                                                                                  \\ \hline
\multirow{2}{*}{\begin{tabular}[c]{@{}c@{}}2-D unsteady \\ MHD flows\end{tabular}} & \begin{tabular}[c]{@{}c@{}}Non-degeneracy\end{tabular}                                        & \multicolumn{1}{c|}{No-slip}                                                    & \begin{tabular}[c]{@{}c@{}}Perfectly\\  conducting\end{tabular} & \begin{tabular}[c]{@{}c@{}}Stable \\ \cite{LXYwell,LXYjus}\end{tabular}                \\ \cline{2-5}
                                                                                   & Degeneracy                                                                                      & \multicolumn{1}{c|}{No-slip}                                                    & \begin{tabular}[c]{@{}c@{}}Insulating\\ (no-slip)\end{tabular}  & Ill-posed \cite{LXY_ill}                                                                      \\ \hline
\multirow{2}{*}{\begin{tabular}[c]{@{}c@{}}2-D steady \\ MHD flows\end{tabular}}   & \begin{tabular}[c]{@{}c@{}}Non-degeneracy\end{tabular}                                        & \multicolumn{1}{c|}{\begin{tabular}[c]{@{}c@{}}Moving\\ plate\end{tabular}}     & \begin{tabular}[c]{@{}c@{}}Perfectly\\  conducting\end{tabular} & \begin{tabular}[c]{@{}c@{}}Local-in-$x$\\ Stable\\ \cite{DLJ,DLXmhd}\end{tabular}                \\ \cline{2-5}
                                                                                   & \begin{tabular}[c]{@{}c@{}}Degeneracy \\ \textbf{allowed }\end{tabular}                                   & \multicolumn{1}{c|}{\begin{tabular}[c]{@{}c@{}}Moving\\ plate\end{tabular}}     & \begin{tabular}[c]{@{}c@{}}Insulating\\ (no-slip)\end{tabular}  & \begin{tabular}[c]{@{}c@{}}Global-in-$x$\\ Stable  \\ (this paper)\end{tabular} \\ \hline
\end{tabular}
\label{tab1}
\end{table}

\subsection{Main result and comments}\label{sec1.2}
Our main result in this paper is stated as follows.
\begin{theorem}\label{maintheorem}
Consider the outer ideal MHD flows $(u_e^0,v_e^0,h_e^0,g_e^0)=(1,0,\sigma,0)$ with $0\leq \sigma\ll 1 $, and assume that each profile in the expansion \eqref{expansion} satisfies the conditions stated as in \eqref{bc_0_prandtl}-\eqref{cp_2}. For sufficiently small mismatch (measured by constant $\delta$) between the viscous MHD solutions and the ideal MHD flows on boundary $\{Y=0\}$ and the small viscosity coefficient $\eps\ll\eps_*$, then there exists $n\in \mathbb{N}$ such that \eqref{expansion} is valid globally on $\Omega=[1,\infty)\times\mathbb{R}_+$. The remainder solution $(u,v,h,g)$ in \eqref{expansion} is unique on the domain $\Omega=[1,\infty)\times\mathbb{R}_+$ with
\begin{equation}\label{maintheorem_estimate}
  \eps^{\frac{n}{2}+\gamma}\|(u,h)x^\frac{1}{4}+\sqrt\eps(v,g)x^{\frac{1}{2}}\|_{L^\infty}\lesssim \eps^{\frac{1}{4}-\gamma-\kappa},
\end{equation}
in which $\eps_*=\delta$ if $\sigma=0$ and $\eps_*=
\min\{\delta,\sigma\}$ for $\sigma>0$, and $\gamma, \kappa$ are arbitrary constants with $0\leq\gamma<\frac{1}{4},0<\kappa\ll 1$ and $\gamma+\kappa<\frac{1}{4}$.
\end{theorem}
With Theorem \ref{maintheorem} at hands, we can obtain the following corollary.
\begin{corollary}\label{convergencetheorem}
With the assumptions stated in Theorem \ref{maintheorem}, there exists a solution $(U,V,H,G)$ to the viscous MHD equations \eqref{MHDflows_system} with \eqref{MHDflowsbc1}-\eqref{MHDflowsbc2} on $\Omega=[1,\infty)\times\mathbb{R}_+$, so that
\begin{align}\label{convergence}
  &\sup_{(X,Y)\in\Omega}\left|(U-1-u_p^0,H-\sigma-h_p^0)\right|X^{\frac{1}{4}}\lesssim\sqrt\eps,\\
  &\sup_{(X,Y)\in\Omega}\left|(V-\sqrt\eps v_p^0-\sqrt\eps v_e^1,G-\sqrt\eps g_p^0-\sqrt\eps g_e^1)\right|X^{\frac{1}{2}}\lesssim\eps,
\end{align}
where $(1,0,\sigma,0)$ is the ideal MHD flow, $(u_p^0,v_p^0,h_p^0,g_p^0)$ and $(v_e^1,g_e^1)$ are profiles constructed as in Section \ref{sec2} and \ref{sec3}.
\end{corollary}

Let us give some comments about our results as follows.

(1) The global existence and uniqueness for the leading order boundary layer profiles will be proved via the von Mises transformation with suitable decomposition instead of the modified energy framework with cancellation mechanism, which is different from \cite{DLJ,DLXmhd,LXYwell}. Indeed, for the unsteady case, Liu, Xie and Yang show the MHD boundary layer system is ill-posed with insulating (no-slip) boundary conditions \cite{LXY_ill}, which implies the necessity of the non-degeneracy of the tangential magnetic field for the stability of the unsteady MHD boundary layer in Sobolev spaces. And our result implies that the degeneracy of tangential magnetic field is allowed in steady case,  which provides an interesting example to say that the non-degeneracy of the tangential magnetic field is not necessary for steady MHD Prandtl theory. This is different from the unsteady case.

 (2) It is well known that in the construction of the approximate solutions, the strong nonlinearity and singularity come from the leading order profiles. It is challenging to establish the global-in-$x$ well-posedness of leading order boundary layer system with decay, which is the key ingredient to obtain the global stability for MHD boundary layer theory.  Inspired by Iyer \cite{Iyerglobal1,Iyerglobal2,Iyerglobal3}, we introduce a suitable decomposition for the solutions. Formally, we decompose the solutions as
$$\mathrm{MHD\ boundary\ layer}\sim \mathrm{Gaussian\ part}+\mathrm{self\ similar\ part}+\mathrm{error}.$$

The main contribution to the decay comes from the self-similar part. Inspired by the view in \cite{GuoIyer}, ``{\it when $x$ gets large (downstream), solutions to the Prandtl equation, converge to an appropriately renormalized Blasius profile}", we believe that the behavior of self-similar solutions would be important for boundary layer system. However, it seems difficult to search for the self-similar solutions like Blasius profiles used in \cite{IyerMasmoudi} for MHD boundary layer system. Therefore, we first turn to look for the self-similar approximate solutions to the problem, which provide the desired decay in our arguments. See Section \ref{sec2} for details.

(3) Some other topics about the stabilizing effects of magnetic field for the MHD boundary layer theory have attracted great attentions. The great progress for the unsteady MHD flows has been obtained by Liu, Xie and Yang \cite{LXYwell,LXYjus,LXY_ill}: the well-posedness theory and the high Reynolds number limit are established in Sobolev space with non-degenerate tangential magnetic field  \cite{LXYwell,LXYjus} and the necessity for the non-degeneracy of the tangential magnetic field in Sobolev spaces was verified in \cite{LXY_ill}. For the steady MHD equations, the stability for Prandtl-type shear flow was obtained by Liu, Yang and Zhang \cite{LYZ} by some careful spectral analysis for the linearized operator around the Prandtl-type shear flow.

Our goal in this paper is to verify the validity of vanishing viscosity limit for the steady MHD by asymptotic expansion of boundary layer thickness $\sqrt{\varepsilon}$. The setting and the aim of this paper are different from those in \cite{LYZ}. By assuming that tangential velocity is much larger than tangential magnetic field, we prove the global in $x$ stability of the Prandtl expansions for the shear outer ideal MHD flows $(1,0,\sigma,0)$ by multi-scale analysis.

(4) Note that the assumption of sufficiently small mismatch (measured by $\delta$) on the moving boundary will play a key role in our analysis, which is inspired by the series of works \cite{Iyerglobal1,Iyerglobal2,Iyerglobal3}. And in this paper, the no moving boundary case $(u_b=0)$ has not been investigated yet, this interesting and challenging problem will be left to our future work.

(5) In this paper, the boundary condition for magnetic field is the insulating boundary condition (\ref{MHDflowsbc2}) for the technical reasons in dealing with the estimates for the remainders. We will consider the perfectly conducting boundary condition problem in the future.

Let us introduce some notations to end this section.

{\bf Notation.}

The quantities denoted by $\mathcal{O}(\delta,\sigma)$ refer to those which can be made small with the smallness of the constants $\delta,\sigma$.
In addition, we introduce the partial expansions
\begin{align}
\label{u_s}
  &u_s^{(i)}:=\sum_{j=0}^{i-1} \eps^{\frac{j}{2}}(u_p^j+u_e^j)+\eps^{\frac{i}{2}}u_e^i,
  \quad \bar u_s^{(i)}:=\sum_{j=0}^i \eps^{\frac{j}{2}}(u_p^j+u_e^j),\\
\label{v_s}
  &v_s^{(i)}:=\sum_{j=0}^{i-1} \eps^{\frac{j}{2}}v_p^j+\sum_{j=1}^{i}\eps^{\frac{j}{2}-\frac{1}{2}}v_e^j,
  \quad \bar v_s^{(i)}:=\sum_{j=0}^i \eps^{\frac{j}{2}}v_p^j+\sum_{j=1}^i\eps^{\frac{j}{2}-\frac{1}{2}}v_e^j,\\
\label{h_s}
  &h_s^{(i)}:=\sum_{j=0}^{i-1} \eps^{\frac{j}{2}}(h_p^j+h_e^j)+\eps^{\frac{i}{2}}h_e^i,
  \quad \bar u_s^{(i)}:=\sum_{j=0}^i \eps^{\frac{j}{2}}(u_p^j+u_e^j),\\
\label{g_s}
  &g_s^{(i)}:=\sum_{j=0}^{i-1} \eps^{\frac{j}{2}}g_p^j+\sum_{j=1}^{i}\eps^{\frac{j}{2}-\frac{1}{2}}g_e^j,
  \quad \bar g_s^{(i)}:=\sum_{j=0}^i \eps^{\frac{j}{2}}g_p^j+\sum_{j=1}^i\eps^{\frac{j}{2}-\frac{1}{2}}g_e^j,\\
\label{p_s}
  &p_s^{(i)}:=p_p^0+\sum_{j=1}^{i-1} \eps^{\frac{j}{2}}(p_p^j+\eps^{\frac{1}{2}}p_p^{j,a})+\sum_{j=1}^{i}\eps^{\frac{j}{2}}(p_e^j+\eps^\frac{j}{2}p_e^{j,a}),\\
\label{p_s_bar}
  &\bar p_s^{(i)}:=p_p^0+\sum_{j=1}^{i} \eps^{\frac{j}{2}}(p_p^j+\eps^{\frac{1}{2}}p_p^{j,a})+\sum_{j=1}^{i}\eps^{\frac{j}{2}}(p_e^j+\eps^\frac{j}{2}p_e^{j,a}).
\end{align}

\section{Construction of the leading order boundary layer}\label{sec2}
In this section, we will give the detailed construction of the leading order boundary layer. To this end, plugging the approximate solutions \eqref{app} into the expansions \eqref{expansion}, and collecting the $\eps^0$-order terms of $(\bar u_s^{(0)},\bar v_s^{(0)},\bar h_s^{(0)},\bar g_s^{(0)},\bar p_s^{(0)})$, we have
\begin{align*}
  &R_{app}^{u,0}=-\Delta_\eps\bar u_s^{(0)}+\bar u_s^{(0)}\bar u_{sx}^{(0)}+\bar v_s^{(0)}\bar u_{sy}^{(0)}+\p_x p_p^0-\bar h_s^{(0)}\bar h_{sx}^{(0)}- \bar g_s^{(0)}\bar h_{sy}^{(0)},\\
  &R_{app}^{v,0}=-\Delta_\eps\bar v_s^{(0)}+\bar u_s^{(0)}\bar v_{sx}^{(0)}+\bar v_s^{(0)}\bar v_{sy}^{(0)}+\frac{\p_y p_p^0}{\eps}-\bar h_s^{(0)}\bar g_{sx}^{(0)}-\bar g_s^{(0)}\bar g_{sy}^{(0)},\\
  &R_{app}^{h,0}=-\Delta_\eps\bar h_s^{(0)}+\bar u_s^{(0)}\bar h_{sx}^{(0)}+\bar v_s^{(0)}\bar h_{sy}^{(0)}-\bar h_s^{(0)}\bar u_{sx}^{(0)}-\bar g_s^{(0)}\bar u_{sy}^{(0)},\\
  &R_{app}^{g,0}=-\Delta_\eps\bar g_s^{(0)}+\bar u_s^{(0)}\bar g_{sx}^{(0)}+\bar v_s^{(0)}\bar g_{sy}^{(0)}-\bar h_s^{(0)}\bar v_{sx}^{(0)}-\bar g_s^{(0)}\bar v_{sy}^{(0)}.
\end{align*}

Then one can deduce that the leading order boundary layer profiles $(u_p^0,v_p^0,h_p^0,g_p^0)$ satisfy
\begin{equation}\label{zero_order_system}
\left\{
\begin{aligned}
  &[(1+u_p^0)\p_x+(v_p^0+\overline v_e^1)\p_y] u_p^0+\p_xp_p^0\\
  &\qquad -[(\sigma+h_p^0)\p_x+(g_p^0+\overline g_e^1)\p_y] h_p^0=\p_y^2 u_p^0,\\
  &[(1+u_p^0)\p_x+(v_p^0+\overline v_e^1)\p_y] h_p^0\\
  &\qquad -[(\sigma+h_p^0)\p_x+(g_p^0+\overline g_e^1)\p_y] u_p^0
  =\p_y^2 h_p^0,\\
  &[(1+u_p^0)\p_x+(v_p^0+\overline v_e^1)\p_y] (g_p^0+\bar g_e^1)\\
  &\qquad -[(\sigma+h_p^0)\p_x+(g_p^0+\overline g_e^1)\p_y] (v_p^0+\bar v_e^1)
  =\p_y^2 g_p^0,\\
  &p_{py}^0=0,\\
  &\p_x u_p^0+\p_y v_p^0=\p_x h_p^0+\p_y g_p^0=0,\\
  &(v_p^0,g_p^0)(x,y)=\int_y^{\infty}\p_x(u_p^0,h_p^0)(x,\theta){\rm{d}}\theta,
\end{aligned}
\right.
\end{equation}
with the boundary conditions
\begin{equation}\label{zero_order_system_BC}
\left\{
\begin{aligned}
  &(v_p^0,g_p^0)(x,0)=-\int_0^{\infty}\p_x(u_p^0,h_p^0)(x,y){\rm{d}}y=-(\overline v_e^1,\overline g_e^1),\\
  &(u_p^0,h_p^0)(x,0)=(-\delta,-\sigma),\quad (u_p^0,h_p^0)(x,\infty)=(0,0),\\
  &(u_p^0,h_p^0)(1,y)=(u_0^0,h_0^0)(y).
\end{aligned}
\right.
\end{equation}
Note that the terms with boundary value $\overline v_e^1$ are extracted from $v_e^1u_{py}^0$ and $v_e^1h_{py}^0$ via the following computation
\begin{align}\label{overlinev_e^1}
 v_e^1(x,Y)\p_y u_p^0=\bar v_e^1\p_y u_p^0+\sqrt\eps yv_{eY}^1u_{py}^0+\eps u_{py}^0\int_0^y\int_y^{\theta}v_{eYY}^1(\sqrt\eps\tau){\rm d}\tau{\rm d}\theta.
\end{align}
The terms of boundary value $\overline g_e^1$ from $g_e^1u_{py}^0$ and $g_e^1h_{py}^0$ are similar. In this way, some new terms should be put into the remainder terms of leading order as follows
\begin{align}
\label{R^u_0}
  R^{u,0}=&-\eps u_{pxx}^0+\sqrt\eps yv_{eY}^1u_{py}^0+\eps u_{py}^0\int_0^y\int_y^{\theta}v_{eYY}^1(\sqrt\eps\tau){\rm d}\tau{\rm d}\theta\nonumber\\
  &-\sqrt\eps yg_{eY}^1h_{py}^0-\eps h_{py}^0\int_0^y\int_y^{\theta}g_{eYY}^1(\sqrt\eps\tau){\rm d}\tau{\rm d}\theta,\\
\label{R^v_0}
  R^{v,0}=&-\Delta_\eps\bar v_s^{(0)}+\bar u_s^{(0)}\bar v_{sx}^{(0)}+\bar v_s^{(0)}\bar v_{sy}^{(0)}-\bar h_s^{(0)}\bar g_{sx}^{(0)}-\bar g_s^{(0)}\bar g_{sy}^{(0)},\\
\label{R^h_0}
  R^{h,0}=&-\eps h_{pxx}^0+(v_e^1-\bar v_e^1)h_{py}^0-(g_e^1-\bar g_e^1)u_{py}^0\nonumber\\
  =&-\eps h_{pxx}^0+\sqrt\eps yv_{eY}^1h_{py}^0+\eps h_{py}^0\int_0^y\int_y^{\theta}v_{eYY}^1(\sqrt\eps\tau){\rm d}\tau{\rm d}\theta\nonumber\\
  &-\sqrt\eps yg_{eY}^1u_{py}^0-\eps u_{py}^0\int_0^y\int_y^{\theta}g_{eYY}^1(\sqrt\eps\tau){\rm d}\tau{\rm d}\theta,\\
\label{R^g_0}
  R^{g,0}=&-\eps g_{pxx}^0+(v_e^1-\bar v_e^1)g_{py}^0-(g_e^1-\bar g_e^1)v_{py}^0\nonumber\\
  &+(1+u_p^0)\p_x(g_e^1-\bar g_e^1)-(\sigma+h_p^0)\p_x(v_e^1-\bar v_e^1)\nonumber\\
  =&-\eps g_{pxx}^0+\sqrt\eps yv_{eY}^1g_{py}^0+\eps g_{py}^0\int_0^y\int_y^{\theta}v_{eYY}^1(\sqrt\eps\tau){\rm d}\tau{\rm d}\theta\nonumber\\
  &-\sqrt\eps yg_{eY}^1v_{py}^0-\eps v_{py}^0\int_0^y\int_y^{\theta}g_{eYY}^1(\sqrt\eps\tau){\rm d}\tau{\rm d}\theta\nonumber\\
  &+\sqrt\eps u_p^0\int_0^y \p_{xY}g_e^1(x,\sqrt\eps\tau){\rm d}\tau-\sqrt\eps h_p^0\int_0^y \p_{xY}v_e^1(x,\sqrt\eps\tau){\rm d}\tau,
\end{align}
where we have used the equation $g_e^1-\sigma v_e^1=0$ (i.e.,\eqref{v_e^1_relation}, verified in Subsection \ref{sec3.1}) in \eqref{R^g_0}.

It follows from $p_{py}^0=0$ and the Bernoulli's law that $p_{p}^0$ is a constant. In this paper we take $p_p^0=0$ without loss of generality. Utilizing divergence-free conditions, we can rewrite the equation \eqref{zero_order_system}$_2$ for the magnetic field as
\begin{equation*}
  \p_y\left[-(1+u_p^0)(g_p^0+\overline g_e^1)+(v_p^0+\overline v_e^1)(\sigma+h_p^0)\right]=\p^2_y h_p^0,
\end{equation*}
which gives that
\begin{equation}\label{h_p^0_equation}
  -(1+u_p^0)(g_p^0+\overline g_e^1)+(v_p^0+\overline v_e^1)(\sigma+h_p^0)=\p_y h_p^0,
\end{equation}
where we have used \eqref{v_e^1_relation} and the vanishing behavior for $\p_y h_p^0$ at $y=\infty$.

By \eqref{h_p^0_equation}, it is easy to verify that  \eqref{zero_order_system}$_3$ is equivalent to \eqref{zero_order_system}$_2$, therefore it is reasonable for us to investigate the system without the third equation. In other words, we will focus on the nonlinear boundary layer system as follows
\begin{equation}\label{u_p^0_system}
\left\{
\begin{aligned}
  &[(1+u_p^0)\p_x+(v_p^0+\overline v_e^1)\p_y] u_p^0-[(\sigma+h_p^0)\p_x+(g_p^0+\overline g_e^1)\p_y] h_p^0
  =\p_y^2 u_p^0,\\
  &[(1+u_p^0)\p_x+(v_p^0+\overline v_e^1)\p_y] h_p^0-[(\sigma+h_p^0)\p_x+(g_p^0+\overline g_e^1)\p_y] u_p^0
  =\p_y^2 h_p^0,\\
  &(v_p^0,g_p^0)(x,y)=\int_y^{\infty}\p_x(u_p^0,h_p^0)(x,\theta){\rm{d}}\theta,\\
  &(v_p^0,g_p^0)(x,0)=-\int_0^{\infty}\p_x(u_p^0,h_p^0)(x,y){\rm{d}}y=-(\overline v_e^1,\overline g_e^1),\\
  &(u_p^0,h_p^0)(x,0)=(-\delta,-\sigma),\quad (u_p^0,h_p^0)(x,\infty)=(0,0),\\
  &(u_p^0,h_p^0)(1,y)=(u_0^0,h_0^0)(y).
\end{aligned}
\right.
\end{equation}

The main result of this section is stated as follows.
\begin{theorem}\label{u_p^0_theorem}
Suppose that \eqref{u_p^0_initialdecay}-\eqref{u_p^0_initial_u} hold.
 For any $k,j,m\in\mathbb{N}$, $2\leq p\leq \infty$, there exists solution $(u_p^0,v_p^0,h_p^0,g_p^0)$ to the nonlinear boundary layer problem \eqref{u_p^0_system} in the domain $\Omega=[1,\infty)\times \mathbb{R}_+$ with the following estimates
\begin{align}
\label{uh_p^0_estimate1}
  &\Vert z^m\p_x^k\p_y^j (u_p^0,h_p^0)\Vert_{L_y^p}\leq C(m,k,j)x^{-k-\frac{j}{2}+\frac{1}{2p}},{\rm ~for~} 2k+j> 2,\\
\label{uh_p^0_estimate2}
  &\Vert z^m\p_x^k\p_y^j (u_p^0,h_p^0)\Vert_{L_y^p}\leq \mathcal{O}(\delta,\sigma;m,k,j)x^{-k-\frac{j}{2}+\frac{1}{2p}},{\rm ~for~} 2k+j\leq 2,\\
\label{vg_p^0_estimate1}
  &\Vert z^m\p_x^k\p_y^j (v_p^0,g_p^0)\Vert_{L_y^p}\leq C(m,k,j)x^{-k-\frac{j}{2}-\frac{1}{2}+\frac{1}{2p}},{\rm ~for~} 2k+j> 1,\\
\label{vg_p^0_estimate2}
  &\Vert z^m\p_x^k\p_y^j (v_p^0,g_p^0)\Vert_{L_y^p}\leq \mathcal{O}(\delta,\sigma;m,k,j)x^{-k-\frac{j}{2}-\frac{1}{2}+\frac{1}{2p}},{\rm ~for~} 2k+j\leq 1.
\end{align}
\end{theorem}
\begin{remark}\label{rk-leading}
We first remark here that Theorem \ref{u_p^0_theorem} still holds for the perfectly conducting boundary condition as $(u_p^0,\p_y h_p^0)=(-\delta,0)$.

In addition, compared with \cite{DLJ,DLXmhd}, Theorem \ref{u_p^0_theorem} allows the degeneracy of tangential magnetic field, i.e., $\sigma+h_0^0(y)\geq 0.$
\end{remark}

To perform our arguments of proof, let us introduce the von Mises coordinates as follows
\begin{equation*}
  \eta=\int_0^y (1+u_p^0(x,\theta)){\rm{d}}\theta,
\end{equation*}
then the problem \eqref{u_p^0_system} is reduced to the following quasilinear parabolic system
\begin{equation}\label{u_p^0_system_eta}
\begin{cases}
  (1+u_p^0)\p_x u_p^0-(\sigma+h_p^0)\p_x h_p^0-(1+u_p^0)^2\p_\eta^2 u_p^0\\
  \qquad=(1+u_p^0)(|\p_\eta u_p^0|^2-|\p_\eta h_p^0|^2),\\
  (1+u_p^0)\p_x h_p^0-(\sigma+h_p^0)\p_x u_p^0-(1+u_p^0)^2\p_\eta^2 h_p^0=0,\\
  u_p^0(x,0)=-\delta,\quad u_p^0(x,\infty)=0,\\
  h_p^0(x,0)=-\sigma,\quad h_p^0(x,\infty)=0,\\
  (u_p^0,h_p^0)(1,\eta)=(u_0^0,h_0^0)(\eta),
\end{cases}
\end{equation}
in which the equation \eqref{h_p^0_equation} has been used.

To ensure the invertibility of the von Mises transformation, we first assume the {\it a priori assumption} that $1+u_p^0\geq c_0>0$, which will be verified in the following way
\begin{equation}\label{1+u_p^0}
  1+u_p^0(x,\eta)=1-\int_\eta^\infty \p_\eta u_p^0 {\rm{d}}\eta \geq 1-\mathcal{O}(\delta,\sigma):= c_0.
\end{equation}
The detailed proof will be given later, here we sketch the key points.
In the last inequality, thanks to \eqref{u_p^0_estimate2}
 (stated latter), the following estimate has been used
\begin{equation}\label{1+u_p^0_pre}
\begin{aligned}
  \int_\eta^\infty \p_\eta u_p^0(x,\eta){\rm{d}}\eta
  &=\int_\eta^\infty \big[(z+1)\p_\eta u_p^0(x,\eta)\cdot (z+1)^{-1}\big]{\rm{d}}\eta\\
  &\leq \Vert (z+1)\p_\eta u_p^0 \Vert_{L_\eta^2(0,\infty)}\cdot \bigg(\int_\eta^\infty \frac{1}{|z+1|^2}{\rm{d}}\eta\bigg)^{\frac{1}{2}}\\
  &\leq \mathcal{O}(\delta,\sigma)x^{-\frac{1}{4}}\cdot \bigg(\int_0^\infty \frac{1}{|z+1|^2}{\rm{d}}z\cdot \sqrt x\bigg)^{\frac{1}{2}}
  \leq \mathcal{O}(\delta,\sigma)
  \end{aligned}
  \end{equation}
with the variable $z=\frac{\eta}{\sqrt x}$.

In a similar manner to \eqref{1+u_p^0}, we have
\begin{align}\label{1+u_p^0-sigma-h_p^0}
  (1+u_p^0)-(\sigma+h_p^0)\geq 1-\sigma-\mathcal{O}(\delta,\sigma):=c_1>0.
\end{align}

The following Proposition \ref{prop_pre_u_p^0} gives the estimates for the solutions $(u_p^0,h_p^0)$ to problem \eqref{u_p^0_system_eta}, which is important to prove Theorem \ref{u_p^0_theorem}.
\begin{proposition}\label{prop_pre_u_p^0}
Under the assumptions of Theorem \ref{u_p^0_theorem}, suppose that $(u_p^0,h_p^0)$ solve problem \eqref{u_p^0_system_eta}, then for any $k,j,m\in\mathbb{N}$, $2\leq p\leq \infty$, it holds that
\begin{align}
\label{u_p^0_estimate1}
  &\Vert z^m\p_x^k\p_\eta^j (u_p^0,h_p^0)\Vert_{L_\eta^p}\leq C(m,k,j)x^{-k-\frac{j}{2}+\frac{1}{2p}},{\rm ~for~} 2k+j> 2,\\
\label{u_p^0_estimate2}
  &\Vert z^m\p_x^k\p_\eta^j (u_p^0,h_p^0)\Vert_{L_\eta^p}\leq \mathcal{O}(\delta,\sigma;m,k,j)x^{-k-\frac{j}{2}+\frac{1}{2p}},{\rm ~for~} 2k+j\leq 2.
\end{align}
\end{proposition}

Let us solve the quasilinear parabolic problem \eqref{u_p^0_system_eta}, and come back to our original problem \eqref{u_p^0_system} for the proof of Theorem \ref{u_p^0_theorem} at the final part of this section.

One can introduce the shifted unknown for the tangential components
\begin{equation}\label{q_definition}
  \hat{u_p^0}=u_p^0+\delta,\quad \hat{h_p^0}=h_p^0+\sigma
\end{equation}
to shift the moving boundary on $\{\eta=0\}$. With this, the boundary value for the shifted unknown becomes
\begin{equation}\label{q_BC}
  (\hat{u_p^0},\hat{h_p^0})(x,0)=(0,0), \quad
  (\hat{u_p^0},\hat{h_p^0})(x,\infty)=(\delta,\sigma).
\end{equation}

To fix the mismatch on boundaries $\{\eta=0\}$ and $\{\eta=\infty\}$, inspired by the idea in \cite{RGmethod,Iyerglobal1}, we seek for the self-similar solution $(\phi_*,\psi_*)(z)=(\phi_*,\psi_*)(\frac{\eta}{\sqrt x})$ to the following ordinary differential system
\begin{equation}\label{u_p^0_system_selfsimilar}
\begin{cases}
  (1-\delta+\phi_*)^2 \phi''_*+\frac{z}{2}(1-\delta+\phi_*) \phi'_*-\frac{z}{2}\psi_* \psi'_*\\
  \qquad =(1-\delta+\phi_*)(|\psi_*'|^2-|\phi_*'|^2),\\
  (1-\delta+\phi_*)^2 \psi_*''+\frac{z}{2}(1-\delta+\phi_*) \psi_*'-\frac{z}{2}\psi_* \phi'_*=0,\\
  \phi_*(0)=0,\quad \phi_*(\infty)=\delta,\\
  \psi_*(0)=0,\quad \psi_*(\infty)=\sigma.
\end{cases}
\end{equation}

To solve the above problem, let us decompose the leading order correctors $(u_p^0,h_p^0)$ into two parts: the self-similar profile $(\phi_*,\psi_*)(z)$ and the error profile $(w,\Omega)(x,\eta)$, i.e.,
\begin{equation}\label{u_p^0_decomposition}
\begin{cases}
  u_p^0(x,\eta)+\delta=\phi_*(z)+w(x,\eta),\\
  h_p^0(x,\eta)+\sigma=\psi_*(z)+\Omega(x,\eta).
\end{cases}
\end{equation}
Moreover, let us define
\begin{equation}\label{phi_def}
  \phi(z)=\phi_*(z)-e_\delta(z),\quad \psi(z)=\psi_*(z)-e_\sigma(z).
\end{equation}
In \eqref{phi_def}, the Gaussian front solutions $(e_\delta,e_\sigma)$ are introduced to match the value at infinity with an explicit expression (see also \cite{Iyerglobal1})
\begin{equation}\label{e_delta_def}
  e_\delta(z)=\frac{\delta}{\sqrt\pi}\int_0^z e^{-\frac{t^2}{4}}{\rm{d}}t,\quad
  e_\sigma(z)=\frac{\sigma}{\sqrt\pi}\int_0^z e^{-\frac{t^2}{4}}{\rm{d}}t,
\end{equation}
which solve the following ODEs
\begin{equation}\label{e_delta_equation}
\begin{cases}
  e''_\delta+\frac{z}{2}e'_\delta=0,\\
  e_\delta(0)=0,\quad e_\delta(\infty)=\delta,
\end{cases}
\end{equation}
and
\begin{equation}\label{e_sigma_equation}
\begin{cases}
  e''_\sigma+\frac{z}{2}e'_\sigma=0,\\
  e_\sigma(0)=0,\quad e_\sigma(\infty)=\sigma.
\end{cases}
\end{equation}
With $z=\frac{\eta}{\sqrt x}$, we also have $(\p_x-\p_{\eta\eta})e_\delta=0$ and $(\p_x-\p_{\eta\eta})e_\sigma=0$ by chain rule. And one can easily deduce that
\begin{align}
\label{e_sigma_estimate}
  \Vert e_\sigma\Vert_{L_z^\infty}\leq \sigma
  ~{\rm and}~\Vert z^m\p_z^i (e_\delta-\delta,e_\sigma-\sigma)\Vert_{L_z^p}\leq \mathcal{O}(\delta,\sigma)
\end{align}
for any $m,i\in\mathbb{N}$ and $2\leq p\leq\infty$.
\subsection{Estimates for the self-similar profiles}
The problem for the self-similar profiles $(\phi,\psi)(z)$ reads as
\begin{equation}\label{u_p^0_system_phipsi}
\begin{cases}
  (1-\delta+\phi+e_\delta)^2\phi''+\frac{z}{2}(1-\delta+\phi+e_\delta)\phi'\\
  \qquad +(1-\delta+\phi+e_\delta)(-\delta+\phi+e_\delta)e_\delta''-\frac{z}{2}(\psi+e_\sigma)(\psi'+e_\sigma')\\
  \qquad =(1-\delta+\phi+e_\delta)\cdot[(\psi'+e_\sigma)^2-(\phi'+e'_\delta)^2],\\
  (1-\delta+\phi+e_\delta)^2\psi''+\frac{z}{2}(1-\delta+\phi+e_\delta)\psi'-\frac{z}{2}(e_\sigma+\psi)(\phi'+e_\delta')\\
  \qquad +(1-\delta+\phi+e_\delta)(-\delta+\phi+e_\delta)e_\sigma''=0,\\
  \phi(0)=\phi(\infty)=0,\quad \psi(0)=\psi(\infty)=0.
\end{cases}
\end{equation}

\begin{lemma}\label{u_p^0_lemma_phipsi}
For sufficiently small constants $\delta,\sigma$, there exists a unique solution $(\phi,\psi)$ to the problem \eqref{u_p^0_system_phipsi} and there holds
\begin{equation}\label{u_p^0_estimate_phiH2}
  \Vert (\phi,\psi)\Vert_{H_w^2}\lesssim R(\delta,\sigma),
\end{equation}
where the $H_w^2$-norm is defined by
\begin{equation}\label{u_p^0_phiH2_norm}
  \Vert f\Vert_{H_w^2}^2:=\Vert f''\Vert_{L^2_z}^2+\Vert (1+z) f'\Vert_{L^2_z}^2+\Vert (1+z) f\Vert_{L^2_z}^2.
\end{equation}
Moreover, for any $m,k\in\mathbb{N}$, it holds that
\begin{equation}\label{u_p^0_estimate_phiHk}
  \Vert z^m(\phi,\psi)\Vert_{H_z^k}\lesssim R(\delta,\sigma;m,k).
\end{equation}
The quantity $R(\delta,\sigma;m,k)$ is defined by $R(\delta,\sigma;m,k)=\max \{\frac{\sigma}{2}, \Vert z^m\p_z^k(e_\delta-\delta)\Vert_{L_z^2}\}.$
\end{lemma}
\begin{proof}
To perform our analysis, let us assume that $z\in[0,r]$ for any fixed constant $r>0$ and let $r\rightarrow\infty$ finally. First, let us consider the following linearized problem of \eqref{u_p^0_system_phipsi}
\begin{equation}\label{u_p^0_system_self_linear1}
\begin{cases}
  -\phi''-\frac{z}{2}\phi'+\frac{z}{2}e_\sigma\psi'=f_1,\\
  -\psi''-\frac{z}{2}\psi'+\frac{z}{2}e_\sigma\phi'=f_2,\\
  \phi(0)=\phi(r)=0,\quad \psi(0)=\psi(r)=0.
\end{cases}
\end{equation}

To study \eqref{u_p^0_system_self_linear1}, let us denote
\begin{align*}
  \Phi=(\phi,\psi)\in H_0^1(0,r)\times H_0^1(0,r)\triangleq X,~ {\rm{and}}~
  F=(f_1,f_2)\in H^{-1}\times H^{-1}\triangleq X^{-1}.
\end{align*}
With this, define the bilinear form $B[\cdot,\cdot]:X\times X\rightarrow \mathbb{R}$ as follows
\begin{align*}
  B[\Phi,\Phi_1]
  :=\int_0^r \left[\phi'\phi'_1-\frac{z}{2}\phi'\phi_1+\frac{z}{2}e_\sigma\psi'\phi_1\right]+\int_0^r \left[\psi'\psi'_1-\frac{z}{2}\psi'\psi_1+\frac{z}{2}e_\sigma\phi'\psi_1\right]~
\end{align*}
for any $\Phi_1=(\phi_1,\psi_1)\in X$.

It is easy to verify that the bilinear form is bounded and coercive on $X$. Indeed, we have
\begin{equation*}
  \big| B[\Phi,\Phi_1] \big|\leq C(\sigma,r)\Vert \Phi\Vert_{X}\cdot\Vert \Phi_1\Vert_{X}
\end{equation*}
and
\begin{align*}
  \big| B[\Phi,\Phi] \big|
  =&\int_0^r|\Phi'|^2+\frac{1}{4}\int_0^r|\Phi|^2-\frac{1}{2}\int_0^r\psi\phi\cdot\p_z(ze_\sigma)\\
  \geq& \int_0^r|\Phi'|^2+\frac{1}{4}\int_0^r|\Phi|^2-\frac{1}{2}(\Vert e_\sigma\Vert_{L^\infty}
  +\Vert ze_\sigma'\Vert_{L^\infty})\Vert \psi\Vert_{L^2}\Vert \phi\Vert_{L^2}\\
  \geq& C(\sigma)\Vert \Phi\Vert_{X}^2
\end{align*}
for $\sigma\ll 1$. Then for any $F\in X^{-1}$, applying the Lax-Milgram theorem, one can establish the existence and uniqueness of the $H^1$ weak solution $(\psi,\phi)$ to problem \eqref{u_p^0_system_self_linear1}. The elliptic regularity theory implies $H^2$ regularity of $\Phi$ for any $F\in L^2(0,r)$.

In next step, we are going to use the contraction mapping theorem by investigating the linearized problem from the nonlinear problem \eqref{u_p^0_system_phipsi} as below:
\begin{equation}\label{u_p^0_system_self_linear2}
\begin{cases}
    -\phi''-\frac{z}{2}\phi'+\frac{z}{2}e_\sigma\psi'
  =\frac{z}{2}(\bar\phi+e_\delta-\delta)\bar\phi'-\frac{z}{2}\bar\psi\bar\psi'-\frac{z}{2}(e_\sigma+\bar\psi)e_\sigma'\\
  \qquad\qquad\qquad\qquad\qquad
    +(\bar\phi+e_\delta-\delta)(2-\delta+\bar\phi+e_\delta)\bar\phi''\\
  \qquad\qquad\qquad\qquad\qquad
  +(\bar\phi+e_\delta-\delta)(1-\delta+\bar\phi+e_\delta)e''_\delta\\
  \qquad\qquad\qquad\qquad\qquad
  -(1-\delta+\bar\phi+e_\delta)\big[(\bar\psi'+e_\sigma')^2-(\bar\phi'+e'_\delta)^2\big],\\
    -\psi''-\frac{z}{2}\psi'+\frac{z}{2}e_\sigma\phi'
  =\frac{z}{2}(\bar\phi+e_\delta-\delta)\bar\psi'-\frac{z}{2}\bar\psi\bar \phi'-\frac{z}{2}(e_\sigma+\bar\psi)e'_\delta\\
  \qquad\qquad\qquad\qquad\qquad
  +(\bar\phi+e_\delta-\delta)(2-\delta+\bar\phi+e_\delta)\bar\psi''\\
  \qquad\qquad\qquad\qquad\qquad
  +(\bar\phi+e_\delta-\delta)(1-\delta+\bar\phi+e_\delta)e_\sigma'',\\
  \phi(0)=\phi(r)=0,\quad \psi(0)=\psi(r)=0.
\end{cases}
\end{equation}

To derive the estimates, we perform the following several steps.

{\bf Step 1. The basic $L^2$ estimates for $(\phi^{(i)},\psi^{(i)})$, $i$=0,1.}

Multiplying $\eqref{u_p^0_system_self_linear2}_1$ by $\phi$ and $\eqref{u_p^0_system_self_linear2}_2$ by $\psi$, integrating them over $[0,r]$ with respect to $z$ variable, for the left hand side, one can obtain that
\begin{align}\label{0BL_err1}
\begin{split}
  &\int_0^r \left[(-\phi''-\frac{z}{2}\phi'+\frac{z}{2}e_\sigma\psi')\cdot\phi+(-\psi''-\frac{z}{2}\psi'+\frac{z}{2}e_\sigma\phi')\cdot\psi\right]{\rm{d}}z\\
  &=\int_0^r |(\phi',\psi')|^2+\frac{1}{4}\int_0^r |(\phi,\psi)|^2-\frac{1}{2}\int_0^r (e_\sigma+ze_\sigma')\psi\phi.
\end{split}
\end{align}

For the terms in the right-hand side, there holds that
\begin{align}\label{0BL_err1.1}
\begin{split}
  I_1&=\int_0^r \left[\frac{z}{2}(\bar\phi+e_\delta-\delta)(\bar\phi',\bar\psi')\cdot (\phi,\psi)\right]{\rm{d}}z\\
  &\leq\Vert (\phi,\psi)\Vert_{L^2}\cdot \Vert (\bar\phi',\bar\psi')\Vert_{L^2}
  \cdot\frac{1}{2}\left[\Vert z\bar\phi\Vert_{L^\infty}+\Vert z(e_\delta-\delta)\Vert_{L^\infty}\right],
\end{split}
\end{align}

\begin{align}\label{0BL_err1.2}
\begin{split}
  I_2=\int_0^r \left[-\frac{z}{2}(\bar\psi \bar\psi'\cdot \phi+\bar\psi \bar\phi'\cdot\psi)\right]{\rm{d}}z
  \leq\Vert (\phi,\psi)\Vert_{L^2}\cdot \Vert (\bar\phi',\bar\psi')\Vert_{L^2}\cdot\frac{1}{2}\Vert z\bar\psi\Vert_{L^\infty},
\end{split}
\end{align}

\begin{align}\label{0BL_err1.3}
\begin{split}
  I_3&=\int_0^r \left[(\bar\phi+e_\delta-\delta)(2-\delta+\bar\phi+e_\delta)(\bar\phi'',\bar\psi'')\cdot (\phi,\psi)\right]{\rm{d}}z\\
  &\leq\Vert (\phi,\psi)\Vert_{L^2}\cdot \Vert (\bar\phi'',\bar\psi'')\Vert_{L^2}
  \cdot\Vert (\bar\phi+e_\delta-\delta)(2-\delta+\bar\phi+e_\delta)\Vert_{L^\infty},
\end{split}
\end{align}

\begin{align}\label{0BL_err1.4}
\begin{split}
  I_4&=\int_0^r -\frac{z}{2}(e_\sigma+\bar\psi)\cdot (e'_\sigma \phi +e'_\delta \psi){\rm{d}}z\\
  &\leq \frac{1}{2}\Vert (\phi,\psi)\Vert_{L^2}\cdot \Vert z(e'_\delta,e'_\sigma)\Vert_{L^2}\cdot
  \left(\Vert e_\sigma\Vert_{L^\infty}+\Vert \bar\psi\Vert_{L^\infty}\right),
\end{split}
\end{align}

\begin{align}\label{0BL_err1.5}
\begin{split}
  I_5&=\int_0^r (\bar\phi+e_\delta-\delta)(1-\delta+\bar\phi+e_\delta)\cdot(e''_\delta \phi+e''_\sigma\psi){\rm{d}}z\\
  &\leq\Vert (\phi,\psi)\Vert_{L^2}\cdot \Vert (e''_\delta,e''_\sigma)\Vert_{L^2}
  \cdot\Vert (\bar\phi+e_\delta-\delta)(1-\delta+\bar\phi+e_\delta)\Vert_{L^\infty},
\end{split}
\end{align}

\begin{align}\label{0BL_err1.6}
\begin{split}
  I_6&=\int_0^r \left\{-(1-\delta+\bar\phi+e_\delta)\left[(\bar\psi'+e'_\sigma)^2-(\bar\phi'+e'_\delta)^2\right]\cdot \phi\right\}{\rm{d}}z\\
  &\leq
  \Vert \phi\Vert_{L^2}\cdot \left[\Vert (\bar\phi',\bar\psi')\Vert_{L^2}\cdot\Vert (\bar\phi',\bar\psi')\Vert_{L^\infty}\cdot\Vert (1-\delta+\bar\phi+e_\delta)\Vert_{L^\infty}\right.\\
  &\quad \left. +(2\Vert (\bar\phi',\bar\psi')\Vert_{L^2}+\Vert (e'_\sigma,e'_\delta)\Vert_{L^2})\Vert e'_\delta\Vert_{L^\infty}\cdot\Vert (1-\delta+\bar\phi+e_\delta)\Vert_{L^\infty}\right].
\end{split}
\end{align}

{\bf Step 2. The higher estimates for $(\phi,\psi)$.}

Testing the system \eqref{u_p^0_system_self_linear2}$_{1,2}$ by $-\phi''$ and $-\psi''$ respectively, we have
\begin{align}\label{0BL_err2}
\begin{split}
  &\int_0^r \left[(\phi''+\frac{z}{2}\phi'-\frac{z}{2}e_\sigma\psi')\cdot \phi''
  +(\psi''+\frac{z}{2}\psi'-\frac{z}{2}e_\sigma\phi')\cdot \psi''\right]{\rm{d}}z\\
  &=\int_0^r |(\phi'',\psi'')|^2+\frac{1}{2}\int_0^r z(\phi'\phi''+\psi'\psi'')-\frac{1}{2}\int_0^r ze_\sigma(\psi'\phi''+\phi'\psi'').
\end{split}
\end{align}

The terms on the right-hand side are handled by similar arguments in estimates \eqref{0BL_err1.1}-\eqref{0BL_err1.6} in Step 1.

{\bf Step 3. The weighted estimates.}

The application of the multipliers $-z\phi'$ and $-z\psi'$ yields that
\begin{align}\label{0BL_err3}
\begin{split}
  &\int_0^r \left[(\phi''+\frac{z}{2}\phi'-\frac{z}{2}e_\sigma\psi')\cdot z\phi'+(\psi''+\frac{z}{2}\psi'-\frac{z}{2}e_\sigma\phi')\cdot z\psi'
\right]{\rm{d}}z\\
  &=\frac{1}{2}\int_0^r |z(\phi',\psi')|^2+\int_0^r z(\phi'\phi''+\psi'\psi'')-\int_0^r z^2e_\sigma\phi'\psi'.
\end{split}
\end{align}

The estimates for the remaining terms can be obtained by following the similar arguments as in \eqref{0BL_err1.1}-\eqref{0BL_err1.6}.

We take the multipliers $z^{2}\phi$ and $z^{2}\psi$ in the final step to achieve that
\begin{align}\label{0BL_err4}
\begin{split}
&  \int_0^r \left[(-\phi''-\frac{z}{2}\phi'+\frac{z}{2}e_\sigma\psi')\cdot (z^{2}\phi)
  +(-\psi''-\frac{z}{2}\psi'+\frac{z}{2}e_\sigma\phi')\cdot (z^{2}\psi)\right]{\rm{d}}z\\
  =&\int_0^r \left( |z(\phi',\psi')|^2+\frac{3}{4} |z(\phi,\psi)|^2+2z(\phi'\phi+\psi'\psi)- \frac{3 e_\sigma+ze_\sigma'}{2}z^{2}\phi\psi\right).
\end{split}
\end{align}

 We treat the terms on the right-hand side as follows
\begin{align}\label{0BL_err4.1}
\begin{split}
  J_1&=\int_0^r \left[\frac{z}{2}(\bar\phi+e_\delta-\delta)(\bar\phi',\bar\psi')\cdot z^{2}(\phi,\psi)\right]{\rm{d}}z\\
  &\leq\Vert z(\phi,\psi)\Vert_{L^2}\cdot \Vert z(\bar\phi',\bar\psi')\Vert_{L^2}
  \cdot\frac{1}{2}\left[\Vert z\bar\phi\Vert_{L^\infty}+\Vert z(e_\delta-\delta)\Vert_{L^\infty}\right],
\end{split}
\end{align}

\begin{align}\label{0BL_err4.2}
\begin{split}
  J_2=&\int_0^r -\frac{z^3}{2}(\bar\psi \bar\psi'\cdot\phi+\bar\psi \bar\phi'\cdot\psi){\rm{d}}z\\
  \leq&\Vert z(\phi,\psi)\Vert_{L^2}\cdot \Vert z(\bar\phi',\bar\psi')\Vert_{L^2}\cdot\frac{1}{2}\Vert z\bar\psi\Vert_{L^\infty},
\end{split}
\end{align}

\begin{align}\label{0BL_err4.3}
\begin{split}
  J_3&=\int_0^r \left[(\bar\phi+e_\delta-\delta)(2-\delta+\bar\phi+e_\delta)(\bar\phi'',\bar\psi'')\cdot z^{2}(\phi,\psi)\right]{\rm{d}}z\\
  &\leq\Vert z(\phi,\psi)\Vert_{L^2}\cdot \Vert (\bar\phi'',\bar\psi'')\Vert_{L^2}
  \cdot(\Vert z\bar\phi\Vert_{L^\infty}\cdot \Vert (2-\delta+\bar\phi+e_\delta)\Vert_{L^\infty}\\
  &\quad +\Vert z(e_\delta-\delta)\Vert_{L^\infty}\cdot \Vert (2-\delta+\bar\phi+e_\delta)\Vert_{L^\infty}),
\end{split}
\end{align}

\begin{align}\label{0BL_err4.4}
\begin{split}
  J_4&=\int_0^r -\frac{z^3}{2}(e_\sigma+\bar\psi)\cdot (e'_\sigma\phi +e'_\delta\psi){\rm{d}}z\\
  &\leq\Vert z(\phi,\psi)\Vert_{L^2}\cdot \Vert z^{2}(e'_\delta,e'_\sigma)\Vert_{L^2}
  \cdot\frac{1}{2}(\Vert e_\sigma\Vert_{L^\infty}+\Vert \bar\psi\Vert_{L^\infty}),
\end{split}
\end{align}

\begin{align}\label{0BL_err4.5}
\begin{split}
  J_5&=\int_0^r z^{2}(\bar\phi+e_\delta-\delta)(1-\delta+\bar\phi+e_\delta)\cdot (e''_\delta \phi+e''_\sigma\psi){\rm{d}}z\\
  &\leq\Vert z(\phi,\psi)\Vert_{L^2}\cdot \Vert z (e''_\delta,e''_\sigma)\Vert_{L^2}
  \cdot\Vert (\bar\phi+e_\delta-\delta)(1-\delta+\bar\phi+e_\delta)\Vert_{L^\infty},
\end{split}
\end{align}

\begin{align}\label{0BL_err4.6}
\begin{split}
  J_6&=\int_0^r \left\{-(1-\delta+\bar\phi+e_\delta)\left[(\bar\psi'+e'_\sigma)^2-(\bar\phi'+e'_\delta)^2\right]\cdot z^{2}\phi\right\}{\rm{d}}z\\
  &\leq
  \Vert z\phi\Vert_{L^2}\cdot \left[\Vert z(\bar\phi',\bar\psi')\Vert_{L^2}\cdot\Vert (\bar\phi',\bar\psi')\Vert_{L^\infty}\cdot\Vert (1-\delta+\bar\phi+e_\delta)\Vert_{L^\infty}\right.\\
  &\quad \left.+(2\Vert z(\bar\phi',\bar\psi')\Vert_{L^2}+\Vert z(e'_\sigma,e'_\delta)\Vert_{L^2})\Vert e'_\delta\Vert_{L^\infty}\cdot\Vert (1-\delta+\bar\phi+e_\delta)\Vert_{L^\infty})\right].
\end{split}
\end{align}

Combining the above steps, choosing a suitable positive constant $K$, we derive
\begin{align}\label{0BL_err5}
\begin{split}
  & \Vert (\phi'',\psi'')\Vert_{L^2}^2+\frac{3}{2}\Vert z(\phi',\psi')\Vert_{L^2}^2\\
  &+K\Vert (\phi',\psi')\Vert_{L^2}^2+\frac{3}{4}\Vert z(\phi,\psi)\Vert_{L^2}^2+\frac{K}{4}\Vert (\phi,\psi)\Vert_{L^2}^2\\
 \leq &\frac{3+\sigma}{2}\Vert z(\phi',\psi')\Vert_{L^2}\Vert (\phi'',\psi'')\Vert_{L^2}+\frac{\sigma}{2}\Vert z(\phi',\psi')\Vert_{L^2}^2\\
 &
  +2\Vert (\phi,\psi)\Vert_{L^2}\Vert z(\phi',\psi')\Vert_{L^2} +\frac{ K}{4}(\sigma+\Vert ze'_\sigma\Vert_{L^\infty})\Vert (\phi,\psi)\Vert_{L^2}^2\\
  &
  +\frac{3\sigma+\Vert ze'_\sigma\Vert_{L^\infty}}{4}\Vert z(\phi,\psi)\Vert_{L^2}^2+\mathcal{J},
\end{split}
\end{align}
where
\begin{equation}\label{0BL_err5.1}
\begin{aligned}
  \mathcal{J}=
  & \left\{\Vert (\phi'',\psi'')\Vert_{L^2}+\Vert z(\phi',\psi')\Vert_{L^2}+\Vert z(\phi,\psi)\Vert_{L^2}+K\Vert (\phi,\psi)\Vert_{L^2}\right\}\\
  & \cdot\left\{\Vert (\bar\phi'',\bar\psi'')\Vert_{L^2}\cdot
  \Vert (1+z)(\bar\phi,e_\delta-\delta)\Vert_{L^\infty}\cdot\Vert (2-\delta+\bar\phi+e_\delta)\Vert_{L^\infty}\right.\\
  &+\frac{1}{2}\Vert (1+z)(\bar\phi',\bar\psi')\Vert_{L^2}\cdot\Vert z(\bar\phi,\bar\psi,e_\delta-\delta)\Vert_{L^\infty}\\
  &+\frac{1}{2}\Vert (1+z)ze'_\delta\Vert_{L^2}\Vert (e_\sigma,\bar\psi)\Vert_{L^\infty}\\
  & \left.+\Vert (1+z) (e'_\delta,e'_\sigma)\Vert_{L^2}\cdot\Vert e'_\delta\Vert_{L^\infty}\cdot\Vert (1-\delta+\bar\phi+e_\delta)\Vert_{L^\infty}\right.\\
  & \left.+\Vert (1+z)(\bar\phi',\bar\psi')\Vert_{L^2}\cdot\Vert (\bar\phi',\bar\psi',e'_\delta)\Vert_{L^\infty}\cdot\Vert (1-\delta+\bar\phi+e_\delta)\Vert_{L^\infty}\right.\\
  & \left.+\Vert (1+z) (e''_\delta,e''_\sigma)\Vert_{L^2}\cdot\Vert (\bar\phi+e_\delta-\delta)(1-\delta+\bar\phi+e_\delta)\Vert_{L^\infty} \right\}.
\end{aligned}
\end{equation}

For any $l\in\mathbb{N}$, we denote the parameter
\begin{equation}\label{R_delta_sigma_def}
  R(\delta,\sigma):={\rm{max}}\{\Vert (e_\sigma,z^l(e_\delta-\delta))\Vert_{L^\infty},\Vert z^l (e'_\delta,e'_\sigma)\Vert_{L^\infty},\Vert z^l (e'_\delta,e'_\sigma,e''_\delta,e''_\sigma)\Vert_{L^2}\}.
\end{equation}

For the right-hand side in \eqref{0BL_err5}, one may use the smallness of $\sigma$ to absorb some other terms to the left-hand side.
In addition, the $L^\infty$ norms of $(\bar\phi,\bar\psi)$ in $\mathcal{J}$ will be controlled by  $\Vert (\bar\phi,\bar\psi)\Vert_{H_w^2}$, for example,
\begin{equation*}
  |z\bar\psi|^2=\int_0^z\p_z|z\bar\psi|^2\leq 2\Vert z\bar\psi\Vert_{L^2}\Vert \bar\psi\Vert_{L^2}+2\Vert z\bar\psi\Vert_{L^2}\Vert z\p_z\bar\psi\Vert_{L^2}
  \leq 4\Vert \bar\psi\Vert_{H_w^2}^2
\end{equation*}
and
\begin{equation*}
  |\bar\psi|^2=\int_z^\infty\p_z|\bar\psi|^2\leq 2\Vert \bar\psi\Vert_{L^2}\Vert \p_z\bar\psi\Vert_{L^2}\leq 2\Vert \bar\psi\Vert_{H_w^2}^2.
\end{equation*}
Similar arguments yield that
\begin{equation*}
  |\bar\psi'|^2\leq 2\Vert \bar\psi\Vert_{H_w^2}^2.
\end{equation*}

Suppose that $\Vert (\bar\phi,\bar\psi)\Vert_{H_w^2}\leq R(\delta,\sigma)$. Applying Cauchy inequality to \eqref{0BL_err5}, we obtain
\begin{align}
  \frac{1}{8}\Vert (\phi,\psi)\Vert_{H_w^2}^2
  \leq (37+K)R(\delta,\sigma)^5+(27+K)R(\delta,\sigma)^3.
\end{align}

Choosing small enough constants $\delta,\sigma$ so that $R(\delta,\sigma)^2[(37+K)R(\delta,\sigma)^2+27+K]<\frac{1}{8}$, we obtain that
\begin{equation}\label{0BL_err_contraction1}
  (\phi,\psi)=T(\bar\phi,\bar\psi)\in B_{R(\delta,\sigma)}\subset H_w^2, ~{\rm{for}}~{\rm{any}}~(\bar\phi,\bar\psi)\in B_{R(\delta,\sigma)}\subset H_w^2,
\end{equation}
which yields that the mapping $T:(\bar\phi,\bar\psi)\mapsto (\phi,\psi)$ maps the ball $B_{R(\delta,\sigma)}=\{\Vert (\bar\phi,\bar\psi)\Vert_{H_w^2}\leq R(\delta,\sigma)\}$ in $H_w^2$ into itself.

It suffices to verify that the mapping $T$ is contracting on $B_{R(\delta,\sigma)}\subset H_w^2$. For any two pairs $(\bar\phi_{1},\bar\psi_{1})$, $(\bar\phi_{2},\bar\psi_{2})$ in $B_{R(\delta,\sigma)}$, the system for $(\phi_{1}-\phi_{2},\psi_{1}-\psi_{2})$ reads as
\begin{equation}\label{u_p^0_system_self_minus_linear2}
\begin{cases}
    -(\phi_{1}''-\phi_{2}'')-\frac{z}{2}(\phi_{1}'-\phi_{2}')+\frac{z}{2}e_\sigma(\psi_{1}'-\psi_{2}')\\
  \quad =\frac{z}{2}(e_\delta-\delta)(\bar\phi_{1}'-\bar\phi_{2}')+\frac{z}{2}[\bar\phi_{1}(\bar\phi_{1}'-\bar\phi_{2}')
  +\bar\phi_{2}'(\bar\phi_{1}-\bar\phi_{2})]\\
  \qquad -\frac{z}{2}e'_\sigma(\bar\psi_{1}-\bar\psi_{2})-\frac{z}{2}[\bar\psi_{1}(\bar\psi_{1}'-\bar\psi_{2}')+\bar\psi_{2}'(\bar\psi_{1}-\bar\psi_{2})]\\
  \qquad +(e_\delta-\delta)(2-\delta+e_\delta)(\bar\phi''_{1}-\bar\phi''_{2})+\bar\phi_{1}''(\bar\phi_{1}+\bar\phi_{2})(\bar\phi_{1}-\bar\phi_{2})\\
  \qquad +|\bar\phi_{2}|^2(\bar\phi_{1}''-\bar\phi_{2}'')+2(1-\delta+e_\delta)[\bar\phi_{1}(\bar\phi_{1}''-\bar\phi_{2}'')
  +\bar\phi_{2}''(\bar\phi_{1}-\bar\phi_{2})]\\
  \qquad +(1-2\delta+2e_\delta)(\bar\phi_{1}-\bar\phi_{2})e''_\delta +e_\delta''(\bar\phi_{1}+\bar\phi_{2})(\bar\phi_{1}-\bar\phi_{2})\\
  \qquad -(1-\delta+e_\delta+\bar\phi_{1})\cdot \left[(\bar\psi_{1}'+\bar\psi_{2}')(\bar\psi_{1}'-\bar\psi_{2}')-(\bar\phi_{1}'+\bar\phi_{2}')(\bar\phi_{1}'-\bar\phi_{2}')\right]\\
  \qquad +2(1-\delta+e_\delta+\bar\phi_{1})\left[(\bar\psi_{1}'-\bar\psi_{2}')e_\sigma'+(\bar\phi_{1}'-\bar\phi_{2}')e_\delta'\right]\\
  \qquad -(\bar\phi_{1}-\bar\phi_{2})\left[|\bar\psi_2'|^2-|\bar\phi_2'|^2+(|e_\sigma'|^2-|e_\delta'|^2)+2(e_\sigma'\bar\psi_2'-e_\delta'\bar\phi_2')\right],\\
      -(\psi_{1}''-\psi_{2}'')-\frac{z}{2}(\psi_{1}'-\psi_{2}')+\frac{z}{2}e_\sigma(\phi_{1}'-\phi_{2}')\\
  \quad =\frac{z}{2}(e_\delta-\delta)(\bar\psi_{1}'-\bar\psi_{2}')+\frac{z}{2}[\bar\phi_{1}(\bar\psi_{1}'-\bar\psi_{2}')
  +\bar\psi_{2}'(\bar\phi_{1}-\bar\phi_{2})]\\
  \qquad -\frac{z}{2}e'_\delta(\bar\psi_{1}-\bar\psi_{2})-\frac{z}{2}[\bar\psi_{1}(\bar\phi_{1}'-\bar\phi_{2}')
  +\bar\phi_{2}'(\bar\psi_{1}-\bar\psi_{2})]\\
  \qquad +(e_\delta-\delta)(2-\delta+e_\delta)(\bar\psi''_{1}-\bar\psi''_{2})
  +\bar\psi_{1}''(\bar\phi_{1}+\bar\phi_{2})(\bar\phi_{1}-\bar\phi_{2})\\
  \qquad +|\bar\phi_{2}|^2(\bar\psi_{1}''-\bar\psi_{2}'')+2(1-\delta+e_\delta)[\bar\phi_{1}(\bar\psi_{1}''-\bar\psi_{2}'')
  +\bar\psi_{2}''(\bar\phi_{1}-\bar\phi_{2})]\\
  \qquad +(1-2\delta+2e_\delta)(\bar\phi_{1}-\bar\phi_{2})e''_\sigma +e_\sigma''(\bar\phi_{1}+\bar\phi_{2})(\bar\phi_{1}-\bar\phi_{2}),\\
  (\phi_{1}-\phi_{2},\psi_{1}-\psi_{2})|_{z=j}=0,\quad j=0,r.
\end{cases}
\end{equation}
Modifying the arguments applied on the problem \eqref{u_p^0_system_self_linear2}, one obtains that
\begin{equation}
  \Vert (\phi_{1}-\phi_{2},\psi_{1}-\psi_{2})\Vert_{H_w^2}^2\leq R(\delta,\sigma)^2 \Vert (\bar\phi_{1}-\bar\phi_{2},\bar\psi_{1}-\bar\psi_{2})\Vert_{H_w^2}^2.
\end{equation}

Then the contraction mapping theorem ensures a pair of unique solution $(\phi,\psi)$ to the nonlinear problem \eqref{u_p^0_system_phipsi} for $z\in [0,r]$. Besides, since the solution belongs to the ball $B_{R(\delta,\sigma)}$, it implies that
\begin{equation}
  \Vert (\phi,\psi)\Vert_{H_w^2(0,r)}\leq R(\delta,\sigma),
\end{equation}
which is independent of the parameter $r$. Thus, let $r\to \infty$, one can achieve the solution to the original nonlinear problem \eqref{u_p^0_system_phipsi}  with estimate \eqref{u_p^0_estimate_phiH2}.

Similarly, the estimate \eqref{u_p^0_estimate_phiH2} with any weight can be also obtained. We shall derive higher-order estimate \eqref{u_p^0_estimate_phiHk} with any weight by taking corresponding derivatives of system \eqref{u_p^0_system_phipsi}. The arguments are similar and we omit the details here.
\end{proof}

As a corollary, the $L_\eta^p$ estimates for the self-similar profile follow. To state the results, let us denote
$  (\phi_*,\psi_*)(z)=(\phi_*,\psi_*)(x,\eta).
$

\begin{proposition}\label{u_p^0_lemma_selfsimilar}
For any $m,k,j\in \mathbb{N}$ and $2\leq p\leq\infty$, the self-similar profiles $(\psi_*,\phi_*)$ satisfy
\begin{equation}\label{u_p^0_estimate_selfsimilar}
  \Vert z^m\p_x^k\p_\eta^j (\phi_*-\delta,\psi_*-\sigma)\Vert_{L_\eta^p}\leq \mathcal{O}(\delta,\sigma;m,k,j)x^{-k-\frac{j}{2}+\frac{1}{2p}}.
\end{equation}
\end{proposition}
\begin{proof}
With the definition \eqref{phi_def} for $\phi,\psi$ and the estimate \eqref{e_sigma_estimate} for $e_\delta,e_\sigma$, it suffices to bound $\psi$ and $\phi$ in $L_\eta^p$-norm. First, for  $p=2$, one gets from \eqref{u_p^0_estimate_phiHk} that
\begin{equation}\label{u_p^0_estimate_selfsimilar_pre1}
\begin{aligned}
  \Vert z^m\p_x^k\p_\eta^j (\phi,\psi)\Vert_{L_\eta^2}
  =&\Vert z^{m+k}\p_z^{k+j}(\phi,\psi)\Vert_{L_z^2}\cdot \frac{x^{-k-\frac{j}{2}+\frac{1}{4}}}{2^k}\\
  \leq &R(\delta,\sigma;m,k,j)x^{-k-\frac{j}{2}+\frac{1}{4}}.
  \end{aligned}
\end{equation}

For $p=\infty$, by the estimate \eqref{u_p^0_estimate_selfsimilar_pre1}, we have
\begin{align*}
  |z^m\p_x^k\p_\eta^j (\phi,\psi)|^2
  &\leq\Vert z^m\p_x^k\p_\eta^j (\phi,\psi)\Vert_{L_\eta^2}\cdot 2m\Vert z^{m-1}\p_x^k\p_\eta^j (\phi,\psi)\Vert_{L_\eta^2}\cdot x^{-\frac{1}{2}}\\
  &\quad +\Vert z^m\p_x^k\p_\eta^j (\phi,\psi)\Vert_{L_\eta^2}\cdot\Vert z^m\p_x^k\p_\eta^{j+1} (\phi,\psi)\Vert_{L_\eta^2}\\
  &\leq R(\delta,\sigma;m,k,j)^2x^{-2k-j},
\end{align*}
which gives that
\begin{equation}\label{u_p^0_estimate_selfsimilar_pre2}
  \Vert z^m\p_x^k\p_\eta^j (\phi,\psi)\Vert_{L_\eta^\infty}\leq R(\delta,\sigma;m,k,j)x^{-k-\frac{j}{2}}.
\end{equation}

Therefore, for $2< p<\infty$, by the interpolation and \eqref{u_p^0_estimate_selfsimilar_pre1}-\eqref{u_p^0_estimate_selfsimilar_pre2}, one has
\begin{align}\label{u_p^0_estimate_selfsimilar_pre3}
\begin{split}
  \Vert z^m\p_x^k\p_\eta^j (\phi,\psi)\Vert_{L_\eta^p}
  &\leq \Vert z^m\p_x^k\p_\eta^j (\phi,\psi)\Vert_{L_\eta^2}^{\frac{2}{p}}\cdot \Vert z^m\p_x^k\p_\eta^j (\phi,\psi)\Vert_{L_\eta^\infty}^{1-\frac{2}{p}}\\
  &\leq \mathcal{O}(\delta,\sigma;m,k,j)x^{-k-\frac{j}{2}+\frac{1}{2p}}.
\end{split}
\end{align}

Similar arguments yield that
\begin{align}\label{u_p^0_estimate_selfsimilar_pre4}
\begin{split}
  \Vert z^m\p_x^k\p_\eta^j (e_\delta-\delta ,e_\sigma-\sigma)\Vert_{L_\eta^p}\leq \mathcal{O}(\delta,\sigma;m,k,j)x^{-k-\frac{j}{2}+\frac{1}{2p}}.
\end{split}
\end{align}

Finally, combining the decomposition \eqref{phi_def}, we finish the proof.
\end{proof}

\subsection{Estimates for the error solutions}
Using the system \eqref{u_p^0_system_eta} for $(u_p^0,h_p^0)$ and \eqref{u_p^0_system_selfsimilar} for $(\phi_*,\psi_*)$, we can derive the following system for the error profile $(w,\Omega)$
\begin{equation}\label{u_p^0_system_error}
\begin{cases}
      (1-\delta)\p_x w-(1-\delta)^2\p_{\eta\eta}w\\
  \quad =[w^2+2(1-\delta)w]\p_\eta^2w-w\p_xw+\Omega\p_x \Omega-\phi_*\p_x w+\psi_*\p_x \Omega\\
  \quad\quad +[w^2+2(1-\delta+\phi_*)w]\p_\eta^2\phi_*-w\p_x\phi_*+\Omega\p_x\psi_*\\
  \quad\quad +[|\phi_*|^2+2(1-\delta+w)\phi_*]\p_\eta^2w+w(|\p_\eta\phi_*|^2-|\p_\eta \psi_*|^2)\\
  \quad\quad +(1-\delta+w+\phi_*)(|\p_\eta w|^2-|\p_\eta \Omega|^2+2\p_\eta w\p_\eta\phi_*-2\p_\eta\Omega\p_\eta\psi_*),\\
      (1-\delta)\p_x \Omega-(1-\delta)^2\p_{\eta\eta}\Omega\\
  \quad =[w^2+2(1-\delta)w]\p_\eta^2\Omega-w\p_x\Omega+\Omega\p_x w\\
  \quad\quad +[w^2+2(1-\delta+\phi_*)w]\p_\eta^2\psi_*-w\p_x\psi_*+\Omega\p_x\phi_*\\
  \quad\quad +[|\phi_*|^2+2(1-\delta+w)\phi_*]\p_\eta^2\Omega-\phi_*\p_x \Omega+\psi_*\p_x w,\\
  w(x,0)=w(x,\infty)=0,\quad w(1,\eta)=w_0(\eta),\\
  \Omega(x,0)=\Omega(x,\infty)=0,\quad \Omega(1,\eta)=\Omega_0(\eta).
\end{cases}
\end{equation}

Due to the assumptions \eqref{u_p^0_initialsmall}, the initial data $(w_0,\Omega_0)(\eta)$ decay rapidly. In addition, the initial data are small in the following sense
\begin{equation}\label{u_p^0_error_intialsmall}
  \Vert \p_\eta^j(w_0,\Omega_0)\langle\eta\rangle^m\Vert_{L_\eta^\infty}\leq \mathcal{O}(\delta,\sigma;m,j),~{\rm for ~ any ~} m\in\mathbb{N},{\rm~and~}j=0,1,2,3.
\end{equation}

To derive the estimate for later use, in this subsection we assume that the small quantity $\mathcal{O}(\delta,\sigma;m,j)$ equals to $R(\delta,\sigma;m,j)^\frac{3}{2}$ when $j=0,1,2,3$. To perform our arguments, let us make precise the norms for controlling the error profile $(w,\Omega)$
\begin{align}\label{u_p^0_error_norm}
  \Vert (w,\Omega)\Vert_{\mathcal{Q}(\sigma_0)}
  :=\Vert (w,\Omega)\Vert_{\mathcal{Q}(\sigma_0,0)}+\Vert (w,\Omega)\Vert_{\mathcal{Q}(\sigma_0,1)},
\end{align}
where
\begin{align}\label{u_p^0_error_norm_def0}
\begin{split}
  \Vert (w,\Omega)\Vert_{\mathcal{Q}(\sigma_0,0)}
  &:=\sup_{x\geq 1}\Vert (w,\Omega)\cdot x^{-\sigma_0}\Vert_{L_\eta^2}+\sup_{x\geq 1}\Vert \p_\eta(w,\Omega)\cdot x^{-\sigma_0+\frac{1}{2}}\Vert_{L_\eta^2}\\
  &\quad +\Vert (w,\Omega)\cdot x^{-\sigma_0-\frac{1}{2}}\Vert_{L_x^2L_\eta^2}+\Vert \p_\eta(w,\Omega)\cdot x^{-\sigma_0}\Vert_{L_x^2L_\eta^2}\\
  &\quad +\Vert \p_{\eta\eta}(w,\Omega)\cdot x^{-\sigma_0+\frac{1}{2}}\Vert_{L_x^2L_\eta^2}+\Vert \p_x(w,\Omega)\cdot x^{-\sigma_0+\frac{1}{2}}\Vert_{L_x^2L_\eta^2}
\end{split}
\end{align}
and
\begin{align}\label{u_p^0_error_norm_def1}
\begin{split}
  \Vert (w,\Omega)\Vert_{\mathcal{Q}(\sigma_0,1)}
:=  &\sup_{x\geq 1}\Vert \p_x(w,\Omega)\cdot x^{-\sigma_0+1}\Vert_{L_\eta^2}+\sup_{x\geq 1}\Vert \p_{x\eta}(w,\Omega)\cdot x^{-\sigma_0+\frac{3}{2}}\Vert_{L_\eta^2}\\
  & +\sup_{x\geq 1}\Vert \p_{\eta\eta}(w,\Omega)\cdot x^{-\sigma_0+1}\Vert_{L_\eta^2}+\Vert \p_{x\eta}(w,\Omega)\cdot x^{-\sigma_0+1}\Vert_{L_x^2L_\eta^2}\\
  & +\Vert \p_{x\eta\eta}(w,\Omega)\cdot x^{-\sigma_0+\frac{3}{2}}\Vert_{L_x^2L_\eta^2}+\Vert \p_{xx}(w,\Omega)\cdot x^{-\sigma_0+\frac{3}{2}}\Vert_{L_x^2L_\eta^2}.
\end{split}
\end{align}

The higher derivatives of the norms for $k> 1$ are defined by
\begin{align}\label{u_p^0_error_norm_def2}
\begin{split}
  \Vert (w,\Omega)\Vert_{\mathcal{Q}(\sigma_0,k)}
  &:=\sup_{x\geq 1}\Vert \p_x^k(w,\Omega)\cdot x^{-\sigma_0+k}\Vert_{L_\eta^2}+\sup_{x\geq 1}\Vert \p_x^k\p_\eta(w,\Omega)\cdot x^{-\sigma_0+k+\frac{1}{2}}\Vert_{L_\eta^2}\\
  &+\Vert \p_x^k\p_\eta(w,\Omega)\cdot x^{-\sigma_0+k}\Vert_{L_x^2L_\eta^2}
  +\Vert \p_x^{k}\p_{\eta}^2(w,\Omega)\cdot x^{-\sigma_0+k+\frac{1}{2}}\Vert_{L_x^2L_\eta^2}.
\end{split}
\end{align}

\begin{remark}\label{0BL_remark_sigma0}
Compared with Iyer \cite{Iyerglobal1}, there is a new term $\Vert (w,\Omega)\cdot x^{-\sigma_0-\frac{1}{2}}\Vert_{L_x^2L_\eta^2}$ in the energy functional \eqref{u_p^0_error_norm_def0}. Due to this dissipative term, our arguments work for both cases with insulating boundary condition and perfectly conducting condition. In particular, for the case with insulating boundary condition, this term can be ignored since the Hardy inequality is valid.
\end{remark}

Since $(w,\Omega)$ decay as $\eta\rightarrow\infty$, one can deduce the following $L_\eta^\infty$ estimate (refer to \cite{Iyerglobal1} for the detailed proof):
\begin{lemma}\label{u_p^0_lemma_Linfty}
For any $0<\sigma_0<\frac{1}{4}$ and $k\geq 0$, it holds that
\begin{align}\label{u_p^0_estimate_Linfty}
\begin{split}
  &\sup_{x\geq 1}\Vert \p_x^k(w,\Omega)x^{\frac{1}{4}-\sigma_0+k}\Vert_{L_\eta^\infty}
  +\sup_{x\geq 1}\Vert \p_x^k(w_\eta,\Omega_\eta)x^{\frac{3}{4}-\sigma_0+k}\Vert_{L_\eta^\infty}\\
  &\quad \lesssim \Vert (w,\Omega)\Vert_{\mathcal{Q}(\sigma_0,k)}+\Vert (w,\Omega)\Vert_{\mathcal{Q}(\sigma_0,k+1)}.
\end{split}
\end{align}
\end{lemma}
As a direct corollary, the following inequality will be frequently used in this subsection
\begin{align}\label{u_p^0_estimate_Linfty_use}
\begin{split}
    \sup_{x\geq 1}\Vert (w,\Omega)\Vert_{L_\eta^\infty}+\sup_{x\geq 1}\Vert \p_x(w,\Omega)x\Vert_{L_\eta^\infty}+\sup_{x\geq 1}\Vert \p_\eta(w,\Omega)x^{\frac{1}{2}}\Vert_{L_\eta^\infty}
    \lesssim \Vert (w,\Omega)\Vert_{\mathcal{Q}(\sigma_0)}.
\end{split}
\end{align}

With the above preparations, one can derive the estimates for $(w,\Omega)$ as follows.
\begin{proposition}\label{u_p^0_lemma_error}
For any fixed $0<\sigma_0<\frac{1}{4}$ and $m,k\in\mathbb{N}$, there exists a unique solution $(w,\Omega)$ for the error system \eqref{u_p^0_system_error} that satisfies
\begin{align}
\label{u_p^0_estimate_error_H0}
  &\Vert (w,\Omega)\Vert_{\mathcal{Q}(\sigma_0)}\leq R(\delta,\sigma;\sigma_0),\\
\label{u_p^0_estimate_error_H0weighted}
  &\Vert z^m(w,\Omega)\Vert_{\mathcal{Q}(\sigma_0)}\leq \mathcal{O}(\delta,\sigma;\sigma_0,m),\\
\label{u_p^0_estimate_error_Hk}
  &\Vert z^m(w,\Omega)\Vert_{\mathcal{Q}(\sigma_0,k)}\leq C(\sigma_0,k,m),{\rm ~for~} k>1,
\end{align}
where the norms $\mathcal{Q}(\sigma_0)$ and $\mathcal{Q}(\sigma_0,k)$ are defined by \eqref{u_p^0_error_norm}-\eqref{u_p^0_error_norm_def2}, respectively.
\end{proposition}

\begin{remark}\label{u_p^0_remark1}
One may also derive estimate for $\sup\limits_{x\geq 1}\Vert z^m\p_x^{k}\p_\eta^j(w,\Omega)x^{-\sigma_0+k+\frac{j}{2}}\Vert_{L_\eta^2}$ for $k,j\in\mathbb{N}$ from Proposition \ref{u_p^0_lemma_error}, which is inspired by $w_x\approx w_{\eta\eta}$ and $\Omega_x\approx\Omega_{\eta\eta}$, due to the system \eqref{u_p^0_system_error} is quasilinear parabolic.
\end{remark}

\begin{proof}
Let us start with the following linearized system to \eqref{u_p^0_system_error}
\begin{equation}\label{u_p^0_system_error_linear}
\begin{cases}
      (1-\delta)\p_x w-(1-\delta)^2\p_{\eta\eta}w\\
  \quad =[\bar w^2+2(1-\delta)\bar w]\p_\eta^2\bar w-\bar w\p_x\bar w+\bar \Omega\p_x\bar \Omega-\phi_*\p_x\bar w+\psi_*\p_x\bar \Omega\\
  \quad\quad +[\bar w^2+2(1-\delta+\phi_*)\bar w]\p_\eta^2\phi_*-\bar w\p_x\phi_*+\bar \Omega\p_x\psi_*\\
  \quad\quad +[|\phi_*|^2+2(1-\delta+\bar w)\phi_*]\p_\eta^2\bar w+\bar w(|\p_\eta\phi_*|^2-|\p_\eta \psi_*|^2)\\
  \quad\quad +(1-\delta+\bar w+\phi_*)(|\p_\eta\bar w|^2-|\p_\eta\bar \Omega|^2+2\p_\eta\bar w\p_\eta\phi_*-2\p_\eta\bar \Omega\p_\eta\psi_*),\\
      (1-\delta)\p_x \Omega-(1-\delta)^2\p_{\eta\eta}\Omega\\
  \quad =[\bar w^2+2(1-\delta)\bar w]\p_\eta^2\bar \Omega-\bar w\p_x\bar \Omega+\bar \Omega\p_x\bar w\\
  \quad\quad +[\bar w^2+2(1-\delta+\phi_*)\bar w]\p_\eta^2\psi_*-\bar w\p_x\psi_*+\bar \Omega\p_x\phi_*\\
  \quad\quad +[|\phi_*|^2+2(1-\delta+\bar w)\phi_*]\p_\eta^2\bar \Omega-\phi_*\p_x\bar \Omega+\psi_*\p_x\bar w,\\
  w(x,0)=w(x,\infty)=0,\quad w(1,\eta)=w_0(\eta),\\
  \Omega(x,0)=\Omega(x,\infty)=0,\quad \Omega(1,\eta)=\Omega_0(\eta).
\end{cases}
\end{equation}

First, multiplying $\eqref{u_p^0_system_error_linear}_1$ by $w x^{-2\sigma_0}$ and $\eqref{u_p^0_system_error_linear}_2$ by $\Omega x^{-2\sigma_0}$, and integrating them over $[0,\infty)$ with respect to $\eta$ variable, we know that the left-hand side reads as
\begin{align}\label{u_p^0_error_1.1}
\begin{split}
  &\int_0^\infty [(1-\delta)\p_x-(1-\delta)^2\p_{\eta\eta}](w,\Omega)\cdot(w,\Omega) x^{-2\sigma_0}{\rm d}\eta\\
  &=\frac{1-\delta}{2}\frac{\rm d}{{\rm d}x}\int_0^\infty |(w,\Omega)|^2 x^{-2\sigma_0}{\rm d}\eta
  +(1-\delta)^2\int_0^\infty |(w_\eta,\Omega_\eta)|^2 x^{-2\sigma_0}{\rm d}\eta\\
  &\quad +(1-\delta)\sigma_0\int_0^\infty |(w,\Omega)|^2 x^{-2\sigma_0-1}{\rm d}\eta.
\end{split}
\end{align}

The terms on the right-hand side can be controlled as follows
\begin{align}\label{u_p^0_error_1.2}
\begin{split}
  K_1=&\int_0^\infty [\bar w^2+2(1-\delta)\bar w]\p_\eta^2 (\bar w,\bar \Omega)\cdot (w,\Omega) x^{-2\sigma_0}{\rm d}\eta\\
  \leq& \Vert (w,\Omega) x^{-\sigma_0-\frac{1}{2}}\Vert_{L_\eta^2} \Vert \p_\eta^2(\bar w,\bar \Omega) x^{-\sigma_0+\frac{1}{2}}\Vert_{L_\eta^2}
   \Vert \bar w\Vert_{L_\eta^\infty}\cdot [\Vert \bar w\Vert_{L_\eta^\infty}+2(1-\delta)],
\end{split}
\end{align}

\begin{align}\label{u_p^0_error_1.3}
\begin{split}
  K_2=&\int_0^\infty [\bar w^2+2(1-\delta+\phi_*)\bar w]\p_\eta^2(\phi_*,\psi_*)\cdot (w,\Omega) x^{-2\sigma_0}{\rm d}\eta\\
  \leq& \Vert (w,\Omega) x^{-\sigma_0-\frac{1}{2}}\Vert_{L_\eta^2} \Vert \p_{\eta}^2(\phi_*,\psi_*)x\Vert_{L_\eta^\infty}
  \Vert \bar w x^{-\sigma_0-\frac{1}{2}}\Vert_{L_\eta^2}\cdot 2[\Vert \bar w,\phi_*\Vert_{L_\eta^\infty}+(1-\delta)],
\end{split}
\end{align}

\begin{align}\label{u_p^0_error_1.4}
\begin{split}
  K_3=&\int_0^\infty [|\phi_*|^2+2(1-\delta+\bar w)\phi_*]\p_\eta^2(\bar w,\bar \Omega)\cdot (w,\Omega) x^{-2\sigma_0}{\rm d}\eta\\
  \leq& \Vert (w,\Omega) x^{-\sigma_0-\frac{1}{2}}\Vert_{L_\eta^2} \Vert \p_\eta^2(\bar w,\bar \Omega) x^{-\sigma_0+\frac{1}{2}}\Vert_{L_\eta^2}
   \Vert \phi_*\Vert_{L_\eta^\infty}\cdot 2[\Vert \bar w,\phi_*\Vert_{L_\eta^\infty}+(1-\delta)],
\end{split}
\end{align}

\begin{align}\label{u_p^0_error_1.5}
\begin{split}
  K_4=&\int_0^\infty \left[(-\bar w\p_x\bar w+\bar \Omega\p_x\bar \Omega)\cdot w+(-\bar w\p_x\bar \Omega+\bar \Omega\p_x\bar w)\cdot\Omega \right] x^{-2\sigma_0}{\rm d}\eta\\
  \leq& \Vert (w,\Omega) x^{-\sigma_0-\frac{1}{2}}\Vert_{L_\eta^2}\cdot \Vert \p_x(\bar w,\bar \Omega)x^{-\sigma_0+\frac{1}{2}}\Vert_{L_\eta^2}\cdot \Vert (\bar w,\bar \Omega)\Vert_{L_\eta^\infty},
\end{split}
\end{align}

\begin{align}\label{u_p^0_error_1.6}
\begin{split}
  K_5=&\int_0^\infty \left[(-\bar w\p_x\phi_*+\bar \Omega\p_x\psi_*)\cdot w+(-\bar w\p_x\psi_*+\bar \Omega\p_x\phi_*)\cdot\Omega\right] x^{-2\sigma_0}{\rm d}\eta\\
  \leq& \Vert (w,\Omega) x^{-\sigma_0-\frac{1}{2}}\Vert_{L_\eta^2}\cdot \Vert (\bar w,\bar\Omega) x^{-\sigma_0-\frac{1}{2}}\Vert_{L_\eta^2}
  \cdot \Vert \p_x(\phi_*,\psi_*)x\Vert_{L_\eta^\infty},
\end{split}
\end{align}

\begin{align}\label{u_p^0_error_1.7}
\begin{split}
  K_6=&\int_0^\infty \left[(-\phi_*\p_x\bar w+\psi_*\p_x\bar \Omega) w+(-\phi_*\p_x\bar \Omega+\psi_*\p_x\bar w)\Omega\right]x^{-2\sigma_0}{\rm d}\eta\\
  \leq& \Vert (w,\Omega) x^{-\sigma_0-\frac{1}{2}}\Vert_{L_\eta^2}\cdot \Vert \p_x(\bar w,\bar\Omega)x^{-\sigma_0+\frac{1}{2}}\Vert_{L_\eta^2}\cdot
  \Vert (\phi_*,\psi_*)\Vert_{L_\eta^\infty},
\end{split}
\end{align}

\begin{align}\label{u_p^0_error_1.8}
\begin{split}
  K_7=&\int_0^\infty (1-\delta+\bar w+\phi_*)(|\p_\eta\bar w|^2-|\p_\eta\bar \Omega|^2)\cdot w x^{-2\sigma_0}{\rm d}\eta\\
  \leq& \Vert w x^{-\sigma_0-\frac{1}{2}}\Vert_{L_\eta^2}\cdot \Vert \p_\eta(\bar w,\bar\Omega)x^{-\sigma_0}\Vert_{L_\eta^2}\cdot
  \Vert \p_\eta(\bar w,\bar\Omega)x^{\frac{1}{2}}\Vert_{L_\eta^\infty}\cdot[\Vert \bar w,\phi_*\Vert_{L_\eta^\infty}+(1-\delta)],
\end{split}
\end{align}

\begin{align}\label{u_p^0_error_1.9}
\begin{split}
  K_8=&\int_0^\infty (1-\delta+\bar w+\phi_*)(2\p_\eta\bar w\p_\eta\phi_*-2\p_\eta\bar \Omega\p_\eta\psi_*)\cdot w x^{-2\sigma_0}{\rm d}\eta\\
  \leq& \Vert w x^{-\sigma_0-\frac{1}{2}}\Vert_{L_\eta^2} \Vert \p_\eta(\bar w,\bar\Omega)x^{-\sigma_0}\Vert_{L_\eta^2}
   \Vert \p_\eta(\psi_*,\phi_*)x^{\frac{1}{2}}\Vert_{L_\eta^\infty}\cdot 2[\Vert \bar w,\phi_*\Vert_{L_\eta^\infty}+(1-\delta)],
\end{split}
\end{align}

\begin{align}\label{u_p^0_error_1.10}
\begin{split}
  K_9=&\int_0^\infty \bar w(|\p_\eta\phi_*|^2-|\p_\eta \psi_*|^2)\cdot w x^{-2\sigma_0}{\rm d}\eta\\
  \leq& \Vert w x^{-\sigma_0-\frac{1}{2}}\Vert_{L_\eta^2}\cdot \Vert \bar wx^{-\sigma_0-\frac{1}{2}}\Vert_{L_\eta^2}\cdot
   \Vert \p_\eta(\psi_*,\phi_*)x^{\frac{1}{2}}\Vert_{L_\eta^\infty}^2.
\end{split}
\end{align}

Putting the above estimates together, using Cauchy inequality and integrating the result  over $[1,x)$, one has
\begin{align}\label{u_p^0_error_1pre}
\begin{split}
  &\sup_{x\geq 1}\Vert (w,\Omega)x^{-\sigma_0}\Vert_{L_\eta^2}^2+\frac{(1-\delta)\sigma_0}{2}\Vert (w,\Omega)x^{-\sigma_0-\frac{1}{2}}\Vert_{L_x^2L_\eta^2}^2
  +\Vert (w_\eta,\Omega_\eta)x^{-\sigma_0}\Vert_{L_x^2L_\eta^2}^2\\
  \leq& \Vert (w_0,\Omega_0)\Vert_{L_\eta^2}^2+\frac{1}{(1-\delta)^2\sigma_0}\Vert (\bar w,\bar\Omega)\Vert_{\mathcal{Q}(\sigma_0)}^2\cdot
   \left[2\Vert (\bar w,\bar\Omega)\Vert_{\mathcal{Q}(\sigma_0)}^4+4\Vert (\bar w,\bar\Omega)\Vert_{\mathcal{Q}(\sigma_0)}^2\right.\\
  &\left. +6R(\delta,\sigma)^4+10R(\delta,\sigma)^2 +7R(\delta,\sigma)^2\Vert (\bar w,\bar\Omega)\Vert_{\mathcal{Q}(\sigma_0)}^2\right],
\end{split}
\end{align}
in which the $L_\eta^\infty$ inequality \eqref{u_p^0_estimate_Linfty_use} for any $0<\sigma_0<\frac{1}{4}$ and the estimate \eqref{u_p^0_estimate_selfsimilar_pre1} for self-similar profile have been used.

Suppose that $\Vert (\bar\Omega,\bar w)\Vert_{\mathcal{Q}(\sigma_0)}\leq R(\delta,\sigma)$. By \eqref{u_p^0_error_intialsmall}, we get
\begin{align}\label{u_p^0_error_1}
\begin{split}
  &\sup_{x\geq 1}\Vert (w,\Omega)x^{-\sigma_0}\Vert_{L_\eta^2}^2+\Vert (w,\Omega)x^{-\sigma_0-\frac{1}{2}}\Vert_{L_x^2L_\eta^2}^2
   +\Vert (w_\eta,\Omega_\eta)x^{-\sigma_0}\Vert_{L_x^2L_\eta^2}^2\\
   &\quad \lesssim R(\delta,\sigma)^3+R(\delta,\sigma)^4+R(\delta,\sigma)^6.
\end{split}
\end{align}

Second, testing the system \eqref{u_p^0_system_error_linear} by $-w_{\eta\eta} x^{-2\sigma_0+1}$ and $-\Omega_{\eta\eta} x^{-2\sigma_0+1}$, we have
\begin{align}\label{u_p^0_error_2.1}
\begin{split}
  &\int_0^\infty [(1-\delta)\p_x-(1-\delta)^2\p_{\eta\eta}](w,\Omega)\cdot(-w_{\eta\eta},-\Omega_{\eta\eta}) x^{-2\sigma_0+1}{\rm d}\eta\\
  &=\frac{1-\delta}{2}\frac{{\rm d}}{{\rm d}x}\int_0^\infty |(w_\eta,\Omega_\eta)|^2 x^{-2\sigma_0+1}
  +(1-\delta)^2\int_0^\infty |(w_{\eta\eta},\Omega_{\eta\eta})|^2 x^{-2\sigma_0+1}\\
  &\quad +\frac{1-\delta}{2}(2\sigma_0-1)\int_0^\infty |(w_\eta,\Omega_\eta)|^2 x^{-2\sigma_0}.
\end{split}
\end{align}

Moreover, from the system \eqref{u_p^0_system_error_linear}, we can bound the term $\Vert (w_x,\Omega_x)x^{-\sigma_0+\frac{1}{2}}\Vert_{L_\eta^2}^2$ as follows
\begin{align}\label{u_p^0_error_2.2}
\begin{split}
 & \Vert (w_x,\Omega_x)x^{-\sigma_0+\frac{1}{2}}\Vert_{L_\eta^2}^2\\
  \leq& (1-\delta)^2\Vert (w_{\eta\eta},\Omega_{\eta\eta}) x^{-\sigma_0+\frac{1}{2}}\Vert_{L_\eta^2}^2
  +\Vert \bar wx^{-\sigma_0-\frac{1}{2}}\Vert_{L_\eta^2}^2\Vert \p_\eta(\phi_*,\psi_*)x^{\frac{1}{2}}\Vert_{L_\eta^\infty}^4
   \\
   &+\Vert \p_{\eta\eta}(\bar w,\bar\Omega) x^{-\sigma_0+\frac{1}{2}}\Vert_{L_\eta^2}^2 \Vert (\bar w,\phi_*)\Vert_{L_\eta^\infty}^4\\
  &+\Vert (\bar w,\bar\Omega) x^{-\sigma_0-\frac{1}{2}}\Vert_{L_\eta^2}^2
  \cdot\Vert (\p_{\eta}^2\phi_*,\p_x\phi_*,\p_x\psi_*)x\Vert_{L_\eta^\infty}^2\\
  &+\Vert (\bar w,\bar\Omega,\phi_*,\psi_*)\Vert_{L_\eta^\infty}^2\cdot\Vert (\bar w_x,\bar\Omega_x)x^{-\sigma_0+\frac{1}{2}}\Vert_{L_\eta^2}^2\\
  &+\Vert \p_\eta(\bar w,\bar\Omega)x^{-\sigma_0}\Vert_{L_\eta^2}^2
  \Vert \p_\eta(\bar w,\bar\Omega,\phi_*,\psi_*)x^{\frac{1}{2}}\Vert_{L_\eta^\infty}^2
 \Vert (\bar w,\phi_*)\Vert_{L_\eta^\infty}^2.
\end{split}
\end{align}

Applying the similar arguments as in the first step, we have
\begin{align}\label{u_p^0_error_2}
\begin{split}
  &\sup_{x\geq 1}\Vert (w_\eta,\Omega_\eta)x^{-\sigma_0+\frac{1}{2}}\Vert_{L_\eta^2}^2
  +\Vert \p_{\eta\eta}(w,\Omega)x^{-\sigma_0+\frac{1}{2}}\Vert_{L_x^2L_\eta^2}^2\\
   &\qquad +\Vert \p_x(w,\Omega)x^{-\sigma_0+\frac{1}{2}}\Vert_{L_x^2L_\eta^2}^2
   \lesssim R(\delta,\sigma)^3+R(\delta,\sigma)^4+R(\delta,\sigma)^6,
\end{split}
\end{align}
where we have used the assumption that $\Vert (\bar w,\bar\Omega)\Vert_{\mathcal{Q}(\sigma_0)}\leq R(\delta,\sigma)$.

Next, we take $x$ derivative to the linear system \eqref{u_p^0_system_error_linear} to get
\begin{equation}\label{u_p^0_system_errorpx}
\begin{cases}
    (1-\delta)\p_{xx} w-(1-\delta)^2\p_{x\eta\eta} w
    =2[(1-\delta+\bar w+\phi_*)\p_x\phi_*+\phi_*\bar w_x]\p_{\eta\eta}\bar w\\
  \quad +[\bar w^2+2(1-\delta)\bar w]\p_{x\eta\eta}\bar w+2(1-\delta+\bar w)\bar w_x\p_{\eta\eta}\bar w-(\bar w+\phi_*)\p_{xx}\bar w+ \bar\Omega\p_{xx}\bar\Omega\\
  \quad +[|\phi_*|^2+2(1-\delta+\bar w)\phi_*]\p_{x\eta\eta}\bar w-\bar w\p_{xx}\phi_*+\bar\Omega\p_{xx}\psi_*+\psi_*\p_{xx}\bar \Omega\\
  \quad +[\bar w^2+2(1-\delta+\phi_*)\bar w]\p_{x\eta\eta}\phi_*+\p_x\phi_*(|\bar w_\eta|^2-|\bar\Omega_\eta|^2+2\bar w_\eta\p_\eta\phi_*-2\bar \Omega_\eta\p_\eta\psi_*)\\
  \quad +2[(1-\delta+\phi_*+\bar w)\bar w_x+\p_x\phi_*\bar w]\p_{\eta\eta}\phi_*-2\bar w_x\p_x\phi_*+2\bar\Omega_x\p_x\psi_*-|\bar\Omega_x|^2\\
  \quad +\bar w_x(|\bar w_\eta|^2-|\bar\Omega_\eta|^2+|\p_\eta\phi_*|^2-|\p_\eta\psi_*|^2-2\bar\Omega_\eta\p_\eta\psi_*+2\bar w_\eta\p_\eta\phi_*)-|\bar w_x|^2\\
  \quad +2(1-\delta+\bar w+\phi_*)(\p_\eta\bar w\p_{x\eta}\bar w-\p_\eta\bar\Omega\p_{x\eta}\bar\Omega-\p_\eta\psi_*\p_{x\eta}\bar\Omega+\p_\eta\psi_*\p_{x\eta} \bar w\\
  \quad -\p_\eta\bar \Omega \p_{x\eta}\psi_*+\p_\eta\bar w \p_{x\eta}\phi_*)+2\bar w(\p_\eta\phi_*\p_{x\eta}\phi_*-\p_\eta\psi_*\p_{x\eta}\psi_*),\\
    (1-\delta)\p_{xx} \Omega-(1-\delta)^2\p_{x\eta\eta}\Omega
    =2[(1-\delta+\bar w+\phi_*)\p_x\phi_*+\phi_* w_x]\p_{\eta\eta}\bar \Omega\\
  \quad +[\bar w^2+2(1-\delta)\bar w]\p_{x\eta\eta}\bar \Omega+2(1-\delta+\bar w)\bar w_x\p_{\eta\eta}\bar\Omega- (\bar w+\phi_*)\p_{xx}\bar\Omega+ \bar\Omega\p_{xx}\bar w\\
  \quad +[|\phi_*|^2+2(1-\delta+\bar w)\phi_*]\p_{x\eta\eta}\bar \Omega-\bar w\p_{xx}\psi_*+\bar \Omega\p_{xx}\phi_*+\psi_*\p_{xx}\bar w\\
  \quad +[\bar w^2+2(1-\delta+\phi_*)\bar w]\p_{x\eta\eta}\psi_*+2[(1-\delta+\phi_*+\bar w)\bar w_x+\p_x\phi_*\bar w]\p_{\eta\eta}\psi_*.
\end{cases}
\end{equation}

Applying the multipliers $w_x x^{-2\sigma_0+2}$ and $\Omega_x x^{-2\sigma_0+2}$, we get
\begin{align}\label{u_p^0_error_3.1}
\begin{split}
  &\int_0^\infty [(1-\delta)\p_x-(1-\delta)^2\p_{\eta\eta}](w_x,\Omega_x)\cdot(w_x,\Omega_x) x^{-2\sigma_0+2}{\rm d}\eta\\
  =&\frac{1-\delta}{2}\frac{{\rm d}}{{\rm d}x}\int_0^\infty |(w_x,\Omega_x)|^2 x^{-2\sigma_0+2}{\rm d}\eta\\
  &
  +(1-\delta)^2\int_0^\infty |(w_{x\eta},\Omega_{x\eta})|^2 x^{-2\sigma_0+2}{\rm d}\eta\\
  & +(1-\delta)(\sigma_0-1)\int_0^\infty |(w_x,\Omega_x)|^2 x^{-2\sigma_0+1}{\rm d}\eta.
\end{split}
\end{align}

One can estimate the terms on the right-hand side as follows:
\begin{align}\label{u_p^0_error_3.2}
\begin{split}
  S_1=&\int_0^\infty [\bar w^2+2(1-\delta)\bar  w+|\phi_*|^2+2(1-\delta+\bar  w)\phi_*]\\
  &\qquad \cdot \left(\p_{x\eta\eta}\bar  w\cdot w_x+\p_{x\eta\eta}\bar \Omega\cdot\Omega_x \right) x^{-2\sigma_0+2}{\rm d}\eta\\
  \leq& \Vert (w_x,\Omega_x)x^{-\sigma_0+\frac{1}{2}}\Vert_{L_\eta^2}\cdot \Vert \p_{x\eta\eta}(w,\Omega)x^{-\sigma_0+\frac{3}{2}}\Vert_{L_\eta^2}\cdot \Vert (\bar  w,\phi_*)\Vert_{L_\eta^\infty}\\
  & \cdot [2\Vert\bar w\Vert_{L_\eta^\infty}+2(1-\delta)+\Vert \phi_*\Vert_{L_\eta^\infty}],
\end{split}
\end{align}

\begin{align}\label{u_p^0_error_3.3}
\begin{split}
  S_2=&\int_0^\infty [\bar w^2+2(1-\delta+\phi_*)\bar w](\p_{x\eta\eta}\phi_*\cdot w_x+\p_{x\eta\eta}\psi_*\cdot\Omega_x) x^{-2\sigma_0+2}{\rm d}\eta\\
  \leq& \Vert (w_x,\Omega_x)x^{-\sigma_0+\frac{1}{2}}\Vert_{L_\eta^2}\cdot \Vert (\p_{x\eta\eta}\phi_*,\p_{x\eta\eta}\psi_*)x^2\Vert_{L_\eta^\infty}
  \cdot \Vert\bar  w x^{-\sigma_0-\frac{1}{2}}\Vert_{L_\eta^2}\\
  & \cdot [\Vert\bar  w\Vert_{L_\eta^\infty}+2(1-\delta)+2\Vert \phi_*\Vert_{L_\eta^\infty}],
\end{split}
\end{align}

\begin{align}\label{u_p^0_error_3.4}
\begin{split}
  S_3=&\int_0^\infty 2\bigg[(1-\delta+\bar w)\bar w_x+(1-\delta+\bar w+\phi_*)\p_x\phi_*\\
  &\qquad\qquad+\phi_*\bar w_x\bigg] \cdot (\p_{\eta\eta}\bar w\cdot w_x+\p_{\eta\eta}\bar \Omega\cdot\Omega_x) x^{-2\sigma_0+2}{\rm d}\eta\\
  \leq& \Vert (w_x,\Omega_x)x^{-\sigma_0+\frac{1}{2}}\Vert_{L_\eta^2}\cdot \Vert (\p_{\eta\eta}\bar w,\p_{\eta\eta}\bar\Omega )x^{-\sigma_0+\frac{1}{2}}\Vert_{L_\eta^2}\cdot \Vert (\bar w_x,\p_x\phi_*)x\Vert_{L_\eta^\infty}\\
  &\quad \cdot 2[\Vert\bar w\Vert_{L_\eta^\infty}+(1-\delta)+\Vert \phi_*\Vert_{L_\eta^\infty}],
\end{split}
\end{align}

\begin{align}\label{u_p^0_error_3.5}
\begin{split}
  S_4=&\int_0^\infty 2[(1-\delta+\phi_*+\bar w)\bar w_x+\p_x\phi_*\bar w] (\p_{\eta\eta}\phi_*w_x+\p_{\eta\eta}\psi\Omega_x) x^{-2\sigma_0+2}{\rm d}\eta\\
  \leq& \Vert (w_x,\Omega_x)x^{-\sigma_0+\frac{1}{2}}\Vert_{L_\eta^2}\cdot 2\Vert \p_{\eta\eta}(\phi_*,\psi_*)x\Vert_{L_\eta^\infty}
  \cdot \left[\Vert \bar w_xx^{-\sigma_0+\frac{1}{2}}\Vert_{L_\eta^2}\right.\\
  &\quad\left. \cdot(\Vert \phi_*\Vert_{L_\eta^\infty}+1-\delta)+\Vert \bar w x^{-\sigma_0-\frac{1}{2}}\Vert_{L_\eta^2}
  \cdot \Vert \p_x\phi_*x\Vert_{L_\eta^\infty}\right],
\end{split}
\end{align}

\begin{align}\label{u_p^0_error_3.6}
\begin{split}
  S_5=&\int_0^\infty \left[-\bar w\p_{xx}\bar w+\bar\Omega\p_{xx}\bar\Omega-\phi_*\p_{xx}\bar w+\psi_*\p_{xx}\bar\Omega\right]\cdot w_x x^{-2\sigma_0+2}{\rm d}\eta\\
  &+\int_0^\infty \left[-\bar w\p_{xx}\bar\Omega+\bar\Omega\p_{xx}\bar w-\phi_*\p_{xx}\bar\Omega+\psi_*\p_{xx}\bar w\right]\cdot\Omega_x x^{-2\sigma_0+2}{\rm d}\eta\\
  \leq& \Vert (w_x,\Omega_x)x^{-\sigma_0+\frac{1}{2}}\Vert_{L_\eta^2}\cdot \Vert \p_{xx}(\bar w,\bar\Omega)x^{-\sigma_0+\frac{3}{2}}\Vert_{L_\eta^2}\cdot\Vert \bar w,\bar\Omega,\phi_*,\psi_*\Vert_{L_\eta^\infty},
\end{split}
\end{align}

\begin{align}\label{u_p^0_error_3.7}
\begin{split}
  S_6=&\int_0^\infty\bigg[(-\bar w\p_{xx}\phi_*+\bar \Omega\p_{xx}\psi_*)w_x\\
  &\qquad\qquad +(-\bar w\p_{xx}\psi_*+\bar \Omega\p_{xx}\phi_*)\Omega_x\bigg]x^{-2\sigma_0+2}{\rm d}\eta\\
  \leq& \Vert (w_x,\Omega_x)x^{-\sigma_0+\frac{1}{2}}\Vert_{L_\eta^2} \cdot\Vert (\bar w,\bar\Omega)x^{-\sigma_0-\frac{1}{2}}\Vert_{L_\eta^2}
  \cdot\Vert \p_{xx}(\phi_*,\psi_*)x^2\Vert_{L_\eta^\infty},
\end{split}
\end{align}

\begin{align}\label{u_p^0_error_3.8}
\begin{split}
  S_7=&\int_0^\infty \left[\bar w_x(|\bar w_\eta|^2-|\bar \Omega_\eta|^2+|\p_\eta\phi_*|^2-|\p_\eta\psi_*|^2-2\bar \Omega_\eta\p_\eta\psi_*+2 \bar w_\eta\p_\eta\phi_*)\right.\\
  &\quad \left.-|\bar w_x|^2-|\bar \Omega_x|^2-2\bar w_x\p_x\phi_*+2\bar \Omega_x\p_x\psi_*\right]\cdot w_x x^{-2\sigma_0+2}{\rm d}\eta\\
  \leq& \Vert w_x x^{-\sigma_0+\frac{1}{2}}\Vert_{L_\eta^2}\cdot \Vert (\bar w_x,\bar\Omega_x)x^{-\sigma_0+\frac{1}{2}}\Vert_{L_\eta^2}\cdot
  \left[\Vert \p_x(\phi_*,\psi_*)x\Vert_{L_\eta^\infty}\right.\\
  &\left.+\Vert (\bar w_x,\bar\Omega_x)x\Vert_{L_\eta^\infty}+\Vert \p_\eta(\phi_*,\psi_*)x^{\frac{1}{2}}\Vert_{L_\eta^\infty}^2+\Vert \p_\eta(\bar w,\bar\Omega)x^{\frac{1}{2}}\Vert_{L_\eta^\infty}^2\right.\\
  &\left.+2\Vert \p_\eta(\phi_*,\psi_*)x^{\frac{1}{2}}\Vert_{L_\eta^\infty}\cdot\Vert \p_\eta(\bar w,\bar\Omega)x^{\frac{1}{2}}\Vert_{L_\eta^\infty}\right],
\end{split}
\end{align}

\begin{align}\label{u_p^0_error_3.9}
\begin{split}
  S_8=&\int_0^\infty \p_x\phi_*(|\bar w_\eta|^2-|\bar \Omega_\eta|^2+2\bar w_\eta\p_\eta\phi_*-2\bar \Omega_\eta\p_\eta\psi_*)\cdot
   w_x x^{-2\sigma_0+2}{\rm d}\eta\\
  \leq& \Vert w_xx^{-\sigma_0+\frac{1}{2}}\Vert_{L_\eta^2} \Vert \p_\eta(\bar w,\bar\Omega)x^{-\sigma_0}\Vert_{L_\eta^2}\\
  &
   \cdot 2\Vert \p_\eta(\bar w,\bar\Omega,\phi_*,\psi_*)x^{\frac{1}{2}}\Vert_{L_\eta^\infty}\Vert\p_x\phi_* x\Vert_{L_\eta^\infty},
\end{split}
\end{align}

\begin{align}\label{u_p^0_error_3.10}
\begin{split}
  S_9=&\int_0^\infty 2(1-\delta+\bar w+\phi_*)(\p_\eta\bar w\p_{x\eta}\bar w-\p_\eta\bar \Omega\p_{x\eta}\bar \Omega-\p_\eta\psi_* \p_{x\eta}\bar \Omega\\
  &\qquad +\p_\eta\psi_* \p_{x\eta}\bar  w)\cdot w_x x^{-2\sigma_0+2}{\rm d}\eta\\
  \leq& \Vert w_xx^{-\sigma_0+\frac{1}{2}}\Vert_{L_\eta^2}\cdot \Vert \p_{x\eta}(\bar w,\bar\Omega)x^{-\sigma_0+1}\Vert_{L_\eta^2}
  \cdot \Vert \p_\eta(\bar w,\bar\Omega,\phi_*,\psi_*)x^{\frac{1}{2}}\Vert_{L_\eta^\infty}\\
  &\quad \cdot 2[(1-\delta)+\Vert\bar  w\Vert_{L_\eta^\infty}+\Vert \phi_*\Vert_{L_\eta^\infty}],
\end{split}
\end{align}

\begin{align}\label{u_p^0_error_3.11}
\begin{split}
  S_{10}=&\int_0^\infty 2(1-\delta+\bar w+\phi_*)(-\p_\eta\bar \Omega \p_{x\eta}\psi_*+\p_\eta\bar w \p_{x\eta}\phi_*)\cdot\bar w_x x^{-2\sigma_0+2}{\rm d}\eta\\
  \leq& \Vert w_xx^{-\sigma_0+\frac{1}{2}}\Vert_{L_\eta^2}\cdot \Vert \p_{\eta}(\bar w,\bar\Omega)x^{-\sigma_0}\Vert_{L_\eta^2}
  \cdot \Vert \p_{x\eta}(\phi_*,\psi_*)x^{\frac{3}{2}}\Vert_{L_\eta^\infty}\\
  &\quad \cdot 2[(1-\delta)+\Vert \bar w\Vert_{L_\eta^\infty}+\Vert \phi_*\Vert_{L_\eta^\infty}],
\end{split}
\end{align}

\begin{align}\label{u_p^0_error_3.12}
\begin{split}
  S_{11}=&\int_0^\infty 2\bar w(\p_\eta\phi_*\p_{x\eta}\phi_*-\p_\eta\psi_*\p_{x\eta}\psi_*)\cdot w_x x^{-2\sigma_0+2}{\rm d}\eta\\
  \leq& \Vert w_xx^{-\sigma_0+\frac{1}{2}}\Vert_{L_\eta^2} \Vert \bar wx^{-\sigma_0-\frac{1}{2}}\Vert_{L_\eta^2}\\
 &  \cdot 2\Vert \p_{x\eta}(\phi_*,\psi_*)x^{\frac{3}{2}}\Vert_{L_\eta^\infty} \Vert \p_\eta(\phi_*,\psi_*)x^{\frac{1}{2}}\Vert_{L_\eta^\infty}.
\end{split}
\end{align}

Combining the above estimates, using the Cauchy inequality and integrating the result on $x\in [1,\infty)$, we have
\begin{align}\label{u_p^0_error_3}
\begin{split}
  &\frac{1-\delta}{2}\sup_{x\geq 1}\Vert (w_x,\Omega_x)x^{-\sigma_0+1}\Vert_{L_\eta^2}^2+(1-\delta)^2
  \Vert \p_{x\eta}(w,\Omega)x^{-\sigma_0+1}\Vert_{L_x^2L_\eta^2}^2\\
  \leq& (1-\delta)\Vert \p_{x}(w,\Omega)x^{-\sigma_0+\frac{1}{2}}\Vert_{L_x^2L_\eta^2}^2
  +\Vert \p_{x}(w_0,\Omega_0)\Vert_{L_\eta^2}^2+\Vert \p_{x}(w,\Omega)x^{-\sigma_0+\frac{1}{2}}\Vert_{L_x^2L_\eta^2}\\
  &\cdot\left[R(\delta,\sigma)+\Vert (\bar w,\bar\Omega)\Vert_{\mathcal{Q}(\sigma_0)}\right]\left[1+R(\delta,\sigma)+\Vert (\bar w,\bar\Omega)\Vert_{\mathcal{Q}(\sigma_0)}\right]\Vert (\bar w,\bar\Omega)\Vert_{\mathcal{Q}(\sigma_0)}\\
  \lesssim& R(\delta,\sigma)^3+R(\delta,\sigma)^4+R(\delta,\sigma)^6,
\end{split}
\end{align}
where \eqref{u_p^0_error_2} and the assumption $\Vert (\bar w,\bar\Omega)\Vert_{\mathcal{Q}(\sigma_0)}\leq R(\delta,\sigma)$ have been used.

Finally, testing the system \eqref{u_p^0_system_errorpx} by functions $-(\p_{x\eta\eta}w,\p_{x\eta\eta}\Omega) x^{-2\sigma_0+3}$, for the left-hand side, we have
\begin{align}\label{u_p^0_error_4.1}
\begin{split}
  &\int_0^\infty [(1-\delta)\p_x-(1-\delta)^2\p_{\eta\eta}](w_x,\Omega_x)\cdot-\p_{x\eta\eta}(w,\Omega) x^{-2\sigma_0+3}{\rm d}\eta\\
  =&\frac{1-\delta}{2}\frac{{\rm d}}{{\rm d}x}\int_0^\infty |\p_{x\eta}(w,\Omega)|^2 x^{-2\sigma_0+3}{\rm d}\eta\\
  &
  +(1-\delta)^2\int_0^\infty |\p_{x\eta\eta}(w,\Omega)|^2 x^{-2\sigma_0+3}{\rm d}\eta\\
  & +\frac{1-\delta}{2}(2\sigma_0-3)\int_0^\infty |(w_{x\eta},\Omega_{x\eta})|^2 x^{-2\sigma_0+2}{\rm d}\eta.
\end{split}
\end{align}

Using the similar arguments as before to bound the right-hand side, together with \eqref{u_p^0_error_4.1}, one has
\begin{align}
  &\frac{1-\delta}{2}\sup_{x\geq 1}\Vert (w_{x\eta},\Omega_{x\eta})x^{-\sigma_0+\frac{3}{2}}\Vert_{L_\eta^2}^2+(1-\delta)^2
  \Vert \p_{x\eta\eta}(w,\Omega)x^{-\sigma_0+\frac{3}{2}}\Vert_{L_x^2L_\eta^2}^2\nonumber\\
  \leq& 2(1-\delta)\cdot\Vert \p_{x\eta}(w,\Omega)x^{-\sigma_0+1}\Vert_{L_x^2L_\eta^2}^2
  +\Vert \p_{x\eta}(w_0,\Omega_0)\Vert_{L_\eta^2}^2+\Vert \p_{x\eta\eta}(w,\Omega)x^{-\sigma_0+\frac{3}{2}}\Vert_{L_x^2L_\eta^2}\nonumber\\
  &\cdot\left[R(\delta,\sigma)+\Vert (\bar w,\bar\Omega)\Vert_{\mathcal{Q}(\sigma_0)}\right]\left[1+R(\delta,\sigma)+\Vert (\bar \Omega, \bar w)\Vert_{\mathcal{Q}(\sigma_0)}\right]\Vert (\bar w,\bar\Omega)\Vert_{\mathcal{Q}(\sigma_0)}\nonumber\\
  \lesssim& R(\delta,\sigma)^3+R(\delta,\sigma)^4+R(\delta,\sigma)^6\label{u_p^0_error_4},
\end{align}
where \eqref{u_p^0_error_3} and the assumption $\Vert (\bar\Omega,\bar w)\Vert_{\mathcal{Q}(\sigma_0)}\leq R(\delta,\sigma)$ have been used.

Putting the estimates \eqref{u_p^0_error_1},\eqref{u_p^0_error_2},\eqref{u_p^0_error_3} and \eqref{u_p^0_error_4} together, as the positive constants $\delta,\sigma$ are chosen to be small enough, we find that the mapping $\tilde T:(\bar w,\bar\Omega)\mapsto (w,\Omega)$ maps the ball $B_{R(\delta,\sigma)}=\{\Vert (\bar w,\bar\Omega)\Vert_{\mathcal{Q}(\sigma_0)}\leq R(\delta,\sigma)\}$ in $\mathcal{Q}(\sigma_0)$ into itself.

It remains to prove that $\tilde T$ is contracting. Let $(w_1,\Omega_1)$ and $(w_2,\Omega_2)$ be any two pairs solving the problem \eqref{u_p^0_system_error_linear} with corresponding $(\bar w_i,\bar\Omega_i)\in B_{R(\delta,\sigma)},i=1,2$, then $(w_1- w_2,\Omega_1-\Omega_2)$ satisfies the following system
\begin{equation}\label{u_p^0_system_error_minus_linear}
\begin{cases}
      (1-\delta)\p_x ( w_1- w_2)-(1-\delta)^2\p_{\eta\eta}( w_1- w_2)\\
  \quad =[\bar w_1^2+2(1-\delta)\bar w_1]\p_\eta^2(\bar w_1-\bar w_2)+(\bar w_1-\bar w_2)(|\p_\eta\phi_*|^2-|\p_\eta \psi_*|^2)\\
  \quad\quad +(\bar w_1-\bar w_2)(|\p_\eta\bar w_1|^2-|\p_\eta\bar \Omega_1|^2+2\p_\eta\bar w_2\p_\eta\phi_*-2\p_\eta\bar \Omega_2\p_\eta\psi_*-\p_x\bar w_2-\p_x\phi_*)\\
  \quad\quad +[(\bar w_1+\bar w_2)+2(1-\delta+\phi_*)](\bar w_1-\bar w_2)\p_\eta^2\phi_*+(\bar\Omega_1-\bar\Omega_2)(\p_x\bar\Omega_2+\p_x\psi_*)\\
  \quad\quad +[|\phi_*|^2+2(1-\delta+\bar w_1)\phi_*]\p_\eta^2(\bar w_1-\bar w_2)+(\p_x\bar\Omega_1-\p_x\bar\Omega_2)(\bar\Omega_1+\psi_*)\\
  \quad\quad +[(\bar w_1+\bar w_2)+2(1-\delta+\phi_*)](\bar w_1-\bar w_2)\p_\eta^2\bar w_2-(\p_x\bar w_1-\p_x\bar w_2)(\bar w_1+\phi_*)\\
  \quad\quad +(1-\delta+\bar w_2+\phi_*)(\p_\eta\bar w_1+\p_\eta\bar w_2+2\p_\eta\phi_*)(\p_\eta\bar w_1-\p_\eta\bar w_2)\\
  \quad\quad +(1-\delta+\bar w_2+\phi_*)(\p_\eta\bar \Omega_1+\p_\eta\bar \Omega_2+2\p_\eta\psi_*)(\p_\eta\bar \Omega_1-\p_\eta\bar \Omega_2),\\
      (1-\delta)\p_x (\Omega_1-\Omega_2)-(1-\delta)^2\p_{\eta\eta}(\Omega_1-\Omega_2)\\
  \quad =[\bar w_1^2+2(1-\delta)\bar w_1]\p_\eta^2(\bar\Omega_1-\bar\Omega_2)-(\bar w_1-\bar w_2)(\p_x\bar\Omega_2+\p_x\psi_*)\\
  \quad\quad +[(\bar w_1+\bar w_2)+2(1-\delta+\phi_*)](\bar w_1-\bar w_2)\p_\eta^2\psi_*+(\bar\Omega_1-\bar\Omega_2)(\p_x\bar w_2+\p_x\phi_*)\\
  \quad\quad +[|\phi_*|^2+2(1-\delta+\bar w_1)\phi_*]\p_\eta^2(\bar\Omega_1-\bar\Omega_2)+(\p_x\bar w_1-\p_x\bar w_2)(\bar\Omega_1+\psi_*)\\
  \quad\quad +[(\bar w_1+\bar w_2)+2(1-\delta+\phi_*)](\bar w_1-\bar w_2)\p_\eta^2\bar\Omega_2-(\p_x\bar\Omega_1-\p_x\bar\Omega_2)(\bar w_1+\phi_*),
\end{cases}
\end{equation}
with the initial data
$( w_1- w_2,\Omega_1-\Omega_2)(1,\eta)=0$
and boundary conditions
$$( w_1- w_2)(x,0)=( w_1- w_2)(x,\infty)=(\Omega_1-\Omega_2)(x,0)=(\Omega_1-\Omega_2)(x,\infty)=0.$$

Following the similar arguments applied in the problem \eqref{u_p^0_system_error_linear}, we have
\begin{equation}
  \Vert (w_1- w_2,\Omega_1-\Omega_2)\Vert_{\mathcal{Q}(\sigma_0)}^2\leq R(\delta,\sigma)^2\Vert (\bar w_1-\bar w_2,\bar\Omega_1-\bar\Omega_2)\Vert_{\mathcal{Q}(\sigma_0)}^2
\end{equation}
for any $(\bar w_i,\bar\Omega_i)\in\mathcal{Q}(\sigma_0)$ satisfying $\Vert (\bar w_i,\bar\Omega_i)\Vert_{\mathcal{Q}(\sigma_0)}\leq R(\delta,\sigma)$,  $i=1,2$.
And thus the contraction mapping theorem admits the unique global solution to the problem \eqref{u_p^0_system_error} satisfying the estimate \eqref{u_p^0_estimate_error_H0}.
One can also repeat the above calculations with any weight to get \eqref{u_p^0_estimate_error_H0weighted} by applying mathematical induction to the terms with weight $z^{m-1}$.

Finally, the estimate \eqref{u_p^0_estimate_error_Hk} follows from taking higher-order derivatives to the nonlinear system $\eqref{u_p^0_system}$ and following the similar arguments as above. Note that for estimates of the higher-order derivative in the forthcoming analysis, the smallness is not necessary, which is the reason why the smallness condition in \eqref{u_p^0_initialsmall} is imposed on $u^0_0,h^0_0$ only for $j<4$.

The proof is completed.
\end{proof}

We turn to prove Proposition \ref{prop_pre_u_p^0}.
\begin{proof}[ Proof of Proposition \ref{prop_pre_u_p^0}]
Due to the decomposition \eqref{u_p^0_decomposition} of $(u_p^0,h_p^0)$, we have
\begin{align*}
  &\Vert z^m\p_x^k\p_\eta^j u_p^0\Vert_{L_\eta^p}=\Vert z^m\p_x^k\p_\eta^j (\phi_*-\delta)\Vert_{L_\eta^p}+\Vert z^m\p_x^k\p_\eta^j w\Vert_{L_\eta^p},\\
  &\Vert z^m\p_x^k\p_\eta^j h_p^0\Vert_{L_\eta^p}=\Vert z^m\p_x^k\p_\eta^j (\psi_*-\sigma)\Vert_{L_\eta^p}+\Vert z^m\p_x^k\p_\eta^j \Omega\Vert_{L_\eta^p}.
\end{align*}

The first part can be controlled via using Proposition \ref{u_p^0_lemma_selfsimilar}. One may prove boundedness of the second part by applying Proposition \ref{u_p^0_lemma_error} with $0<\sigma_0<\frac{1}{4}$, and performing interpolation with similar arguments as in Proposition \ref{u_p^0_lemma_selfsimilar}. Thus, we have
\begin{align*}
  &\Vert z^m\p_x^k\p_\eta^j (u_p^0,h_p^0)\Vert_{L_\eta^p}\leq C(m,k,j)x^{-k-\frac{j}{2}+\frac{1}{2p}},{\rm ~for~} 2k+j> 2,\\
  &\Vert z^m\p_x^k\p_\eta^j (u_p^0,h_p^0)\Vert_{L_\eta^p}\leq \mathcal{O}(\delta,\sigma;m,k,j)x^{-k-\frac{j}{2}+\frac{1}{2p}},{\rm ~for~} 2k+j\leq 2.
\end{align*}

This completes the proof.
\end{proof}

\begin{remark}\label{u_p^0_remark}
Meanwhile, note that the estimate \eqref{u_p^0_estimate2} gives the positive lower bound for $1+u_p^0$ as stated in \eqref{1+u_p^0}.
\end{remark}
Now we are on the position to prove Theorem \ref{u_p^0_theorem}.
\begin{proof}[Proof of Theorem \ref{u_p^0_theorem}]
Let us begin with transferring the estimates \eqref{u_p^0_estimate1}-\eqref{u_p^0_estimate2} to $(x,y)$ variables. First of all, the estimates \eqref{1+u_p^0} and \eqref{u_p^0_estimate1}-\eqref{u_p^0_estimate2} show that $c_0\leq 1+u_p^0\leq C_0$, then we get from $\eta=\int_0^y (1+u_p^0(x,\theta)){\rm d}\theta$ that $c_0y\leq \eta\leq C_0y$. Furthermore, the Jacobians of the coordinates transformation: $\eta'(y)=1+u_p^0$ and $y'(\eta)=\frac{1}{1+u_p^0}$ are not degenerate. Then the estimates \eqref{uh_p^0_estimate1}-\eqref{uh_p^0_estimate2} follow.

The estimates \eqref{vg_p^0_estimate1}-\eqref{vg_p^0_estimate2} for the vertical components can be deduced easily from \eqref{uh_p^0_estimate1}-\eqref{uh_p^0_estimate2} via divergence-free conditions. Precisely, for $p=2$, when $j=0$, one has
\begin{align*}
  \Vert z^m\p_x^k (v_p^0,g_p^0)\Vert_{L_y^2}
  &\leq \Vert z^my\p_x^k\p_y (v_p^0,g_p^0)\Vert_{L_y^2}\\
  &\leq \Vert z^{m+1}\p_x^{k+1} (u_p^0,h_p^0)\Vert_{L_y^2}\cdot \sqrt x \leq C(m,k)x^{-k-\frac{1}{4}}.
\end{align*}
For $j\geq 1$, we have
\begin{align*}
  \Vert z^m\p_x^k\p_y^j (v_p^0,g_p^0)\Vert_{L_y^2}
  \leq \Vert z^my\p_x^{k+1}\p_y^{j-1} (u_p^0,g_p^0)\Vert_{L_y^2}
  \leq C(m,k,j)x^{-k-\frac{j}{2}-\frac{1}{4}}.
\end{align*}

Particularly, in the last two inequalities, for $2k+j\leq 1$, the coefficient $C(m,k,j)$ can be made small by $\delta,\sigma$, which is denoted by $\mathcal{O}(\delta,\sigma;m,k,j)$.
Additionally, the result for the case $2<p\leq\infty$ can be obtained by interpolation, which is similar to that in Proposition \ref{u_p^0_lemma_selfsimilar}, see also Corollary 2.12 in \cite{Iyerglobal1}.
\end{proof}

\section{Construction of the $\sqrt{\eps}$-order correctors}\label{sec3}
\subsection{The $\sqrt{\eps}$-order ideal MHD correctors}\label{sec3.1}
We begin this subsection with investigating the $\sqrt{\eps}$-order ideal MHD correctors which satisfy the following system
\begin{equation}\label{u_e^1_system}
\begin{cases}
  \p_x u_e^1-\sigma\p_x h_e^1+\p_x p_e^1=0,\\
  \p_x v_e^1-\sigma\p_x g_e^1+\p_Y p_e^1=0,\\
  \p_x h_e^1-\sigma\p_x u_e^1=0,\\
  \p_x g_e^1-\sigma\p_x v_e^1=0,\\
 \p_x u_e^1+\p_Y v_e^1=\p_x h_e^1+\p_Y g_e^1=0,\\
  (v_e^1,g_e^1)(x,0)=-(v_p^0,g_p^0)(x,0),\\
  (v_e^1,g_e^1)\rightarrow (0,0)\ {\rm ~as~} Y\rightarrow\infty.
\end{cases}
\end{equation}

It follows from the equations \eqref{u_e^1_system}$_{3,4}$ and divergence-free conditions that
\begin{equation*}
  \nabla_{x,Y}(g_e^1-\sigma v_e^1)=0,
\end{equation*}
and thus, the boundary condition \eqref{u_e^1_system}$_7$ shows that
\begin{equation}\label{v_e^1_relation}
  g_e^1=\sigma v_e^1.
\end{equation}

Specifically, for the case of $\sigma=0$, equations \eqref{u_e^1_system}$_{3,4}$ imply that $g_e^1=b$ for some constants $b$, and since $g_e^1\rightarrow 0$ as $Y\to \infty$, one shall obtain that $g_e^1=0$. Then any function $h_e^1(Y)$ only depends on the variable $Y$ creates a solution to the system by virtue of equation \eqref{u_e^1_system}$_3$. In other words, when $\sigma=0$, it would be an interesting topic to investigate more general solutions $(u_e^1,v_e^1,h_e^1(Y),0)$ to system \eqref{u_e^1_system}. However, in this paper, for convenience, we take $h_e^1(Y)$ to be zero if $\sigma=0$, as we need the decay property in $x$ variable, which will be useful for estimating our remainder terms.

Therefore, based on the above analysis, using the divergence-free conditions again, one has
\begin{equation*}
  u_e^1=\int_x^\infty v_{eY}^1(\theta,Y){\rm d}\theta,\quad h_e^1=\int_x^\infty g_{eY}^1(\theta,Y){\rm d}\theta,
\end{equation*}
then for any fixed constant $0\leq \sigma\ll 1$, we have
\begin{equation}\label{u_e^1_relation}
  h_e^1=\sigma u_e^1.
\end{equation}

On the other hand, the vorticity formulation of the equations \eqref{u_e^1_system}$_{1,2}$ reads as
\begin{equation}\label{u_e^1_relation2}
  -\Delta v_e^1+\sigma\Delta g_e^1=0.
\end{equation}

Hence, we deduce from \eqref{v_e^1_relation} with \eqref{u_e^1_relation2} that
\begin{equation}\label{v_e^1_system}
  -\Delta v_e^1=0,\quad v_e^1(x,0)=-v_p^0(x,0),\quad v_e^1(x,\infty)=0.
\end{equation}

Without loss of generality, taking the pressure to be $p_e^1=-(1-\sigma^2)u_e^1$, and using \eqref{u_e^1_system}$_2$, one can derive that the following Cauchy-Riemann equations for $(u_e^1,v_e^1)$
\begin{equation}\label{u_e^1_CR}
  v_{ex}^1-u_{eY}^1=0,\quad u_{ex}^1+v_{eY}^1=0.
\end{equation}
Cauchy-Riemann equations \eqref{u_e^1_CR} with the harmonic structure \eqref{v_e^1_system} of $v_e^1$ also imply $\Delta u_e^1=0$.

Note that one can deduce the following decay rate for the boundary condition on $\{x=0\}$ via using the estimate \eqref{vg_p^0_estimate2} for $v_p^1$
\begin{align}\label{u_e^1_bc_decay}
  |v_e^1(x,0)|\leq \mathcal{O}(\delta,\sigma)x^{-\frac{1}{2}}.
\end{align}
It ensures us to omit the technical construction and analysis of Lemma 3.4 and Proposition 3.6 in \cite{Iyerglobal1}, and thus gives that:
\begin{proposition}\label{u_e^1_prop}
For $(u_e^1,v_e^1)$ solving the following boundary value problem
\begin{equation}\label{u_e^1vhg_system}
  -\Delta v_e^1=0,\quad v_e^1(x,0)=-v_p^0(x,0),\quad v_e^1(x,\infty)=0,\quad u_e^1=\int_x^\infty v_{eY}^1(\theta,Y){\rm d}\theta,
\end{equation}
and $(h_e^1,g_e^1)$ constructed by $(h_e^1,g_e^1)=\sigma(u_e^1,v_e^1)$, with the pressure $p_e^1=-(1-\sigma^2)u_e^1$. Then $(u_e^1,v_e^1,h_e^1,g_e^1,p_e^1)$ solve system \eqref{u_e^1_system} with
\begin{align}\label{u_e^1_estimate}
  &\Vert (u_e^1,v_e^1,h_e^1,g_e^1)\Vert_{L_Y^\infty}\leq \mathcal{O}(\delta,\sigma)x^{-\frac{1}{2}},\\
  &\Vert (u_{ex}^1,v_{eY}^1,h_{ex}^1,g_{eY}^1)\Vert_{L_Y^\infty}\leq \mathcal{O}(\delta,\sigma)x^{-\frac{3}{2}},\\
  &\Vert \p_x^k(v_e^1,g_e^1)Y\Vert_{L_Y^\infty}\leq C(k,j)x^{-k-\frac{1}{2}},\\
  &\Vert \p_x^k\p_Y^j(u_e^1,v_e^1,h_e^1,g_e^1)\Vert_{L_Y^\infty}\leq C(k,j)x^{-k-j-\frac{1}{2}}.
\end{align}
\end{proposition}

\subsection{The $\sqrt{\eps}$-order boundary layer}\label{sec3.2}
This subsection is devoted to the construction of the $\sqrt{\eps}$-order boundary layer correctors.  Recall the expressions of $R^{u,0},R^{v,0}$ in \eqref{R^u_0}-\eqref{R^v_0}, and rewrite the $\sqrt{\eps}$-order expansions of each equation as follows
\begin{align}
\label{R_app^u,1}
  R_{app}^{u,1}
  =&-\Delta_\eps\bar u_s^{(1)}+\bar u_s^{(1)}\bar u_{sx}^{(1)}+\bar v_s^{(1)}\bar u_{sy}^{(1)}+\p_x\bar p_s^{(1)}-\bar h_s^{(1)}\bar h_{sx}^{(1)}-\bar g_s^{(1)}\bar h_{sy}^{(1)}\nonumber\\
  =&R^{u,0}+\sqrt\eps[-\Delta_\eps u_p^1+\overline u_s^{(1)}u_{px}^1+u_p^1u_{sx}^{(1)}+\overline v_s^{(1)}u_{py}^1+v_p^1u_{sy}^{(1)}-\overline h_s^{(1)}h_{px}^1\nonumber\\
  &-h_p^1h_{sx}^{(1)}-\overline g_s^{(1)}h_{py}^1-g_p^1h_{sy}^{(1)}+p_{px}^1]
  +\eps p_{px}^{1,a}+R_E^{u,1}+(E_E^{u,1}+\eps p_{ex}^{1,a})\nonumber\\
  &-\eps^{\frac{3}{2}}\Delta u_e^1+\sqrt\eps(p_{ex}^1+u_{ex}^1-\sigma h_{ex}^1),\\
\label{R_app^v,1}
R_{app}^{v,1}
  =&-\Delta_\eps\bar v_s^{(1)}+\bar u_s^{(1)}\bar v_{sx}^{(1)}+\bar v_s^{(1)}\bar v_{sy}^{(1)}+\eps^{-1}\p_y\bar p_s^{(1)}-\bar h_s^{(1)}\bar g_{sx}^{(1)}-\bar g_s^{(1)}\bar g_{sy}^{(1)}\nonumber\\
  =&R^{v,0}+\sqrt\eps[-\Delta_\eps v_p^1+\overline u_s^{(1)}v_{px}^1+u_p^1v_{sx}^{(1)}+\overline v_s^{(1)}v_{py}^1+v_p^1v_{sy}^{(1)}-\overline h_s^{(1)}g_{px}^1\nonumber\\
  &-h_p^1g_{sx}^{(1)}-\overline g_s^{(1)}g_{py}^1-g_p^1g_{sy}^{(1)}+\eps^{-1}p_{py}^1]+p_{py}^{1,a}+R_E^{v,1}+(E_E^{v,1}+\p_y p_{e}^{1,a})\nonumber\\
  &-\eps\Delta v_e^1+(\p_Y p_{e}^1+v_{ex}^1-\sigma g_{ex}^1),
\end{align}
in which $R^{u,0},R^{v,0}$ are defined as in \eqref{R^u_0}-\eqref{R^v_0}. Note that the last two terms in $R_{app}^{u,1}$ and $R_{app}^{v,1}$ can be cancelled by utilizing Cauchy-Riemann equations \eqref{u_e^1_CR} of $(u_e^1,v_e^1)$ and equations \eqref{u_e^1_system}$_{1,2}$.
In the above expansions, $R_E^{u,1},R_E^{v,1}$ denote the following coupling Euler-Prandtl terms
\begin{align}\label{R^u,1_E}
  R_E^{u,1}=\sqrt\eps (u_{ex}^1u_p^0+ u_e^1u_{px}^0- h_{ex}^1h_p^0- h_e^1h_{px}^0)+\eps (v_p^0u_{eY}^1- g_p^0h_{eY}^1),
\end{align}
\begin{align}\label{R^v,1_E}
\begin{split}
  R_E^{v,1}=&(u_p^0v_{ex}^1+v_e^1v_{py}^0-h_p^0g_{ex}^1-g_e^1g_{py}^0)\\
  &+\sqrt\eps (u_e^1v_{px}^0+ v_p^0v_{eY}^1- h_e^1g_{px}^0- g_p^0g_{eY}^1).
\end{split}
\end{align}

We introduce an auxiliary boundary layer pressure
\begin{align}\label{p_p^1,a}
  p_p^{1,a}=\int_y^\infty (R^{v,0}+R_E^{v,1})
\end{align}
 in the expansion \eqref{R_app^v,1} of $R_{app}^{v,1}$. To preserve the divergence-free condition, let us introduce a new term $\eps p_{px}^{1,a}$ in the expansion \eqref{R_app^u,1} of $R_{app}^{u,1}$:
\begin{align}\label{p_px^1,a}
\begin{split}
&p_{px}^{1,a}
  =\int_y^\infty\left[
  -\Delta_\eps\bar v_{sx}^{(0)}+\bar u_s^{(0)}v_{sxx}^{(1)}+v_s^{(1)}\bar v_{sxy}^{(0)}-\bar h_s^{(0)} g_{sxx}^{(1)}-g_s^{(1)}\bar g_{sxy}^{(0)}+u_p^0v_{exx}^1\right.\\
  &\left.+v_e^1v_{pxy}^0-h_p^0g_{exx}^1-g_e^1g_{pxy}^0+\sqrt\eps\left( u_e^1v_{pxx}^0+ v_p^0v_{exY}^1- h_e^1g_{pxx}^0- g_p^0g_{exY}^1\right)\right].
\end{split}
\end{align}

The pure-Euler terms $E_E^{u,1},E_E^{v,1}$ will result in some new terms in our analysis, which are expressed as
\begin{align}\label{E_E^u,1}
\begin{split}
  E_E^{u,1}
  =&\eps u_e^1u_{ex}^1+\eps v_e^1u_{eY}^1-\eps h_e^1h_{ex}^1-\eps g_e^1h_{eY}^1\\
  =&\eps (1-\sigma^2)[u_e^1u_{ex}^1+v_e^1v_{ex}^1]
  =\frac{1-\sigma^2}{2}\eps\p_x\left(|u_e^1|^2+|v_e^1|^2\right)
\end{split}
\end{align}
and
\begin{align}\label{E_E^v,1}
\begin{split}
  E_E^{v,1}
  =&\sqrt\eps u_e^1v_{ex}^1+\sqrt\eps v_e^1v_{eY}^1-\sqrt\eps h_e^1g_{ex}^1-\sqrt\eps g_e^1g_{eY}^1\\
  =&\sqrt\eps (1-\sigma^2)[u_e^1u_{eY}^1+v_e^1v_{eY}^1]
  =\frac{1-\sigma^2}{2}\sqrt\eps\p_Y\left(|u_e^1|^2+|v_e^1|^2\right),
\end{split}
\end{align}
where the Cauchy-Riemann equations \eqref{u_e^1_CR} have been used. Note that the terms can be rewritten as gradient-type, one shall introduce an auxiliary ideal MHD pressure to cancel them. To be precise, we take
\begin{align}\label{p_e^1,a}
  p_e^{1,a}=-\frac{1-\sigma^2}{2}\left(|u_e^1|^2+|v_e^1|^2\right)
\end{align}
to ensure that
\begin{align}\label{E_E+p_e^1,a}
  E_E^{u,1}+\eps p_{ex}^{1,a}=0,\quad
  E_E^{v,1}+\p_y p_{e}^{1,a}=0.
\end{align}

Similarly, for the magnetic field, one can also derive
\begin{align*}
  R_{app}^{h,1}
  =&-\Delta_\eps\bar h_s^{(1)}+\bar u_s^{(1)}\bar h_{sx}^{(1)}+\bar v_s^{(1)}\bar h_{sy}^{(1)}-\bar h_s^{(1)}\bar u_{sx}^{(1)}-\bar g_s^{(1)}\bar u_{sy}^{(1)}\\
  =&R^{h,0}+\sqrt\eps[-\Delta_\eps h_p^1+\overline u_s^{(1)}h_{px}^1+u_p^1h_{sx}^{(1)}+\overline v_s^{(1)}h_{py}^1+v_p^1h_{sy}^{(1)}\\
  &-\overline h_s^{(1)}u_{px}^1-h_p^1u_{sx}^{(1)}-\overline g_s^{(1)}u_{py}^1-g_p^1u_{sy}^{(1)}]+R_E^{h,1}+E_E^{h,1},\\
  R_{app}^{g,1}
  =&-\Delta_\eps\bar g_s^{(1)}+\bar u_s^{(1)}\bar g_{sx}^{(1)}+\bar v_s^{(1)}\bar g_{sy}^{(1)}-\bar h_s^{(1)}\bar v_{sx}^{(1)}-\bar g_s^{(1)}\bar v_{sy}^{(1)}\\
  =&R^{g,0}+\sqrt\eps[-\Delta_\eps g_p^1+\overline u_s^{(1)}g_{px}^1+u_p^1g_{sx}^{(1)}+\overline v_s^{(1)}g_{py}^1+v_p^1g_{sy}^{(1)}\\
  &-\overline h_s^{(1)}v_{px}^1-h_p^1v_{sx}^{(1)}-\overline g_s^{(1)}v_{py}^1-g_p^1v_{sy}^{(1)}]+R_E^{g,1}+E_E^{g,1},
\end{align*}
in which $R^{h,0},R^{g,0}$ are defined as in \eqref{R^h_0}-\eqref{R^g_0}. The Euler-Prandtl terms $R_E^{h,1},R_E^{g,1}$ are expressed as
\begin{align}\label{R^h_1_E}
  R_E^{h,1}=\sqrt\eps(u_p^0 h_{ex}^1+u_e^1 h_{px}^0-h_p^0 u_{ex}^1-h_e^1 u_{px}^0)+\eps(v_p^0 h_{eY}^1-g_p^0 u_{eY}^1)
\end{align}
and
\begin{align}\label{R^g_1_E}
  R_E^{g,1}=\sqrt\eps(u_e^1 g_{px}^0+v_p^0 g_{eY}^1-h_e^1 v_{px}^0-g_p^0 v_{eY}^1).
\end{align}

Fortunately, the corresponding pure-Euler terms $E_E^{h,1},E_E^{g,1}$ vanish, i.e.,
\begin{align}\label{E_E^h,1}
\begin{split}
  E_E^{h,1}
  =&\eps u_e^1h_{ex}^1+\eps v_e^1h_{eY}^1-\eps h_e^1u_{ex}^1-\eps g_e^1u_{eY}^1\\
  =&\eps (\sigma-\sigma)[u_e^1u_{ex}^1+v_e^1v_{ex}^1]=0
\end{split}
\end{align}
and
\begin{align}\label{E_E^g,1}
\begin{split}
  E_E^{g,1}
  =&\sqrt\eps u_e^1g_{ex}^1+\sqrt\eps v_e^1g_{eY}^1-\sqrt\eps h_e^1v_{ex}^1-\sqrt\eps g_e^1v_{eY}^1\\
  =&\sqrt\eps (\sigma-\sigma)[u_e^1u_{eY}^1+v_e^1v_{eY}^1]=0.
\end{split}
\end{align}

Therefore, the initial boundary value problem for the $\sqrt{\eps}$-order boundary layer corrector $(u_p^1,v_p^1,h_p^1,g_p^1)$ is reduced to
\begin{equation}\label{u_p^1_system}
\begin{cases}
  -u_{pyy}^1+(1+u_p^0)u_{px}^1+p_{px}^1=(\sigma+h_p^0)h_{px}^1-\mathcal{P}_u^{(1)}+f_u^{(1)},\\
  -h_{pyy}^1+(1+u_p^0)h_{px}^1=(\sigma+h_p^0)u_{px}^1-\mathcal{P}_h^{(1)}+f_h^{(1)},\\
  -g_{pyy}^1+(1+u_p^0)\p_x(g_p^1-\bar g_p^1)=(\sigma+h_p^0)\p_x(v_p^1-\bar v_p^1)-\mathcal{P}_g^{(1)}+f_g^{(1)},\\
  p_{py}^1=0,\\
  u_{px}^1+v_{py}^1=h_{px}^1+g_{py}^1=0,\\
  (u_p^1,v_p^1,h_p^1,g_p^1)(x,0)=(-\overline u_e^1,-\overline v_e^2,-\overline h_e^1,-\overline g_e^2)(x),\\
  (u_p^1,v_p^1,h_p^1,g_p^1)(x,\infty)=(0,0,0,0),\quad (u_p^1,h_p^1)(1,y)=(u_0^1,h_0^1)(y),
\end{cases}
\end{equation}
where the convective terms and the forcing terms are defined by
\begin{align}\label{P_u1}
\begin{cases}
  \mathcal{P}_u^{(1)}=&u_p^1u_{sx}^{(1)}+(v_p^1-\bar v_p^1)u_{py}^0+v_s^{(1)}u_{py}^1\\
  &-h_p^1h_{sx}^{(1)}-(g_p^1-\bar g_p^1)h_{py}^0-g_s^{(1)}h_{py}^1,\\
  \mathcal{P}_h^{(1)}=&u_p^1h_{sx}^{(1)}+(v_p^1-\bar v_p^1)h_{py}^0+v_s^{(1)}h_{py}^1\\
  &-h_p^1u_{sx}^{(1)}-(g_p^1-\bar g_p^1)u_{py}^0-g_s^{(1)}u_{py}^1,\\
  \mathcal{P}_g^{(1)}=&u_p^1g_{sx}^{(1)}+(v_p^1-\bar v_p^1)g_{py}^0+v_s^{(1)}g_{py}^1\\
  &-h_p^1v_{sx}^{(1)}-(g_p^1-\bar g_p^1)v_{py}^0-g_s^{(1)}v_{py}^1,
\end{cases}
\end{align}
and
\begin{align}\label{f_u1}
\begin{split}
  &(f_u^{(1)},f_h^{(1)},f_g^{(1)})=-\eps^{-\frac{1}{2}}(R^{u,0}+R_E^{u,1}+\eps p_{px}^{1,a},R^{h,0}+R_E^{h,1},R^{g,0}+R_E^{g,1}).
\end{split}
\end{align}

The boundary contribution of $-\bar v_p^1u_{py}^0$ in \eqref{P_u1} is extracted from the following term
\begin{align}\label{overlinev_e^2}
  v_e^2(x,Y)u_{py}^0= \overline v_e^2u_{py}^0+\sqrt\eps yu_{py}^0v_{eY}^2+\eps u_{py}^0\int_0^y\int_y^\theta v_{eYY}^2(\sqrt\eps\tau) {\rm d}\tau{\rm d}\theta.
\end{align}

The arguments can be also applied for the other boundary terms in \eqref{u_p^1_system}. In addition, evaluating the equation \eqref{u_p^1_system}$_1$ at $y=\infty$, and using the identity $p_{py}^1=0$, one can deduce that $p_p^1=C$. Without loss of generality, we take $p_p^1=0$.
Then the remainder terms with $\sqrt{\eps}$-order are reduced to
\begin{align}
\label{R^u_1}
  R^{u,1}
  =&\sqrt\eps[-\eps u_{pxx}^1+\sqrt\eps(u_e^1+u_p^1)u_{px}^1+\sqrt\eps v_p^1(u_{py}^1+\sqrt\eps u_{eY}^1)-\sqrt\eps(h_e^1+h_p^1)h_{px}^1 \nonumber\\
  &-\sqrt\eps g_p^1(h_{py}^1+\sqrt\eps h_{eY}^1)+\sqrt\eps yv_{eY}^2u_{py}^0+\eps u_{py}^0\int_0^y\int_y^{\theta}v_{eYY}^2(\sqrt\eps\tau){\rm d}\tau{\rm d}\theta\nonumber\\
  &-\sqrt\eps yg_{eY}^2h_{py}^0-\eps h_{py}^0\int_0^y\int_y^{\theta}g_{eYY}^2(\sqrt\eps\tau){\rm d}\tau{\rm d}\theta],\\
\label{R^v_1}
  R^{v,1}
  =&\sqrt\eps[-\Delta_\eps v_p^1+\overline u_s^{(1)}v_{px}^1+u_p^1v_{sx}^{(1)}+v_s^{(1)}v_{py}^1+v_p^1\overline v_{sy}^{(1)}\nonumber\\
  &\qquad -\overline h_s^{(1)}g_{px}^1-h_p^1g_{sx}^{(1)}-g_s^{(1)}g_{py}^1-g_p^1\overline g_{sy}^{(1)}],\\
\label{R^h_1}
  R^{h,1}
  =&\sqrt\eps[-\eps h_{pxx}^1+\sqrt\eps(u_e^1+u_p^1)h_{px}^1+\sqrt\eps v_p^1(h_{py}^1+\sqrt\eps h_{eY}^1)-\sqrt\eps(h_e^1+h_p^1)u_{px}^1 \nonumber\\
  &-\sqrt\eps g_p^1(u_{py}^1+\sqrt\eps u_{eY}^1)+h_{py}^0(v_e^2-\bar v_e^2)-u_{py}^0(g_e^2-\bar g_e^2)]\nonumber\\
  =&\sqrt\eps[-\eps h_{pxx}^1+\sqrt\eps(u_e^1+u_p^1)h_{px}^1+\sqrt\eps v_p^1(h_{py}^1+\sqrt\eps h_{eY}^1)-\sqrt\eps(h_e^1+h_p^1)u_{px}^1 \nonumber\\
  &-\sqrt\eps g_p^1(u_{py}^1+\sqrt\eps u_{eY}^1)+\sqrt\eps yv_{eY}^2h_{py}^0+\eps h_{py}^0\int_0^y\int_y^{\theta}v_{eYY}^2(\sqrt\eps\tau){\rm d}\tau{\rm d}\theta\nonumber\\
  &-\sqrt\eps yg_{eY}^2u_{py}^0-\eps u_{py}^0\int_0^y\int_y^{\theta}g_{eYY}^2(\sqrt\eps\tau){\rm d}\tau{\rm d}\theta],\\
\label{R^g_1}
  R^{g,1}
  =&\sqrt\eps[-\eps g_{pxx}^1+\sqrt\eps(u_e^1+u_p^1)g_{px}^1+\sqrt\eps v_p^1(g_{py}^1+g_{eY}^1)\nonumber\\
  &-\sqrt\eps(h_e^1+h_p^1)v_{px}^1 -\sqrt\eps g_p^1(v_{py}^1+v_{eY}^1)+g_{py}^0(v_e^2-\bar v_e^2)\nonumber\\
  &-v_{py}^0(g_e^2-\bar g_e^2)+(1+u_p^0)\p_x(g_e^2-\bar g_e^2)-(\sigma+h_p^0)\p_x(v_e^2-\bar v_e^2)]\nonumber\\
  =&\sqrt\eps[-\eps g_{pxx}^1+\sqrt\eps(u_e^1+u_p^1)g_{px}^1+\sqrt\eps v_p^1(g_{py}^1+g_{eY}^1)-\sqrt\eps(h_e^1+h_p^1)v_{px}^1 \nonumber\\
  &-\sqrt\eps g_p^1(v_{py}^1+v_{eY}^1)+\sqrt\eps yv_{eY}^2g_{py}^0+\eps g_{py}^0\int_0^y\int_y^{\theta}v_{eYY}^2(\sqrt\eps\tau){\rm d}\tau{\rm d}\theta\nonumber\\
  &-\sqrt\eps yg_{eY}^2v_{py}^0-\eps v_{py}^0\int_0^y\int_y^{\theta}g_{eYY}^2(\sqrt\eps\tau){\rm d}\tau{\rm d}\theta\nonumber\\
  &+\sqrt\eps u_p^0\int_0^y\p_{xY}g_e^2(\sqrt\eps\tau){\rm d}\tau-\sqrt\eps h_p^0\int_0^y\p_{xY}v_e^2(\sqrt\eps\tau){\rm d}\tau].
\end{align}

Thanks to divergence-free conditions, we rewrite the equations \eqref{u_p^1_system}$_{2,3}$ as
\begin{align}\label{h_p^1_equation}
\begin{split}
  \p_y\left[-h_{py}^1-(1+u_p^0)(g_p^1-\bar g_p^1)+(\sigma+h_p^0)(v_p^1-\bar v_p^1)-u_p^1g_s^{(1)}+h_p^1v_s^{(1)}\right]\\
  +\eps^{\frac{1}{2}}\p_y\left[\eps g_{px}^0+\sqrt\eps(v_p^0h_e^1-g_p^0u_e^1)+h_p^0(v_e^1-\bar v_e^1)-u_p^0(g_e^1-\bar g_e^1)\right]=0
\end{split}
\end{align}
and
\begin{align}\label{g_p^1_equation}
\begin{split}
  \p_x\left[-h_{py}^1-(1+u_p^0)(g_p^1-\bar g_p^1)+(\sigma+h_p^0)(v_p^1-\bar v_p^1)-u_p^1g_s^{(1)}+h_p^1v_s^{(1)}\right]\\
  +\eps^{\frac{1}{2}}\p_x\left[\eps g_{px}^0+\sqrt\eps(v_p^0h_e^1-g_p^0u_e^1)+h_p^0(v_e^1-\bar v_e^1)-u_p^0(g_e^1-\bar g_e^1)\right]=0.
\end{split}
\end{align}

We can deduce from the equation \eqref{h_p^1_equation} with the boundary conditions of the profiles at infinity and the identity $\bar g_e^2=\sigma \bar v_e^2$ that
\begin{align}\label{h_p^1_equation1}
\begin{split}
  &-h_{py}^1-(1+u_p^0)(g_p^1-\bar g_p^1)+(\sigma+h_p^0)(v_p^1-\bar v_p^1)-u_p^1g_s^{(1)}+h_p^1v_s^{(1)}\\
  &+\eps^{\frac{1}{2}}\left[\eps g_{px}^0+\sqrt\eps(v_p^0h_e^1-g_p^0u_e^1)+h_p^0(v_e^1-\bar v_e^1)-u_p^0(g_e^1-\bar g_e^1)\right]=0.
\end{split}
\end{align}
It implies that the equation \eqref{g_p^1_equation} is a direct consequence of \eqref{h_p^1_equation}, i.e., the equation \eqref{u_p^1_system}$_3$ for vertical magnetic field is equivalent to \eqref{u_p^1_system}$_2$ for tangential component. Therefore, we only need to consider the problem \eqref{u_p^1_system} without \eqref{u_p^1_system}$_{3}$. Due to the divergence-free conditions and the decay behavior for  $(v_p^1,g_p^1)$ as $y\to \infty$, we stress that the vertical components $(v_p^1,g_p^1)$ are constructed by
\begin{equation}\label{v_p^1_identity}
  (v_p^1,g_p^1)(x,y)=\int_y^\infty \p_x(u_p^1,h_p^1)(x,\theta){\rm d}\theta.
\end{equation}

Note that the forcing terms $(f_u^{(1)},f_h^{(1)})$ can be estimated as follows (refer to \cite{Iyerglobal1} for the details)
\begin{equation}\label{f_u1_estimate}
  \Vert z^m \p_x^k(f_u^{(1)},f_h^{(1)})\Vert_{L_y^2}\leq C(k,m)x^{-k-\frac{5}{4}}.
\end{equation}

With $(u^0_p,v^0_p,h^0_p,g^0_p)$ and $(u^1_e,v^1_e,h^1_e,g^1_e)$ constructed in previous arguments, we can obtain the main result of this subsection as follows.
\begin{theorem}\label{u_p^1_theorem}
For any $m,k,j\in\mathbb{N}$, there exists solution $(u_p^1,v_p^1,h_p^1,g_p^1)$ to the  problem \eqref{u_p^1_system} in the domain $\Omega=[1,\infty)\times \mathbb{R}_+$ satisfying
\begin{align}
\label{uh_p^1_estimate}
  &\Vert z^m\p_x^k\p_y^j (u_p^1,h_p^1)\Vert_{L_y^\infty}\leq C(\delta,\sigma,\sigma_1;m,k,j)x^{-k-\frac{j}{2}-\frac{1}{4}+\sigma_1},\\
\label{vg_p^1_estimate}
  &\Vert z^m\p_x^k\p_y^j (v_p^1,g_p^1)\Vert_{L_y^\infty}\leq C(\delta,\sigma,\sigma_1;m,k,j)x^{-k-\frac{j}{2}-\frac{3}{4}+\sigma_1}.
\end{align}
\end{theorem}

To perform our arguments, let us introduce the following norms
\begin{align}\label{u_p^1_norm}
\begin{split}
  \Vert (u_p^1,h_p^1)\Vert_{\mathcal{P}(\sigma_1)}
  &:=\sup_{x\geq 1}\Vert (u_p^1,h_p^1)\cdot x^{-\sigma_1}\Vert_{L_y^2}+\sup_{x\geq 1}\Vert \p_y(u_p^1,h_p^1)\cdot x^{\frac{1}{2}-\sigma_1}\Vert_{L_y^2}\\
  &\quad +\Vert (u_p^1,h_p^1)\cdot x^{-\frac{1}{2}-\sigma_1}\Vert_{L_x^2L_y^2}+\Vert \p_y(u_p^1,h_p^1)\cdot x^{-\sigma_1}\Vert_{L_x^2L_y^2}\\
  &\quad +\Vert \p_x(u_p^1,h_p^1)\cdot x^{\frac{1}{2}-\sigma_1}\Vert_{L_x^2L_y^2},
\end{split}
\end{align}
and
\begin{align}\label{u_p^1_norm_k}
\begin{split}
  \Vert (u_p^1,h_p^1)\Vert_{\mathcal{P}_k(\sigma_1)}
  &:=\sup_{x\geq 1}\Vert \p_x^k(u_p^1,h_p^1)\cdot x^{k-\sigma_1}\Vert_{L_y^2}
  +\sup_{x\geq 1}\Vert \p_x^k\p_y(u_p^1,h_p^1)\cdot x^{k+\frac{1}{2}-\sigma_1}\Vert_{L_y^2}\\
  &\quad +\Vert \p_x^k(u_p^1,h_p^1)\cdot x^{k-\frac{1}{2}-\sigma_1}\Vert_{L_x^2L_y^2}
  +\Vert \p_x^k\p_y(u_p^1,h_p^1)\cdot x^{k-\sigma_1}\Vert_{L_x^2L_y^2}\\
  &\quad +\Vert \p_x^{k+1}(u_p^1,h_p^1)\cdot x^{k+\frac{1}{2}-\sigma_1}\Vert_{L_x^2L_y^2}.
\end{split}
\end{align}

\begin{proposition}\label{u_p^1_prop}
Suppose that  $(u_p^1,h_p^1)$ solve \eqref{u_p^1_system}, then for any fixed $\sigma_1>0$ and $m,k\in\mathbb{N}$, it holds that
\begin{align}
\label{u_p^1_estimate_P}
  &\Vert (u_p^1,h_p^1)\Vert_{\mathcal{P}(\sigma_1)}\leq C(\delta,\sigma,\sigma_1),\\
\label{u_p^1_estimate_Pk_weighted}
  &\Vert z^m(u_p^1,h_p^1)\Vert_{\mathcal{P}_k(\sigma_1)}\leq C(\delta,\sigma,\sigma_1;k,m),
\end{align}
where the norms $\mathcal{P}(\sigma_1)$ and $\mathcal{P}_k(\sigma_1)$ are defined by \eqref{u_p^1_norm}-\eqref{u_p^1_norm_k}, respectively.
\end{proposition}
\begin{proof}
Testing the system \eqref{u_p^1_system}$_{1,2}$ by $(u_p^1x^{-2\sigma_1},h_p^1x^{-2\sigma_1})$ respectively, and integrating by parts with respect to $y$ variable, one may obtain
\begin{align}
  &\frac{1}{2}\frac{\rm d}{{\rm d}x}\int_0^\infty (1+u_p^0)|(u_p^1,h_p^1)|^2 x^{-2\sigma_1}{\rm d}y
  +\int_0^\infty |\p_y(u_p^1,h_p^1)|^2 x^{-2\sigma_1}{\rm d}y \nonumber\\
  & +\sigma_1\int_0^\infty (1+u_p^0)|(u_p^1,h_p^1)|^2 x^{-2\sigma_1-1}{\rm d}y \nonumber\\
  =&\int_0^\infty (\sigma+h_p^0)\p_x(h_p^1u_p^1)x^{-2\sigma_1}{\rm d}y
  +\left[\p_yu_p^1(x,0)\overline u_e^1(x)+\p_yh_p^1(x,0)\overline h_e^1(x)\right]x^{-2\sigma_1} \nonumber\\
  &+\left[\frac{1}{2}\int_0^\infty u_{px}^0|(u_p^1,h_p^1)|^2 x^{-2\sigma_1}{\rm d}y
  +\int_0^\infty -(\mathcal{P}_u^{(1)},\mathcal{P}_h^{(1)})\cdot(u_p^1,h_p^1) x^{-2\sigma_1}{\rm d}y\right] \nonumber\\
  &+\int_0^\infty (f_u^{(1)},f_h^{(1)})\cdot(u_p^1,h_p^1) x^{-2\sigma_1}{\rm d}y
  :=I_1+I_2+I_3+I_4.\label{u_p^1_estimate_1.0}
\end{align}
For the first term on the right-hand side, we have
\begin{align}\label{u_p^1_estimate_1.1}
\begin{split}
I_1=&\int_0^\infty (\sigma+h_p^0)\p_x(h_p^1u_p^1)x^{-2\sigma_1}{\rm d}y=\frac{\rm d}{{\rm d}x}\int_0^\infty (\sigma+h_p^0)u_p^1h_p^1 x^{-2\sigma_1}{\rm d}y\\
  &-\int_0^\infty h_{px}^0u_p^1h_p^1 x^{-2\sigma_1}{\rm d}y+2\sigma_1\int_0^\infty (\sigma+h_p^0)u_p^1h_p^1 x^{-1-2\sigma_1}{\rm d}y\\
  \leq& \frac{\rm d}{{\rm d}x}\int_0^\infty (\sigma+h_p^0)u_p^1h_p^1 x^{-2\sigma_1}{\rm d}y
  +\Vert\sigma_1(\sigma+h_p^0)\Vert_{L^\infty}\Vert (u_{p}^1,h_{p}^1)x^{\frac{1}{2}-\sigma_1}\Vert_{L_y^2}^2\\
  &+\Vert xh_{px}^0\Vert_{L^\infty}\Vert (u_{p}^1,h_{p}^1)x^{\frac{1}{2}-\sigma_1}\Vert_{L_y^2}^2.
\end{split}
\end{align}
The boundary terms can be estimated as
\begin{align}\label{u_p^1_estimate_1.2}
\begin{split}
I_2=&\left[\p_yu_p^1(x,0)\overline u_e^1(x)+\p_yh_p^1(x,0)\overline h_e^1(x)\right]x^{-2\sigma_1}\\
  \leq& |\p_y(u_p^1,h_p^1)(x,0)x^{-\frac{1}{2}-2\sigma_1}|
  \leq\delta_0|\p_y(u_p^1,h_p^1)(x,0)x^{-\sigma_1}|^2+C|x^{-\frac{1}{2}-\sigma_1}|^2,
\end{split}
\end{align}
in which the estimate \eqref{u_e^1_estimate} has been used.

Due to \eqref{P_u1}, we treat the third term as follows
\begin{align}\label{u_p^1_estimate_1.3}
\begin{split}
I_3=&\frac{1}{2}\int_0^\infty u_{px}^0|(u_p^1,h_p^1)|^2 x^{-2\sigma_1}{\rm d}y
  +\int_0^\infty -(\mathcal{P}_u^{(1)},\mathcal{P}_h^{(1)})\cdot(u_p^1,h_p^1) x^{-2\sigma_1}{\rm d}y\\
  \leq& \Vert xu_{sx}^{(1)},yu_{py}^0,yh_{py}^0,xv_{sy}^{(1)},xg_{sy}^{(1)}\Vert_{L_y^\infty}\Vert (u_p^1,h_p^1) x^{-\frac{1}{2}-\sigma_1}\Vert_{L_y^2}^2\\
  & +\Vert yu_{py}^0,yh_{py}^0\Vert_{L_y^\infty}\Vert (v_{py}^1,g_{py}^1)x^{-\frac{1}{2}-\sigma_1}\Vert_{L_y^2}^2.
\end{split}
\end{align}

Thanks to \eqref{f_u1_estimate}, we have
\begin{align}\label{u_p^1_estimate_1.4}
\begin{split}
I_4=&\int_0^\infty (f_u^{(1)},f_h^{(1)})\cdot(u_p^1,h_p^1) x^{-2\sigma_1}{\rm d}y\\
  \leq& \Vert (f_u^{(1)},f_h^{(1)})x^{\frac{1}{2}-\sigma_1}\Vert_{L_y^2}\Vert (u_p^1,h_p^1) x^{-\frac{1}{2}-\sigma_1}\Vert_{L_y^2}\\
  \leq& Cx^{-\frac{3}{2}}+\delta_0\Vert (u_p^1,h_p^1) x^{-\frac{1}{2}-\sigma_1}\Vert_{L_y^2}^2.
\end{split}
\end{align}

Plugging \eqref{u_p^1_estimate_1.1}-\eqref{u_p^1_estimate_1.4} into \eqref{u_p^1_estimate_1.0}, integrating the result on $[1,x)$, we get
\begin{align}\label{u_p^1_estimate_1.5}
\begin{split}
  &\int_0^\infty |(u_p^1,h_p^1)|^2x^{-2\sigma_1}{\rm d}y+\int_1^x\int_0^\infty |(u_p^1,h_p^1)|^2x^{-1-2\sigma_1}{\rm d}y{\rm d}x\\
  &\quad +\int_1^x\int_0^\infty |(u_{py}^1,h_{py}^1)|^2x^{-2\sigma_1}{\rm d}y{\rm d}x\\
  \leq &\int_0^\infty (\sigma+h_p^0)u_p^1h_p^1 x^{-2\sigma_1}{\rm d}y
  +\delta_0\Vert \p_y(u_p^1,h_p^1)(x,0)x^{-\sigma_1}\Vert_{L_x^2L_y^2}^2\\
  &+\mathcal{O}(\delta)\Vert (v_{py}^1,g_{py}^1)x^{\frac{1}{2}-\sigma_1}\Vert_{L_x^2L_y^2}^2+C.
\end{split}
\end{align}

By virtue of the smallness of $\sigma$ and the estimate \eqref{uh_p^0_estimate2} of $h_p^0$, we have
\begin{align}\label{u_p^1_estimate_1.6}
  \int_0^\infty (\sigma+h_p^0)u_p^1h_p^1 x^{-2\sigma_1}{\rm d}y
  \leq \frac{1}{2}[\sigma+\mathcal{O}(\delta,\sigma)]\int_0^\infty |(u_p^1,h_p^1)|^2x^{-2\sigma_1}{\rm d}y.
\end{align}

Note that the term $\p_y(u_p^1,h_p^1)(x,0)$ can be bounded by
\begin{align}\label{u_p^1_estimate_1.7}
\begin{split}
  |\p_yu_p^1(x,0)|=\left|\int_0^\infty \p_y(e^{-y}u_{py}^1){\rm d}y\right|
  \leq \Vert e^{-y}\Vert_{L_y^2}\left(\Vert \p_y^2u_p^1\Vert_{L_y^2}+\Vert \p_yu_p^1\Vert_{L_y^2}\right),
\end{split}
\end{align}
then we obtain
\begin{align}\label{u_p^1_estimate_1.8}
\begin{split}
&  \Vert \p_y(u_p^1,h_p^1)(x,0)x^{-\sigma_1}\Vert_{L_x^2}^2\\
  &\leq \Vert \p_y(u_p^1,h_p^1)x^{-\sigma_1}\Vert_{L_x^2L_y^2}^2+\Vert \p_y^2(u_p^1,h_p^1)x^{-\sigma_1}\Vert_{L_x^2L_y^2}^2.
\end{split}
\end{align}

Using the estimates \eqref{u_p^1_estimate_1.6}-\eqref{u_p^1_estimate_1.8}, and taking supremum with respect to $x$ variable in \eqref{u_p^1_estimate_1.5}, one can deduce that
\begin{align}\label{u_p^1_estimate_1.9}
\begin{split}
  &\sup_{x\geq 1}\int_0^\infty |(u_p^1,h_p^1)|^2x^{-2\sigma_1}{\rm d}y+\int_1^\infty\int_0^\infty |(u_p^1,h_p^1)|^2x^{-1-2\sigma_1}{\rm d}y{\rm d}x\\
  &\quad +\int_1^\infty\int_0^\infty |(u_{py}^1,h_{py}^1)|^2x^{-2\sigma_1}{\rm d}y{\rm d}x\\
  \leq &\mathcal{O}(\delta)\Vert (v_{py}^1,g_{py}^1)x^{\frac{1}{2}-\sigma_1}\Vert_{L_x^2L_y^2}^2
  +\delta_0\Vert (u_{pyy}^1,h_{pyy}^1)x^{-\sigma_1}\Vert_{L_x^2L_y^2}^2+C.
\end{split}
\end{align}

Next, the application of multipliers $(u_{px}^1x^{1-2\sigma_1},h_{px}^1x^{1-2\sigma_1})$ to the system \eqref{u_p^1_system} gives that
\begin{align}\label{u_p^1_estimate_2.0}
\begin{split}
  &\frac{1}{2}\frac{\rm d}{{\rm d}x}\int_0^\infty |(u_{py}^1,h_{py}^1)|^2 x^{1-2\sigma_1}{\rm d}y
  +\int_0^\infty (1+u_p^0)|(u_{px}^1,h_{px}^1)|^2 x^{1-2\sigma_1}{\rm d}y\\
  =&2\int_0^\infty (\sigma+h_p^0)h_{px}^1u_{px}^1x^{1-2\sigma_1}{\rm d}y
  -\int_0^\infty(\mathcal{P}_u^{(1)},\mathcal{P}_h^{(1)})\cdot(u_{px}^1,h_{px}^1) x^{-2\sigma_1}{\rm d}y\\
  &+\int_0^\infty (f_u^{(1)},f_h^{(1)})\cdot(u_{px}^1,h_{px}^1) x^{-2\sigma_1}{\rm d}y+\left[u_{py}^1(x,0)\overline u_{ex}^1(x)x^{1-2\sigma_1}\right.\\
  &\left.+h_{py}^1(x,0)\overline h_{ex}^1(x)x^{1-2\sigma_1}\right]+\frac{1-2\sigma_1}{2}\int_0^\infty |(u_{py}^1,h_{py}^1)|^2 x^{-2\sigma_1}{\rm d}y\\
  :=&J_1+J_2+J_3+J_4+J_5.
\end{split}
\end{align}

The first four terms on the right hand side can be treated as follows
\begin{align}\label{u_p^1_estimate_2.1}
\begin{split}
  J_1=2\int_0^\infty (\sigma+h_p^0)h_{px}^1u_{px}^1x^{1-2\sigma_1}{\rm d}y
  \leq \Vert\sigma+h_p^0\Vert_{L^\infty}\Vert (u_{px}^1,h_{px}^1)x^{\frac{1}{2}-\sigma_1}\Vert_{L_y^2}^2,
\end{split}
\end{align}
\begin{align}\label{u_p^1_estimate_2.2}
\begin{split}
  J_2=&\int_0^\infty -(\mathcal{P}_u^{(1)},\mathcal{P}_h^{(1)})\cdot(u_{px}^1,h_{px}^1) x^{1-2\sigma_1}{\rm d}y\\
  \leq& \Vert xu_{sx}^{(1)}\Vert_{L^\infty}\Vert (u_{p}^1,h_{p}^1)x^{-\frac{1}{2}-\sigma_1}\Vert_{L_y^2}
  \Vert (u_{px}^1,h_{px}^1)x^{\frac{1}{2}-\sigma_1}\Vert_{L_y^2}\\
  &+\Vert yu_{py}^0,yh_{py}^0\Vert_{L^\infty}\Vert(v_{py}^1,g_{py}^1) x^{1-2\sigma_1}\Vert_{L_y^2}\Vert(u_{px}^1,h_{px}^1) x^{1-2\sigma_1}\Vert_{L_y^2}\\
  &+\Vert v_s^{(1)}x^{\frac{1}{2}},g_s^{(1)}x^{\frac{1}{2}}\Vert_{L^\infty}\Vert(u_{py}^1,h_{py}^1) x^{1-2\sigma_1}\Vert_{L_y^2}\Vert(u_{px}^1,h_{px}^1) x^{1-2\sigma_1}\Vert_{L_y^2},
\end{split}
\end{align}
\begin{align}\label{u_p^1_estimate_2.3}
\begin{split}
  J_3=&\int_0^\infty (f_u^{(1)},f_h^{(1)})\cdot(u_{px}^1,h_{px}^1) x^{1-2\sigma_1}{\rm d}y\\
  \leq& \Vert (f_u^{(1)},f_h^{(1)})x^{\frac{1}{2}-\sigma_1}\Vert_{L_y^2}\Vert (u_{px}^1,h_{px}^1) x^{\frac{1}{2}-\sigma_1}\Vert_{L_y^2}\\
  \leq& Cx^{-\frac{3}{2}}++\delta_0\Vert (u_{px}^1,h_{px}^1) x^{\frac{1}{2}-\sigma_1}\Vert_{L_y^2}^2,
\end{split}
\end{align}
\begin{align}\label{u_p^1_estimate_2.4}
\begin{split}
  J_4=&u_{py}^1(x,0)\overline u_{ex}^1(x)x^{1-2\sigma_1}+h_{py}^1(x,0)\overline h_{ex}^1(x)x^{1-2\sigma_1}\\
  \leq& |\p_y(u_p^1,h_p^1)(x,0)x^{-\frac{1}{2}-2\sigma_1}|
  \leq \delta_0|\p_y(u_p^1,h_p^1)(x,0)x^{-\sigma_1}|^2+C|x^{-\frac{1}{2}-\sigma_1}|^2.
\end{split}
\end{align}

 Plugging \eqref{u_p^1_estimate_2.1}-\eqref{u_p^1_estimate_2.4} into \eqref{u_p^1_estimate_2.0}, and integrating the result on $[1,x)$, we can obtain
\begin{align}\label{u_p^1_estimate_2.5}
\begin{split}
  &\sup_{x\geq 1}\int_0^\infty |(u_{py}^1,h_{py}^1)|^2 x^{1-2\sigma_1}{\rm d}y
  +\int_0^\infty (1+u_p^0)|(u_{px}^1,h_{px}^1)|^2 x^{1-2\sigma_1}{\rm d}y\\
  \leq& \Vert (u_{py}^1,h_{py}^1)x^{-\sigma_1}\Vert_{L_x^2L_y^2}^2
   +\mathcal{O}(\delta,\sigma)\Vert (u_{p}^1,h_{p}^1)x^{-\frac{1}{2}-\sigma_1}\Vert_{L_x^2L_y^2}^2\\
   &+\delta_0\Vert (u_{pyy}^1,h_{pyy}^1)x^{-\sigma_1}\Vert_{L_x^2L_y^2}^2+C.
\end{split}
\end{align}

According to the system \eqref{u_p^1_system}, we bound $\Vert (u_{pyy}^1,h_{pyy}^1)x^{-\sigma_1}\Vert_{L_x^2L_y^2}$ as
\begin{align}\label{u_p^1_estimate_2.6}
\begin{split}
  &\Vert (u_{pyy}^1,h_{pyy}^1)x^{-\sigma_1}\Vert_{L_x^2L_y^2}\\
  \leq& \Vert (u_{px}^1,h_{px}^1)x^{-\sigma_1}\Vert_{L_x^2L_y^2}\Vert (1+u_p^0,\sigma+h_p^0)\Vert_{L^\infty}
  +\Vert (f_{u}^{(1)},f_{h}^{(1)})x^{-\sigma_1}\Vert_{L_x^2L_y^2} \\
  &+\Vert xu_{sx}^{(1)},yu_{py}^0,yh_{py}^0,xv_{sy}^{(1)},xg_{sy}^{(1)}\Vert_{L_y^\infty}\Vert (u_p^1,h_p^1) x^{-\frac{1}{2}-\sigma_1}\Vert_{L_y^2}\\
  &+\Vert yu_{py}^0,yh_{py}^0\Vert_{L_y^\infty}\Vert (v_{py}^1,g_{py}^1)x^{-\frac{1}{2}-\sigma_1}\Vert_{L_y^2}.
\end{split}
\end{align}
Thus our desired estimate \eqref{u_p^1_estimate_P} follows from the inequalities \eqref{u_p^1_estimate_1.9},\eqref{u_p^1_estimate_2.5}-\eqref{u_p^1_estimate_2.6}.

One can also repeat the above arguments with any positive weight $z^m$ to get the weighted estimates. Finally, the estimate \eqref{u_p^1_estimate_Pk_weighted} follows from taking higher-order derivatives to the system $\eqref{u_p^1_system}$ and performing the similar arguments as above.
\end{proof}

Now we prove Theorem \ref{u_p^1_theorem}.

\begin{proof}[Proof of Theorem \ref{u_p^1_theorem}]
With Proposition \ref{u_p^1_prop} at hands, the existence of the solution to \eqref{u_p^1_system} can be obtained via the contraction mapping theorem and standard parabolic theory. Here we focus on the estimates. We take the estimates for velocity field $(u_p^1,v_p^1)$ as an example, similar arguments also work for magnetic field $(h_p^1,g_p^1)$. Let us begin with the first term, as for $j=m=0$, there holds that
\begin{align}\label{u_p^1_lemma_1}
\begin{split}
|\p_x^k u_p^1x^{k+\frac{1}{4}-\sigma_1}|^2
  &\leq\left|2\int_y^\infty \p_x^ku_p^1\cdot \p_x^ku_{py}^1x^{2(k+\frac{1}{4}-\sigma_1)}\right|\\
  &\leq 2\Vert \p_x^ku_p^1x^{k-\sigma_1}\Vert_{L_y^2}\cdot \Vert \p_x^ku_{py}^1x^{k+\frac{1}{2}-\sigma_1}\Vert_{L_y^2}
  \leq \Vert u_p^1\Vert_{\mathcal{P}_k(\sigma_1)}^2.
\end{split}
\end{align}
For $j=0$ with $m>0$, one has
\begin{align}\label{u_p^1_lemma_2}
\begin{split}
  &|z^m\p_x^ku_p^1x^{k+\frac{1}{4}-\sigma_1}|^2
  =\left|2\int_y^\infty\p_x^ku_p^1\cdot x^{2(k+\frac{1}{4}-\sigma_1)}\cdot \p_y(\p_x^ku_p^1 z^{2m})\right|\\
  \leq& \Vert z^m\p_x^ku_p^1x^{k-\sigma_1}\Vert_{L_y^2}\cdot\left(\Vert z^m\p_x^k\p_yu_p^1x^{k+\frac{1}{2}-\sigma_1}\Vert_{L_y^2}
  +\Vert z^{m-1}\p_x^ku_p^1x^{k-\sigma_1}\Vert_{L_y^2}\right)\\
  \leq& C(k,m)\Vert z^mu_p^1\Vert_{\mathcal{P}_k(\sigma_1)}^2.
\end{split}
\end{align}
Similarly, for $j>0$, we get that
\begin{align}\label{u_p^1_lemma_3}
\begin{split}
  |z^m\p_x^ku_{py}^1 x^{k+\frac{3}{4}-\sigma_1}|^2
  \leq& \Vert x^{k+\frac{1}{2}-\sigma_1}z^m\p_x^ku_{py}^1\Vert_{L_y^2}\cdot \Vert x^{k+1-\sigma_1}\p_y(z^m\p_x^k u_{py}^1)\Vert_{L_y^2}\\
  \leq& \Vert z^mu_p^1\Vert_{\mathcal{P}_k(\sigma_1)}^2+\Vert z^mu_p^1\Vert_{\mathcal{P}_{k+1}(\sigma_1)}^2.
\end{split}
\end{align}

For the vertical components, note that the behavior of $(v_p^1,g_p^1)$ as $y\to \infty$, then  we have
\begin{align}\label{v_p^1_lemma_1}
  |z^mv_{p}^1|^2
  =&\int_y^\infty \p_y|z^mv_{p}^1|^2
  \leq 2\int_y^\infty |z^mv_p^1|\cdot |z^mv_{py}^1+mz^{m-1}x^{-\frac{1}{2}}v_p^1|\nonumber\\
  \leq& 2\left\Vert z^m\int_y^\infty v_{py}^1\right\Vert_{L_y^2}\cdot \left(\Vert z^mv_{py}^1\Vert_{L_y^2}+m\left\Vert z^{m}y^{-1}\int_y^\infty v_{py}^1\right\Vert_{L_y^2}\right)\nonumber\\
  \lesssim& (m+1)\Vert x^{\frac{1}{2}}z^{m+1}v_{py}^1\Vert_{L_y^2}\cdot\Vert z^mv_{py}^1\Vert_{L_y^2}
  \leq C(\delta,\sigma,\sigma_1;m,k)x^{-\frac{3}{2}+2\sigma_1},
\end{align}
where we have used Lemma \ref{Hardy-type} with $p=2,\alpha=2$. Similarly, the case of $k\geq 1$ can be also concluded.
\end{proof}

\section{Construction of $\eps^{\frac{i}{2}}$-order correctors ($2\leq i\leq n$)}\label{sec4}
\subsection{Derivation of the system}\label{sec4.1}
To obtain the estimates for the remainder profiles for later use, we should expand the approximate solutions to higher order in $\sqrt{\eps}$. Let us start with deriving the system for the $\eps$-order correctors. And the similar arguments can be also applied for the $\eps^{\frac{i}{2}}$-order correctors with $i>2$. Recall that
\begin{align}
\label{R_app^u,2}
R_{app}^{u,2}
  =&-\Delta_\eps\bar u_s^{(2)}+\bar u_s^{(2)}\bar u_{sx}^{(2)}+\bar v_s^{(2)}\bar u_{sy}^{(2)}+\p_x\bar p_x^{(2)}-\bar h_s^{(2)}\bar h_{sx}^{(2)}-\bar g_s^{(2)}\bar h_{sy}^{(2)}\nonumber\\
  =&R^{u,1}+\eps[-\Delta_\eps u_p^2+\overline u_s^{(2)}u_{px}^2+u_p^2u_{sx}^{(2)}+v_s^{(2)}u_{py}^2+v_p^2\overline u_{sy}^{(2)}+p_{px}^2\nonumber\\
  &-\overline h_s^{(2)}h_{px}^2-h_p^2h_{sx}^{(2)}-g_s^{(2)}h_{py}^2-g_p^2\overline h_{sy}^{(2)}]+\eps^{\frac{3}{2}} p_{px}^{2,a}+R_E^{u,2}\nonumber\\
  &+(E_E^{u,2}+\eps^2 p_{ex}^{2,a})-\eps^2 \Delta u_e^2+\eps(p_{ex}^2+u_{ex}^2-\sigma h_{ex}^2),\\
\label{R_app^v,2}
R_{app}^{v,2}
  =&-\Delta_\eps\bar v_s^{(2)}+\bar u_s^{(2)}\bar v_{sx}^{(2)}+\bar v_s^{(2)}\bar v_{sy}^{(2)}+\eps^{-1}\p_y\bar p_s^{(2)}-\bar h_s^{(2)}\bar g_{sx}^{(2)}-\bar g_s^{(2)}\bar g_{sy}^{(2)}\nonumber\\
  =&R^{v,1}+\eps[-\Delta_\eps v_p^2+\overline u_s^{(2)}v_{px}^2+u_p^2v_{sx}^{(2)}+v_s^{(2)}v_{py}^2+v_p^2\overline v_{sy}^{(2)}+\eps^{-1}p_{py}^2\nonumber\\
  &-\overline h_s^{(2)}g_{px}^2-h_p^2g_{sx}^{(2)}-g_s^{(2)}g_{py}^2-g_p^2\overline g_{sy}^{(2)}]+\eps^{\frac{1}{2}}p_{py}^{2,a}+R_E^{v,2}\nonumber\\
  &+(E_E^{v,2}+\eps\p_y p_{e}^{2,a})-\eps^{\frac{3}{2}}\Delta v_e^2+\sqrt\eps(p_{eY}^2+v_{ex}^2-\sigma g_{ex}^2),\\
\label{R_app^h,2}
R_{app}^{h,2}
  =&-\Delta_\eps\bar h_s^{(2)}+\bar u_s^{(2)}\bar h_{sx}^{(2)}+\bar v_s^{(2)}\bar h_{sy}^{(2)}-\bar h_s^{(2)}\bar u_{sx}^{(2)}-\bar g_s^{(2)}\bar u_{sy}^{(2)}\nonumber\\
  =&R^{h,1}+\eps[-\Delta_\eps h_p^2+\overline u_s^{(2)}h_{px}^2+u_p^2h_{sx}^{(2)}+v_s^{(2)}h_{py}^2+v_p^2\overline h_{sy}^{(2)}-\overline h_s^{(2)}u_{px}^2\nonumber\\
  &-h_p^2u_{sx}^{(2)}-g_s^{(2)}u_{py}^2-g_p^2\overline u_{sy}^{(2)}]+R_E^{h,2}-\eps^2 \Delta h_e^2+\eps(h_{ex}^2-\sigma u_{ex}^2),\\
\label{R_app^g,2}
R_{app}^{g,2}
  =&-\Delta_\eps\bar g_s^{(2)}+\bar u_s^{(2)}\bar g_{sx}^{(2)}+\bar v_s^{(2)}\bar g_{sy}^{(2)}-\bar h_s^{(2)}\bar v_{sx}^{(2)}-\bar g_s^{(2)}\bar v_{sy}^{(2)}\nonumber\\
  =&R^{g,1}+\eps[-\Delta_\eps g_p^2+\overline u_s^{(2)}g_{px}^2+u_p^2g_{sx}^{(2)}+v_s^{(2)}g_{py}^2+v_p^2\overline g_{sy}^{(2)}-\overline h_s^{(2)}v_{px}^2\nonumber\\
  &-h_p^2v_{sx}^{(2)}-g_s^{(2)}v_{py}^2-g_p^2\overline v_{sy}^{(2)}]+R_E^{g,2}-\eps^{\frac{3}{2}}\Delta g_e^2+\sqrt\eps(g_{ex}^2-\sigma v_{ex}^2),
\end{align}
in which  $R^{u,1},R^{v,1},R^{h,1},R^{g,1}$ are defined as in \eqref{R^u_1}-\eqref{R^g_1}.

From the above expansions, we take the following equations for the $\eps$-order ideal MHD correctors $(u_e^2,v_e^2,h_e^2,g_e^2)$
\begin{equation}\label{u_e^2_system_pre}
\begin{cases}
  \p_x u_e^2-\sigma\p_x h_e^2+\p_x p_e^2=0,\\
  \p_x v_e^2-\sigma\p_x g_e^2+\p_Y p_e^2=0,\\
  \p_x h_e^2-\sigma\p_x u_e^2=0,\\
  \p_x g_e^2-\sigma\p_x v_e^2=0,\\
  \p_x u_e^2+\p_Y v_e^2=\p_x h_e^2+\p_Y g_e^2=0.
\end{cases}
\end{equation}

The Euler-Prandtl terms $R_E^{u,2},R_E^{v,2},R_E^{h,2},R_E^{g,2}$ are defined by
\begin{align}
\label{R^u,2_E}
R_E^{u,2}
  &=\eps \left[u_e^2 \sum_{j=0}^1 \eps^{\frac{j}{2}}u_{px}^j-h_e^2 \sum_{j=0}^1 \eps^{\frac{j}{2}}h_{px}^j
  + u_{ex}^2 \sum_{j=0}^1 \eps^{\frac{j}{2}}u_p^j-h_{ex}^2 \sum_{j=0}^1 \eps^{\frac{j}{2}}h_p^j\right]\nonumber\\
  &+\sqrt\eps [v_e^2\eps^{\frac{1}{2}}u_{py}^1- g_e^2\eps^{\frac{1}{2}}h_{py}^1]+\eps^{\frac{3}{2}}\left[u_{eY}^2\sum_{j=0}^1 \eps^{\frac{j}{2}}v_p^j-h_{eY}^2\sum_{j=0}^1 \eps^{\frac{j}{2}}g_p^j\right],\\
\label{R^v,2_E}
R_E^{v,2}
  &=\sqrt\eps \left[v_e^2 \sum_{j=0}^1 \eps^{\frac{j}{2}}v_{py}^j-g_e^2 \sum_{j=0}^1 \eps^{\frac{j}{2}}g_{py}^j
  + v_{ex}^2 \sum_{j=0}^1 \eps^{\frac{j}{2}}u_p^j-g_{ex}^2 \sum_{j=0}^1 \eps^{\frac{j}{2}}h_p^j\right]\nonumber\\
  &\qquad +\eps \left[u_e^2\sum_{j=0}^1 \eps^{\frac{j}{2}}v_{px}^j- h_e^2\sum_{j=0}^1 \eps^{\frac{j}{2}}g_{px}^j
  +v_{eY}^2\sum_{j=0}^1 \eps^{\frac{j}{2}}v_p^j-g_{eY}^2\sum_{j=0}^1 \eps^{\frac{j}{2}}g_p^j\right],\\
\label{R^h,2_E}
R_E^{h,2}
  &=\eps \left[u_e^2 \sum_{j=0}^1 \eps^{\frac{j}{2}}h_{px}^j-h_e^2 \sum_{j=0}^1 \eps^{\frac{j}{2}}u_{px}^j
  + h_{ex}^2 \sum_{j=0}^1 \eps^{\frac{j}{2}}u_p^j-u_{ex}^2 \sum_{j=0}^1 \eps^{\frac{j}{2}}h_p^j\right]\nonumber\\
  & +\sqrt\eps \left[v_e^2\eps^{\frac{1}{2}}h_{py}^1-g_e^2\eps^{\frac{1}{2}}u_{py}^1\right]+\eps^{\frac{3}{2}}\left[h_{eY}^2\sum_{j=0}^1 \eps^{\frac{j}{2}}v_p^j-u_{eY}^2\sum_{j=0}^1 \eps^{\frac{j}{2}}g_p^j\right],\\
\label{R^g,2_E}
R_E^{g,2}
  &=\eps \left[u_e^2 \sum_{j=0}^1 \eps^{\frac{j}{2}}g_{px}^j-h_e^2 \sum_{j=0}^1 \eps^{\frac{j}{2}}v_{px}^j
  + g_{eY}^2 \sum_{j=0}^1 \eps^{\frac{j}{2}}v_p^j-v_{eY}^2 \sum_{j=0}^1 \eps^{\frac{j}{2}}g_p^j\right]\nonumber\\
  &\qquad +\sqrt\eps \left[v_e^2\eps^{\frac{1}{2}}g_{py}^1- g_e^2\eps^{\frac{1}{2}}v_{py}^1\right]
  +\sqrt\eps \left[g_{ex}^2\eps^{\frac{1}{2}}u_p^1-v_{ex}^2\eps^{\frac{1}{2}}h_p^1\right].
\end{align}

As the discussions for the $\sqrt{\eps}$-order correctors, let us introduce the following auxiliary boundary layer pressure
\begin{align}\label{p_p^2,a}
p_p^{2,a}=\eps^{-\frac{1}{2}}\int_y^\infty (R^{v,1}+R_E^{v,2})
\end{align}
 in the expansion \eqref{R_app^v,2} of $R_{app}^{v,2}$, which appears as new term $\eps^{\frac{3}{2}} p_{px}^{2,a}$ in expansion \eqref{R_app^u,2} of $R_{app}^{u,2}$ to preserve divergence free.

Observe that the pure Euler terms $E_E^{u,2}$ and $E_E^{v,2}$ are of gradient type as follows
\begin{align}\label{E_E^u,2}
\begin{split}
  E_E^{u,2}
  =&\eps^2 (u_e^2u_{ex}^2+v_e^2u_{eY}^2-h_e^2h_{ex}^2-g_e^2h_{eY}^2)\\
  &+\eps^{\frac{3}{2}}(u_e^2u_{ex}^1+u_e^1u_{ex}^2-h_e^2h_{ex}^1-h_{ex}^2h_e^1)\\
  &+\eps^{\frac{3}{2}}(v_e^2u_{eY}^1+v_e^1u_{eY}^2-g_e^2h_{eY}^1-g_e^1h_{eY}^2)\\
  =&\eps^2(1-\sigma^2)\p_x
  \left[\frac{1}{2}|u_e^2|^2+\frac{1}{2}|v_e^2|^2+\eps^{-\frac{1}{2}}(u_e^2u_e^1+v_e^2v_e^1)\right]
\end{split}
\end{align}
and
\begin{align}\label{E_E^v,2}
\begin{split}
  E_E^{v,2}
  =&\eps^\frac{3}{2} (u_e^2v_{ex}^2+v_e^2v_{eY}^2-h_e^2g_{ex}^2-g_e^2g_{eY}^2)\\
  &+\eps(u_e^1v_{ex}^2+v_e^2v_{eY}^1-h_e^1g_{ex}^2-g_e^2g_{eY}^1)\\
  &+\eps(u_e^2v_{ex}^1+v_e^1v_{eY}^2-h_e^2g_{ex}^1-g_e^1g_{eY}^2),\\
  =&\eps^{\frac{3}{2}}(1-\sigma^2)\p_Y
  \left[\frac{1}{2}|u_e^2|^2+\frac{1}{2}|v_e^2|^2+\eps^{-\frac{1}{2}}(u_e^2u_e^1+v_e^2v_e^1)\right].
\end{split}
\end{align}

We introduce the auxiliary ideal MHD pressure $p_e^{2,a}$
\begin{align}\label{p_e^2,a}
  p_e^{2,a}=-(1-\sigma^2)\left[\frac{1}{2}|u_e^2|^2+\frac{1}{2}|v_e^2|^2+\eps^{-\frac{1}{2}}(u_e^2u_e^1+v_e^2v_e^1)\right]
\end{align}
to ensure that
\begin{align}\label{E_E+p_e^2,a}
  E_E^{u,2}+\eps^2 p_{ex}^{2,a}=0,\quad
  E_E^{v,2}+\eps p_{ey}^{2,a}=0.
\end{align}

The initial boundary value problem for the $\eps$-order boundary layer correctors $(u_p^2,v_p^2,h_p^2,g_p^2)$ is taken as
\begin{equation}\label{u_p^2_system}
\begin{cases}
  -u_{pyy}^2+(1+u_p^0)u_{px}^2+p_{px}^2=(\sigma+h_p^0)h_{px}^2 -\mathcal{P}_u^{(2)}+f_u^{(2)},\\
  -h_{pyy}^2+(1+u_p^0)h_{px}^2=(\sigma+h_p^0)u_{px}^2 -\mathcal{P}_h^{(2)}+f_h^{(2)},\\
  -g_{pyy}^2+(1+u_p^0)\p_x(g_p^2-\bar g_p^2)=(\sigma+h_p^0)\p_x(v_p^2-\bar v_p^2)-\mathcal{P}_g^{(2)}+f_g^{(2)},\\
  p_{py}^2=0,\\
  u_{px}^2+v_{py}^2=h_{px}^2+g_{py}^2=0,
\end{cases}
\end{equation}
where the convective terms and the forcing terms are defined by
\begin{equation}\label{P_u2}
\begin{cases}
  \mathcal{P}_u^{(2)}=&u_p^2u_{sx}^{(2)}+(v_p^2-\bar v_p^2)u_{py}^0+v_s^{(2)}u_{py}^2\\
  &-h_p^2h_{sx}^{(2)}-(g_p^2-\bar g_p^2)h_{py}^0-g_s^{(2)}h_{py}^2,\\
  \mathcal{P}_h^{(2)}=&u_p^2h_{sx}^{(2)}+(v_p^2-\bar v_p^2)h_{py}^0+v_s^{(2)}h_{py}^2\\
  &-h_p^2u_{sx}^{(2)}-(g_p^2-\bar g_p^2)u_{py}^0-g_s^{(2)}u_{py}^2,\\
  \mathcal{P}_g^{(2)}=&u_p^2g_{sx}^{(2)}+(v_p^2-\bar v_p^2)g_{py}^0+v_s^{(2)}g_{py}^2\\
  &-h_p^2v_{sx}^{(2)}-(g_p^2-\bar g_p^2)v_{py}^0-g_s^{(2)}v_{py}^2,
\end{cases}
\end{equation}
and
\begin{align}\label{f_u2}
\begin{split}
  &(f_u^{(2)},f_h^{(2)},f_g^{(2)})=-\eps^{-1}(R^{u,1}+R_E^{u,2}+\eps p_{px}^{2,a},R^{h,1}+R_E^{h,2},R^{g,1}+R_E^{g,2}),
\end{split}
\end{align}
here $R^{u,1},R^{h,1},R^{g,1}$ are defined as in \eqref{R^u_1},\eqref{R^h_1}-\eqref{R^g_1}. Similarly, one can derive the equations for the $\eps^\frac{i}{2}$-order correctors with $i>2$.

 By divergence-free conditions, one can rewrite equations \eqref{u_p^2_system}$_{2,3}$ as
\begin{equation}\label{h_p^2_equation}
\begin{aligned}
  &\p_y\left[-h_{py}^2-(1+u_p^0)(g_p^2-\bar g_p^2)+(\sigma+h_p^0)(v_p^2-\bar v_p^2)-u_p^2g_s^{(2)}+h_p^2v_s^{(2)}\right]\\
  &+\p_y\bigg[\eps^{\frac{3}{2}} g_{p x}^{1}+\sqrt\eps\{h_p^0(v_e^2+\bar v_e^2)-u_p^0(g_e^2+\bar g_e^2)\}\\
  &+\eps\{g_p^1(u_e^1+u_p^1)-v_p^1(h_e^1+h_p^1)\}+\eps(h_{p}^{1} v_{e}^{2}-u_{p}^{1} g_{e}^{2})\bigg]\\
  &+\eps\p_y\left(h_{e}^{2} \sum_{j=0}^{1} \eps^{\frac{j}{2}} v_{p}^{j}-u_{e}^{2} \sum_{j=0}^{1} \eps^{\frac{j}{2}} g_{p}^{j}\right)=0
\end{aligned}
\end{equation}
and
\begin{equation}\label{g_p^2_equation}
\begin{aligned}
  &\p_x\left[-h_{py}^2-(1+u_p^0)(g_p^2-\bar g_p^2)+(\sigma+h_p^0)(v_p^2-\bar v_p^2)-u_p^2g_s^{(2)}+h_p^2v_s^{(2)}\right]\\
  &+\p_x\bigg[\eps^{\frac{3}{2}} g_{p x}^{1}+\sqrt\eps\{h_p^0(v_e^2+\bar v_e^2)-u_p^0(g_e^2+\bar g_e^2)\}\\
  &+\eps\{g_p^1(u_e^1+u_p^1)-v_p^1(h_e^1+h_p^1)\}+\eps(h_{p}^{1} v_{e}^{2}-u_{p}^{1} g_{e}^{2})\bigg]\\
  &+\eps\p_x\left(h_{e}^{2} \sum_{j=0}^{1} \eps^{\frac{j}{2}} v_{p}^{j}-u_{e}^{2} \sum_{j=0}^{1} \eps^{\frac{j}{2}} g_{p}^{j}\right)=0.
\end{aligned}
\end{equation}

Note that the equation \eqref{u_p^2_system}$_3$ for vertical magnetic field is equivalent to \eqref{u_p^2_system}$_2$. Then, we only need to consider the problem \eqref{u_p^2_system} without the third equation \eqref{u_p^2_system}$_{3}$, so as for the problem of $\eps^{\frac{i}{2}}$-order boundary layer correctors.

And thus, the remainder terms with $\eps^{\frac{i}{2}}$-order ($2\leq i\leq n-1$) are reduced to
\begin{align}
\label{R^u_i}
  R^{u,i}
  =&\eps^{\frac{i}{2}}[-\eps u_{pxx}^i+ u_{px}^i \sum_{j=1}^i \eps^{\frac{j}{2}}(u_e^j+u_p^j)- h_{px}^i \sum_{j=1}^i \eps^{\frac{j}{2}}(h_e^j+h_p^j)
  +\sqrt\eps yv_{eY}^{i+1}u_{py}^0\nonumber\\
  &+v_p^i \sum_{j=1}^i \eps^{\frac{j}{2}}(u_{py}^j+\sqrt\eps u_{eY}^j)-g_p^i \sum_{j=1}^i \eps^{\frac{j}{2}}(h_{py}^j+\sqrt\eps h_{eY}^j)
  -\sqrt\eps yg_{eY}^{i+1}h_{py}^0\nonumber\\
  &+\sqrt\eps u_{py}^0\int_0^y\int_y^{\theta}v_{eYY}^{i+1}(\sqrt\eps\tau){\rm d}\tau{\rm d}\theta
  -\sqrt\eps h_{py}^0\int_0^y\int_y^{\theta}g_{eYY}^{i+1}(\sqrt\eps\tau){\rm d}\tau{\rm d}\theta],\\
\label{R^v_i}
  R^{v,i}
  =&\eps^{\frac{i}{2}}[-\Delta_\eps v_p^i+\overline u_s^{(i)}v_{px}^i+u_p^iv_{sx}^{(i)}+v_s^{(i)}v_{py}^i+v_p^i\overline v_{sy}^{(i)}\nonumber\\
  &\qquad -\overline h_s^{(i)}g_{px}^i-h_p^ig_{sx}^{(i)}-g_s^{(i)}g_{py}^i-g_p^i\overline g_{sy}^{(i)}],\\
\label{R^h_i}
  R^{h,i}
  =&\eps^{\frac{i}{2}}[-\eps h_{pxx}^i+ h_{px}^i \sum_{j=1}^i \eps^{\frac{j}{2}}(u_e^j+u_p^j)- u_{px}^i \sum_{j=1}^i \eps^{\frac{j}{2}}(h_e^j+h_p^j)
  +h_{py}^0(v_e^{i+1}-\bar v_e^{i+1})\nonumber\\
  &+v_p^i \sum_{j=1}^i \eps^{\frac{j}{2}}(h_{py}^j+\sqrt\eps h_{eY}^j)-g_p^i \sum_{j=1}^i \eps^{\frac{j}{2}}(u_{py}^j+\sqrt\eps u_{eY}^j)
  -u_{py}^0(g_e^{i+1}-\bar g_e^{i+1})]\nonumber\\
  =&\eps^{\frac{i}{2}}[-\eps h_{pxx}^i+ h_{px}^i \sum_{j=1}^i \eps^{\frac{j}{2}}(u_e^j+u_p^j)- u_{px}^i \sum_{j=1}^i \eps^{\frac{j}{2}}(h_e^j+h_p^j)
  +\sqrt\eps yv_{eY}^{i+1}h_{py}^0\nonumber\\
  &+v_p^i \sum_{j=1}^i \eps^{\frac{j}{2}}(h_{py}^j+\sqrt\eps h_{eY}^j)-g_p^i \sum_{j=1}^i \eps^{\frac{j}{2}}(u_{py}^j+\sqrt\eps u_{eY}^j)
  -\sqrt\eps yg_{eY}^{i+1}u_{py}^0\nonumber\\
  &+\sqrt\eps h_{py}^0\int_0^y\int_y^{\theta}v_{eYY}^{i+1}(\sqrt\eps\tau){\rm d}\tau{\rm d}\theta
  -\sqrt\eps u_{py}^0\int_0^y\int_y^{\theta}g_{eYY}^{i+1}(\sqrt\eps\tau){\rm d}\tau{\rm d}\theta],\\
\label{R^g_i}
  R^{g,i}
  =&\eps^{\frac{i}{2}}[-\eps g_{pxx}^i+ g_{px}^i \sum_{j=1}^i \eps^{\frac{j}{2}}(u_e^j+u_p^j)- v_{px}^i \sum_{j=1}^i \eps^{\frac{j}{2}}(h_e^j+h_p^j)+g_{py}^0(v_e^{i+1}-\bar v_e^{i+1})\nonumber\\
  &+v_p^i \sum_{j=1}^i \eps^{\frac{j}{2}}(g_{py}^j+g_{eY}^j)-g_p^i \sum_{j=1}^i \eps^{\frac{j}{2}}(v_{py}^j+ v_{eY}^j)-v_{py}^0(g_e^{i+1}-\bar g_e^{i+1})\nonumber\\
  &+(1+u_p^0)\p_x(g_e^{i+1}-\bar g_e^{i+1})-(\sigma+h_p^0)\p_x(v_e^{i+1}-\bar v_e^{i+1})]\nonumber\\
  =&\eps^{\frac{i}{2}}[-\eps g_{pxx}^i+ g_{px}^i \sum_{j=1}^i \eps^{\frac{j}{2}}(u_e^j+u_p^j)- g_{px}^i \sum_{j=1}^i \eps^{\frac{j}{2}}(h_e^j+h_p^j)
  +\sqrt\eps yv_{eY}^{i+1}g_{py}^0\nonumber\\
  &+v_p^i \sum_{j=1}^i \eps^{\frac{j}{2}}(g_{py}^j+g_{eY}^j)-g_p^i \sum_{j=1}^i \eps^{\frac{j}{2}}(v_{py}^j+ v_{eY}^j)
  -\sqrt\eps yg_{eY}^{i+1}v_{py}^0\nonumber\\
  &+\sqrt\eps g_{py}^0\int_0^y\int_y^{\theta}v_{eYY}^{i+1}(\sqrt\eps\tau){\rm d}\tau{\rm d}\theta
  -\sqrt\eps v_{py}^0\int_0^y\int_y^{\theta}g_{eYY}^{i+1}(\sqrt\eps\tau){\rm d}\tau{\rm d}\theta\nonumber\\
  &+\sqrt\eps u_p^0\int_0^y\p_{xY}g_e^{i+1}(\sqrt\eps\tau){\rm d}\tau-\sqrt\eps h_p^0\int_0^y\p_{xY}v_e^{i+1}(\sqrt\eps\tau){\rm d}\tau].
\end{align}

For later use, we also give the formulation of the general Euler-Prandtl terms $R_E^{u,i},R_E^{v,i},R_E^{h,i},R_E^{g,i}$ ($2\leq i\leq n-1$) as follows
\begin{align}
\label{R^u,i_E}
&R_E^{u,i}
  =\eps^{\frac{i}{2}} \left[u_e^i \sum_{j=0}^{i-1} \eps^{\frac{j}{2}}u_{px}^j-h_e^i \sum_{j=0}^{i-1} \eps^{\frac{j}{2}}h_{px}^j
  +\sqrt\eps (u_{eY}^i\sum_{j=0}^{i-1} \eps^{\frac{j}{2}}v_p^j-h_{eY}^i\sum_{j=0}^{i-1} \eps^{\frac{j}{2}}g_p^j)\right]\nonumber\\
  &~ +\eps^{\frac{i-1}{2}} \left[v_e^i\sum_{j=1}^{i-1}\eps^{\frac{j}{2}}u_{py}^j-g_e^i\sum_{j=1}^{i-1}\eps^{\frac{j}{2}}h_{py}^j
  +\sqrt\eps(u_{ex}^i \sum_{j=0}^{i-1} \eps^{\frac{j}{2}}u_p^j-h_{ex}^i \sum_{j=0}^{i-1} \eps^{\frac{j}{2}}h_p^j )\right],\\
\label{R^v,i_E}
&R_E^{v,i}
=\sqrt\eps \left[v_e^2 \sum_{j=0}^1 \eps^{\frac{j}{2}}v_{py}^j-g_e^2 \sum_{j=0}^1 \eps^{\frac{j}{2}}g_{py}^j
  + v_{ex}^2 \sum_{j=0}^1 \eps^{\frac{j}{2}}u_p^j-g_{ex}^2 \sum_{j=0}^1 \eps^{\frac{j}{2}}h_p^j\right]\nonumber\\
  &\qquad +\eps \left[u_e^2\sum_{j=0}^1 \eps^{\frac{j}{2}}v_{px}^j- h_e^2\sum_{j=0}^1 \eps^{\frac{j}{2}}g_{px}^j
  +v_{eY}^2\sum_{j=0}^1 \eps^{\frac{j}{2}}v_p^j-g_{eY}^2\sum_{j=0}^1 \eps^{\frac{j}{2}}g_p^j\right],\\
\label{R^h,i_E}
&R_E^{h,i}
=\eps^{\frac{i}{2}} \left[u_e^i \sum_{j=0}^{i-1} \eps^{\frac{j}{2}}h_{px}^j-h_e^i \sum_{j=0}^{i-1} \eps^{\frac{j}{2}}u_{px}^j
  +\sqrt\eps (h_{eY}^i\sum_{j=0}^{i-1} \eps^{\frac{j}{2}}v_p^j-u_{eY}^i\sum_{j=0}^{i-1} \eps^{\frac{j}{2}}g_p^j)\right]\nonumber\\
  &~ +\eps^{\frac{i-1}{2}} \left[v_e^i\sum_{j=1}^{i-1}\eps^{\frac{j}{2}}h_{py}^j-g_e^i\sum_{j=1}^{i-1}\eps^{\frac{j}{2}}u_{py}^j
  +\sqrt\eps(h_{ex}^i \sum_{j=0}^{i-1} \eps^{\frac{j}{2}}u_p^j-u_{ex}^i \sum_{j=0}^{i-1} \eps^{\frac{j}{2}}h_p^j )\right],\\
\label{R^g,i_E}
&R_E^{g,i}
=\eps^{\frac{i}{2}} \left[u_e^i \sum_{j=0}^{i-1} \eps^{\frac{j}{2}}g_{px}^j-h_e^i \sum_{j=0}^{i-1} \eps^{\frac{j}{2}}v_{px}^j
  + g_{eY}^i \sum_{j=0}^{i-1} \eps^{\frac{j}{2}}v_p^j-v_{eY}^i \sum_{j=0}^{i-1} \eps^{\frac{j}{2}}g_p^j\right]\nonumber\\
  &\qquad +\eps^{\frac{i-1}{2}} \left[v_e^i\sum_{j=1}^{i-1}\eps^{\frac{j}{2}}g_{py}^j-g_e^i\sum_{j=1}^{i-1}\eps^{\frac{j}{2}}v_{py}^j
  +g_{ex}^i\sum_{j=1}^{i-1} \eps^{\frac{j}{2}}u_p^j-v_{ex}^i\sum_{j=1}^{i-1} \eps^{\frac{j}{2}}h_p^j\right].
\end{align}
Similar to the arguments in the $\sqrt{\eps}$-order correctors, we introduce the following $\eps^{\frac{i}{2}}$-order ($2\leq i\leq n$) auxiliary boundary layer pressure
\begin{align}\label{p_p^i,a}
\eps^{\frac{i+1}{2}}p_p^{i,a}=\eps\int_y^\infty (R^{v,i-1}+R_E^{v,i}).
\end{align}

See also \cite{Iyerglobal1} for more details about the derivation.

\subsection{The $\eps^{\frac{i}{2}}$-order ideal MHD correctors ($2\leq i\leq n$)}\label{sec4.2}
As mentioned in the previous sections, the $\eps^\frac{i}{2}$-order ideal MHD correctors $(u_e^i,v_e^i,h_e^i,g_e^i)$ satisfy the following system
\begin{equation}\label{u_e^i_system}
\begin{cases}
  \p_x u_e^i-\sigma\p_x h_e^i+\p_x p_e^i=0,\\
  \p_x v_e^i-\sigma\p_x g_e^i+\p_Y p_e^i=0,\\
  \p_x h_e^i-\sigma\p_x u_e^i=0,\\
  \p_x g_e^i-\sigma\p_x v_e^i=0,\\
  \p_x u_e^i+\p_Y v_e^i=\p_x h_e^i+\p_Y g_e^i=0,\\
  (v_e^i,g_e^i)(x,0)=-(v_p^{i-1},g_p^{i-1})(x,0),\\
  (v_e^i,g_e^i)\rightarrow (0,0){\rm ~as~} Y\rightarrow\infty.
\end{cases}
\end{equation}
It is noted that the $\eps^{\frac{i}{2}}$-order auxiliary ideal MHD pressure
$$  p_e^{i,a}=\frac{-(1-\sigma^2)}{2}(|u_e^i|^2+|v_e^i|^2)-(1-\sigma^2)\sum_{j=1}^{i-1}\eps^{\frac{j-i}{2}}(u_e^iu_e^j+v_e^iv_e^j)
$$
has been introduced to ensure that
$E_E^{u,i}+\eps^i p_{ex}^{i,a}=0, E_E^{v,i}+\eps^{i-\frac{1}{2}} p_{eY}^{i,a}=0.$

Observe that system \eqref{u_e^i_system} has the same structure to that for the $\sqrt{\eps}$-order ideal MHD corrector as in \eqref{u_e^1_system}, one can perform the similar arguments of Section \ref{sec4} (see also \cite{Iyerglobal1} for details) to obtain the decay properties for $(u_e^i,v_e^i,h_e^i,g_e^i)$. Without loss of generality, we take the pressure to be $p_e^i=-(1-\sigma^2)u_e^i$, and thus derive the Cauchy-Riemann equations for $(u_e^i,v_e^i)$:
\begin{equation}\label{u_e^i_CR}
  v_{ex}^i-u_{eY}^i=0,\quad u_{ex}^i+v_{eY}^i=0.
\end{equation}
In addition, the magnetic field can be constructed by $(h_e^i,g_e^i)=\sigma(u_e^i,v_e^i)$.

By estimate \eqref{v_p^1_lemma_1}, for $2\leq i\leq n$, the boundary term on $\{x=0\}$ enjoys the following decay estimate
\begin{align}\label{v_e^i_bc_decay}
  |v_e^i(x,0)|\leq C(\sigma_{i-1},n)x^{-\frac{3}{4}+\sigma_{i-1}}.
\end{align}
Compared with \eqref{u_e^1_bc_decay}, there is a better $x^{\frac{1}{4}-\sigma_{i-1}}$
in \eqref{v_e^i_bc_decay}, which is the key point to improve the decay rate for the intermediate ideal MHD correctors. Precisely, we have the following proposition.
\begin{proposition}\label{u_e^i_prop}
For $2\leq i\leq n$, let $(u_e^i,v_e^i)$ solve
\begin{equation}\label{u_e^ivhg_system}
  -\Delta v_e^i=0,\ v_e^i(x,0)=-v_p^{i-1}(x,0),\ v_e^i(x,\infty)=0,\ u_e^i=\int_x^\infty v_{eY}^i(\theta,Y){\rm d}\theta,
\end{equation}
and $(h_e^i,g_e^i)=\sigma(u_e^i,v_e^i)$, $p_e^i=-(1-\sigma^2)u_e^i$. Then $(u_e^i,v_e^i,h_e^i,g_e^i,p_e^i)$ solve system \eqref{u_e^i_system} with
\begin{align}\label{u_e^i_estimate1}
  &\Vert \p_x^k\p_Y^j(u_e^i,v_e^i,h_e^i,g_e^i)\Vert_{L_Y^\infty}\leq C(k,j,n)x^{-k-j-\frac{3}{4}+\sigma_{i-1}}.
\end{align}
In addition, we have
\begin{align}\label{u_e^i_estimate2}
  &\Vert \p_x^k(v_e^i,g_e^i)Y\Vert_{L_Y^\infty}\leq C(k,j,n)x^{-k-\frac{3}{4}+\sigma_{i-1}}.
\end{align}
\end{proposition}
\begin{proof}
The proof of this proposition is similar to that for the $\sqrt{\eps}$-order ideal MHD correctors and we omit it here.
\end{proof}

\subsection{The $\eps^\frac{i}{2}$-order boundary layer correctors ($2\leq i\leq n-1$)}\label{sec4.3}
The initial boundary value problem for the boundary layer correctors is taken as
\begin{equation}\label{u_p^i_system}
\begin{cases}
  -u_{pyy}^i+(1+u_p^0)u_{px}^i+p_{px}^i=(\sigma+h_p^0)h_{px}^i +\mathcal{P}_u^{(i)}+f_u^{(i)},\\
  -h_{pyy}^i+(1+u_p^0)h_{px}^i=(\sigma+h_p^0)u_{px}^i +\mathcal{P}_h^{(i)}+f_h^{(i)},\\
  p_{py}^i=0,\\
  u_{px}^i+v_{py}^i=h_{px}^i+g_{py}^i=0,\\
  (u_p^i,v_p^i,h_p^i,g_p^i)(x,0)=(-\overline u_e^i,-\overline v_e^{i+1},-\overline h_e^i,-\overline g_e^{i+1})(x),\\
  (u_p^i,v_p^i,h_p^i,g_p^i)(x,\infty)=(0,0,0,0),\\
   (u_p^i,h_p^i)(1,y)=(u_0^i,h_0^i)(y),
\end{cases}
\end{equation}
in which
\begin{equation}\label{P_ui}
\begin{cases}
  \mathcal{P}_u^{(i)}=&u_p^iu_{sx}^{(i)}+(v_p^i-\overline v_p^i)u_{py}^0+v_s^{(i)}u_{py}^i\\
  &-h_p^ih_{sx}^{(i)}-(g_p^i-\overline g_p^i)h_{py}^0-g_s^{(i)}h_{py}^i,\\
  \mathcal{P}_h^{(i)}=&u_p^ih_{sx}^{(i)}+(v_p^i-\overline v_p^i)h_{py}^0+v_s^{(i)}h_{py}^i\\
  &-h_p^iu_{sx}^{(i)}-(g_p^i-\overline g_p^i)u_{py}^0-g_s^{(i)}u_{py}^i,\\
  f_u^{(i)}=&-\eps^{-\frac{i}{2}}(R^{u,i-1}+R_E^{u,i}+\eps^{\frac{i+1}{2}} p_{px}^{i,a}),\\
  f_h^{(i)}=&-\eps^{-\frac{i}{2}}(R^{h,i-1}+R_E^{h,i}),
\end{cases}
\end{equation}
where the terms $R^{u,i-1},R^{h,i-1},R_E^{u,i},R_E^{h,i}, p_{p}^{i,a}$ are defined as in \eqref{R^u_i}, \eqref{R^h_i}, \eqref{R^u,i_E}, \eqref{R^h,i_E} and \eqref{p_p^i,a}. The source terms $(f_u^{(i)},f_h^{(i)})$ satisfy the following estimate (refer to \cite{Iyerglobal1} for the details)
\begin{equation}\label{f_ui_estimate}
  \Vert z^m \p_x^k(f_u^{(i)},f_h^{(i)})\Vert_{L_y^2}\leq C(k,m)x^{-k-\frac{5}{4}+2\sigma_{i-1}}.
\end{equation}

One shall infer $p_p^i=C$ by evaluating the equation \eqref{u_p^i_system}$_1$ at $y=\infty$ and using $p_{py}^i=0$. Here we prescribe $p_p^i=0$, without loss of generality. Thanks to the divergence free conditions and the behavior at $\infty$, we will construct the vertical components $(v_p^i,g_p^i)$ by
\begin{align}\label{v_p^i}
  (v_p^i,g_p^i)(x,y)=\int_y^\infty \p_x(u_p^i,h_p^i)(x,\theta){\rm d}\theta.
\end{align}

Now we are ready to give the well-posedness result for the intermediate boundary layer profiles with enhanced decay.
\begin{theorem}\label{u_p^i_theorem}
For any $m,k,j\in\mathbb{N}$ and $2\leq i\leq n-1$, there exists solution $(u_p^i,v_p^i,h_p^i,g_p^i)$ to the problem \eqref{u_p^i_system} on the domain $\Omega=[1,\infty)\times \mathbb{R}_+$ with
\begin{align}
\label{uh_p^i_estimate}
  &\Vert z^m\p_x^k\p_y^j (u_p^i,h_p^i)\Vert_{L_y^\infty}\leq C(\delta,\sigma,\sigma_i;m,k,j)x^{-k-\frac{j}{2}-\frac{1}{2}+\sigma_i},\\
\label{vg_p^i_estimate}
  &\Vert z^m\p_x^k\p_y^j (v_p^i,g_p^i)\Vert_{L_y^\infty}\leq C(\delta,\sigma,\sigma_i;m,k,j)x^{-k-\frac{j}{2}-1+\sigma_i}.
\end{align}
The constant $\sigma_i$ is prescribed by $\sigma_i=\frac{3^i}{3^n\times 1000}$.
\end{theorem}
\begin{remark}\label{remark_enhance}
The enhanced decay rates, in essence, come from the improved decay of the boundary condition as stated in \eqref{v_e^i_bc_decay}, will play key roles in the global-in-$x$ stability for the boundary layer expansion, see also Remark \ref{remark_n>2} and \ref{remark_n>2_S6}.
\end{remark}

\begin{proposition}\label{u_p^i_prop}
 Suppose that $(u_p^i,h_p^i)$ solve \eqref{u_p^i_system}, then for any $m,k\in\mathbb{N}$ and $2\leq i\leq n-1$, it holds that
\begin{align}
\label{u_p^i_estimate_P}
  &\Vert (u_p^i,h_p^i)x^{\frac{1}{4}}\Vert_{\mathcal{P}(\sigma_i)}\leq C(\delta,\sigma,\sigma_i),\\
\label{u_p^i_estimate_Pk_weighted}
  &\Vert z^m(u_p^i,h_p^i)x^{\frac{1}{4}}\Vert_{\mathcal{P}_k(\sigma_i)}\leq C(\delta,\sigma,\sigma_i;k,m),
\end{align}
where the norms $\mathcal{P}(\sigma_i)$,$\mathcal{P}_k(\sigma_i)$ are defined as in \eqref{u_p^1_norm},\eqref{u_p^1_norm_k} with $\sigma_1$ replaced by $\sigma_i$. The constant $\sigma_i$ is prescribed by $\sigma_i=\frac{3^i}{3^n\times 1000}$.
\end{proposition}
\begin{proof}
The estimate \eqref{u_p^i_estimate_P} will be obtained by the three steps. In what follows, we denote the unknowns $(u_p^i,v_p^i,h_p^i,g_p^i)$ as $(u,v,h,g)$ for convenience.

First of all, introducing the stream functions $(-\varphi_y,\varphi_x)=(u,v)$ and $(-\psi_y,\psi_x)=(h,g)$, one shall rewrite the system \eqref{u_p^i_system} as
\begin{equation}\label{u_p^i_system_streamfunction}
\left\{
\begin{aligned}
  &(\p_x-\p_y^2)\varphi=\sigma\p_x\psi-\int_y^\infty u_p^0u_x+\int_y^\infty h_p^0h_x-\int_y^\infty \mathcal{P}_u^{(i)}+\int_y^\infty f_u^{(i)},\\
  &(\p_x-\p_y^2)\psi=\sigma\p_x\varphi-\int_y^\infty u_p^0h_x+\int_y^\infty h_p^0u_x-\int_y^\infty \mathcal{P}_h^{(i)}+\int_y^\infty f_h^{(i)}.
\end{aligned}
\right.
\end{equation}

Then the boundary conditions and initial data are prescribed by
\begin{equation}\label{u_p^i_bc_streamfunction}
\left\{
\begin{aligned}
  &(\varphi_y,\psi_y)(x,0)=(\overline u_e^i,\overline h_e^i)(x),\quad (\varphi,\psi)|_{y\rightarrow\infty}=(0,0),\\
  &\varphi(1,y)=\int_y^\infty u_0^i(\theta){\rm d}\theta,\quad \psi(1,y)=\int_y^\infty h_0^i(\theta){\rm d}\theta.
\end{aligned}
\right.
\end{equation}

Testing \eqref{u_p^i_system_streamfunction}$_{1,2}$ by $(\varphi x^{-\frac{1}{2}-2\sigma_i},\psi x^{-\frac{1}{2}-2\sigma_i})$ respectively, we get
\begin{align}\label{u_p^i_estimate_1.0}
\begin{split}
  &\frac{1}{2}\frac{\rm d}{{\rm d}x}\int_0^\infty |(\varphi,\psi)|^2 x^{-\frac{1}{2}-2\sigma_i}{\rm d}y
  +\int_0^\infty |\p_y(\varphi,\psi)|^2 x^{-\frac{1}{2}-2\sigma_i}{\rm d}y\\
  & +(\frac{1}{4}+\sigma_i)\int_0^\infty |(\varphi,\psi)|^2 x^{-\frac{3}{2}-2\sigma_i-1}{\rm d}y\\
  =&\int_0^\infty \sigma\p_x(\varphi\psi)x^{-\frac{1}{2}-2\sigma_i}{\rm d}y
  +\left[\overline u_e^i(x)\varphi(x,0)+\overline h_e^i(x)\psi(x,0)\right]x^{-\frac{1}{2}-2\sigma_i}\\
  &+\left[\int_0^\infty \varphi x^{-\frac{1}{2}-2\sigma_i}\left(-\int_y^\infty u_{px}^0u_x+\int_y^\infty h_p^0h_x\right){\rm d}y\right.\\
  &\left.+\int_0^\infty \psi x^{-\frac{1}{2}-2\sigma_i}\left(-\int_y^\infty u_{px}^0h_x+\int_y^\infty h_p^0u_x\right){\rm d}y\right]\\
  &-\int_0^\infty \left(\varphi x^{-\frac{1}{2}-2\sigma_i}\int_y^\infty \mathcal{P}_u^{(i)}
  +\psi x^{-\frac{1}{2}-2\sigma_i}\int_y^\infty \mathcal{P}_h^{(i)}\right){\rm d}y\\
  &+\int_0^\infty \left(\varphi x^{-\frac{1}{2}-2\sigma_i}\int_y^\infty f_u^{(i)}
  +\psi x^{-\frac{1}{2}-2\sigma_i}\int_y^\infty f_h^{(i)}\right){\rm d}y\\
  :=&I_1+I_2+I_3+I_4+I_5.
\end{split}
\end{align}

The first term on the right hand side can be handled as
\begin{align}\label{u_p^i_estimate_1.1}
\begin{split}
I_1=&\frac{\rm d}{{\rm d}x}\int_0^\infty \sigma\varphi\psi x^{-\frac{1}{2}-2\sigma_i}{\rm d}y
  +(\frac{1}{2}+2\sigma_i)\int_0^\infty \sigma\varphi\psi x^{-\frac{3}{2}-2\sigma_i}{\rm d}y\\
  \leq& \frac{\rm d}{{\rm d}x}\int_0^\infty \sigma\varphi\psi x^{-\frac{1}{2}-2\sigma_i}{\rm d}y
  +(\frac{1}{4}+\sigma_i)\sigma\Vert (\varphi,\psi)x^{-\frac{3}{4}-\sigma_i}\Vert_{L_y^2}^2.
\end{split}
\end{align}

The boundary terms are bounded by
\begin{align}\label{u_p^i_estimate_1.2}
\begin{split}
I_2=&\left[\overline u_e^i(x)\varphi(x,0)+\overline h_e^i(x)\psi(x,0)\right]x^{-\frac{1}{2}-2\sigma_i}\\
  \leq& |(\varphi,\psi)(x,0)x^{-\frac{3}{4}+\sigma_{i-1}-\frac{1}{2}-2\sigma_i}|
  \leq \delta_0|(\varphi,\psi)(x,0)x^{-\frac{3}{4}-\sigma_i}|^2+C|x^{-\frac{1}{2}-\sigma'}|^2,
\end{split}
\end{align}
in which we have used $\sigma':=\sigma_i-\sigma_{i-1}>0$. Concerning the third term, we get
\begin{align}\label{u_p^i_estimate_1.3}
\begin{split}
I_3=&\int_0^\infty \varphi x^{-\frac{1}{2}-2\sigma_i}\left(-\int_y^\infty u_{px}^0u_x+\int_y^\infty h_p^0h_x\right){\rm d}y\\
  &+\int_0^\infty \psi x^{-\frac{1}{2}-2\sigma_i}\left(-\int_y^\infty u_{px}^0h_x+\int_y^\infty h_p^0u_x\right){\rm d}y\\
  \leq& \mathcal{O}(\delta,\sigma)\Vert (\varphi,\psi)x^{-\frac{3}{4}-\sigma_i}\Vert_{L_y^2}\Vert (\p_x u,\p_x h)x^{\frac{3}{4}-\sigma_i}\Vert_{L_y^2},
\end{split}
\end{align}
in which we have used the estimate
\begin{align}\label{u_p^i_estimate_1.31}
\begin{split}
  \left\Vert x^{\frac{1}{4}-\sigma_i}\int_y^\infty u_p^0u_x\right\Vert_{L_y^2}
  \leq& \Vert x^{\frac{1}{4}-\sigma_i}y u_p^0u_x\Vert_{L_y^2}
  \leq \Vert x^{\frac{1}{4}-\sigma_i}y u_p^0u_x\Vert_{L_y^2}\\
  \leq& \Vert x^{\frac{3}{4}-\sigma_i}u_x\Vert_{L_y^2}\Vert zu_p^0\Vert_{L_y^\infty}
  \leq \mathcal{O}(\delta,\sigma) \Vert x^{\frac{3}{4}-\sigma_i}u_x\Vert_{L_y^2}.
\end{split}
\end{align}

And the last two terms are bounded by
\begin{align}\label{u_p^i_estimate_1.4}
\begin{split}
I_4=&-\int_0^\infty \left(\varphi x^{-\frac{1}{2}-2\sigma_i}\int_y^\infty \mathcal{P}_u^{(i)}
  +\psi x^{-\frac{1}{2}-2\sigma_i}\int_y^\infty \mathcal{P}_h^{(i)}\right){\rm d}y\\
  \leq&\mathcal{O}(\delta,\sigma) \Vert x^{\frac{3}{4}-\sigma_i}(\p_yv,\p_yg)\Vert_{L_y^2}^2
  +\mathcal{O}(\delta,\sigma)\Vert (\varphi,\psi)x^{-\frac{3}{4}-\sigma_i}\Vert_{L_y^2}^2\\
  &+\mathcal{O}(\delta,\sigma)\Vert (\p_y\varphi,\p_y\psi)x^{-\frac{1}{4}-\sigma_i}\Vert_{L_y^2}^2,
\end{split}
\end{align}
and
\begin{align}\label{u_p^i_estimate_1.5}
\begin{split}
I_5=&\int_0^\infty \left(\varphi x^{-\frac{1}{2}-2\sigma_i}\int_y^\infty f_u^{(i)}
  +\psi x^{-\frac{1}{2}-2\sigma_i}\int_y^\infty f_h^{(i)}\right){\rm d}y\\
  \leq& \Vert (\varphi,\psi)x^{-\frac{3}{4}-\sigma_i}\Vert_{L_y^2}\cdot
  \left(\left\Vert x^{\frac{3}{4}-\sigma_i}z\frac{\int_y^\infty f_u^{(2)}}{y}\right\Vert_{L_y^2}
  +\left\Vert x^{\frac{3}{4}-\sigma_i}z\frac{\int_y^\infty f_h^{(2)}}{y}\right\Vert_{L_y^2}\right)\\
  \leq& \Vert (\varphi,\psi)x^{-\frac{3}{4}-\sigma_i}\Vert_{L_y^2}\cdot \Vert x^{\frac{3}{4}-\sigma_i}z(f_u^{(2)},f_h^{(2)})\Vert_{L_y^2}\\
  \leq& \delta_0\Vert (\varphi,\psi)x^{-\frac{3}{4}-\sigma_i}\Vert_{L_y^2}^2+Cx^{-1-\sigma'},
\end{split}
\end{align}
where we have used \eqref{f_ui_estimate} and $\sigma_i-2\sigma_{i-1}=\sigma_{i-1}$ in the last inequality.

Integrating with respect to $x$, and taking supremum, we conclude that
\begin{align}\label{u_p^i_estimate1}
\begin{split}
  &\sup_{x\geq 1}\int_0^\infty |(\varphi,\psi)|^2x^{-\frac{1}{2}-2\sigma_i}{\rm d}y
  +\int_1^\infty\int_0^\infty |(\varphi,\psi)|^2x^{-\frac{3}{2}-2\sigma_i}{\rm d}y{\rm d}x\\
  &+\int_1^\infty\int_0^\infty |(u,h)|^2x^{-\frac{1}{2}-2\sigma_i}{\rm d}y{\rm d}x\\
  \leq& C+\mathcal{O}(\delta,\sigma)\Vert (\p_xu,\p_xh)x^{\frac{3}{4}-\sigma_i}\Vert_{L_x^2L_y^2}^2.
\end{split}
\end{align}

Indeed, $I_1$ has been absorbed into the left hand side (due to the smallness of $\sigma$)
\begin{align}\label{u_p^i_estimate1.1}
  \int_1^x I_1{\rm d}x
  \leq \frac{\sigma}{2}\int_0^\infty |(\varphi,\psi)|^2x^{-\frac{1}{2}-2\sigma_i}{\rm d}y
  +\sigma\int_1^\infty\int_y^\infty |(\varphi,\psi)|^2x^{-\frac{3}{2}-2\sigma_i}{\rm d}y{\rm d}x.
\end{align}

Similarly, $I_2$ can be estimated by
\begin{align}\label{u_p^i_estimate1.2}
\begin{split}
  \int_1^x I_2{\rm d}x\leq \delta_0(\Vert(\varphi,\psi)x^{-\frac{3}{4}-\sigma_i}\Vert_{L_x^2L_y^2}^2
  +\Vert (u,h)x^{-\frac{3}{4}-\sigma_i}\Vert_{L_x^2L_y^2}^2)+C.
\end{split}
\end{align}

Second, we test \eqref{u_p^i_system}$_{1,2}$ by $(ux^{\frac{1}{2}-2\sigma_i},hx^{\frac{1}{2}-2\sigma_i})$ respectively and derive that
\begin{align}
  &\frac{1}{2}\frac{\rm d}{{\rm d}x}\int_0^\infty (1+u_p^0)|(u,h)|^2 x^{\frac{1}{2}-2\sigma_i}{\rm d}y
  +\int_0^\infty |(\p_y u,\p_y h)|^2 x^{\frac{1}{2}-2\sigma_i}{\rm d}y\nonumber\\
  =&\int_0^\infty (\sigma+h_p^0)\p_x(uh)x^{\frac{1}{2}-2\sigma_i}{\rm d}y+\left[\overline u_e^i(x)\p_yu(x,0)+\overline h_e^i(x)\p_yh(x,0)\right]x^{\frac{1}{2}-2\sigma_i}\nonumber\\
  &+\int_0^\infty \left[\frac{1}{2}u_{px}^0|(u,h)|^2 x^{\frac{1}{2}-2\sigma_i}
  -(\mathcal{P}_u^{(i)},\mathcal{P}_h^{(i)})\cdot(u,h) x^{-\frac{1}{2}-2\sigma_i}\right]{\rm d}y\nonumber\\
  &+\int_0^\infty (f_u^{(i)},f_h^{(i)})\cdot(u,h) x^{-\frac{1}{2}-2\sigma_i}{\rm d}y\nonumber\\
  &
  +(\frac{1}{4}-\sigma_i)\int_0^\infty (1+u_p^0)|(u,h)|^2 x^{-\frac{1}{2}-2\sigma_i}{\rm d}y
  :=J_1+J_2+J_3+J_4+J_5.\label{u_p^i_estimate_2.0}
\end{align}
All terms on the right hand side are estimated as follows
\begin{align}
\label{u_p^i_estimate_2.1}
J_1=&\frac{\rm d}{{\rm d}x}\int_0^\infty (\sigma+h_p^0)hux^{\frac{1}{2}-2\sigma_i}{\rm d}y-\int_0^\infty h_{px}^0hux^{\frac{1}{2}-2\sigma_i}{\rm d}y\nonumber\\
  &-(\frac{1}{2}-2\sigma_i)\int_0^\infty (\sigma+h_p^0)hux^{-\frac{1}{2}-2\sigma_i}{\rm d}y\nonumber\\
  \leq& \frac{\rm d}{{\rm d}x}\int_0^\infty (\sigma+h_p^0)hux^{\frac{1}{2}-2\sigma_i}{\rm d}y
  +\Vert h_{px}^0x\Vert_{L_y^\infty}\Vert (u,h)x^{-\frac{1}{4}-\sigma_i}\Vert_{L_y^2}^2\\
  &+(\frac{1}{2}-2\sigma_i)\Vert \sigma +h_p^0\Vert_{L_y^\infty}\Vert (u,h)x^{-\frac{1}{4}-\sigma_i}\Vert_{L_y^2}^2,\nonumber\\
\label{u_p^i_estimate_2.2}
J_2\leq& \delta_0|(\p_y u,\p_y h)(x,0)x^{\frac{1}{4}-\sigma_i}|^2+C|x^{-\frac{1}{2}-\sigma'}|^2,\\
\label{u_p^i_estimate_2.3}
J_3\leq& \mathcal{O}(\delta,\sigma)\Vert (\p_y v,\p_y g)x^{\frac{3}{4}-\sigma_i}\Vert_{L_y^2}^2
  +\mathcal{O}(\delta,\sigma)\Vert (u,h)x^{-\frac{1}{4}-\sigma_i}\Vert_{L_y^2}^2,\\
\label{u_p^i_estimate_2.5}
J_4+J_5&\leq \delta_0\|(u,h)x^{-\frac{1}{4}-\sigma_i}\|_{L^2_y}^2+Cx^{-1-\sigma'}+C\|(u,h)x^{-\frac{1}{4}-\sigma_i}\|_{L^2_y}^2.
 \end{align}

Combining the above estimates, we have
\begin{align}\label{u_p^i_estimate2}
\begin{split}
  &\sup_{x\geq 1}\int_0^\infty |(u,h)|^2x^{\frac{1}{2}-2\sigma_i}{\rm d}y
  +\int_1^\infty\int_0^\infty |(\p_y u,\p_y h)|^2x^{\frac{1}{2}-2\sigma_i}{\rm d}y{\rm d}x\\
  \leq& C+\Vert (u,h)x^{-\frac{1}{4}-\sigma_i}\Vert_{L_y^2}^2+\mathcal{O}(\delta,\sigma)\Vert (u_x,h_x)x^{\frac{1}{4}-\sigma_i}\Vert_{L_x^2L_y^2}^2\\
  & +\delta_0\Vert (\p_y^2u,\p_y^2h)x^{\frac{3}{4}-\sigma_i}\Vert_{L_x^2L_y^2}^2.
\end{split}
\end{align}

Third, testing \eqref{u_p^i_system}$_{1,2}$ by $(u_xx^{\frac{3}{2}-2\sigma_i},h_xx^{\frac{3}{2}-2\sigma_i})$ respectively, we have
\begin{align}
  &\frac{1}{2}\frac{\rm d}{{\rm d}x}\int_0^\infty |(\p_y u,\p_y h)|^2 x^{\frac{3}{2}-2\sigma_i}{\rm d}y
  +\int_0^\infty (1+u_p^0)|(u_x,h_x)|^2 x^{\frac{3}{2}-2\sigma_i}{\rm d}y\nonumber\\
  =&2\int_0^\infty (\sigma+h_p^0)u_x h_xx^{\frac{3}{2}-2\sigma_i}{\rm d}y
  +\left[\overline u_{ex}^i(x)\p_yu(x,0)+\overline h_{ex}^i(x)\p_yh(x,0)\right]x^{\frac{3}{2}-2\sigma_i}\nonumber\\
  &-\int_0^\infty (\mathcal{P}_u^{(i)},\mathcal{P}_h^{(i)})\cdot(u_x,h_x) x^{-\frac{3}{2}-2\sigma_i}{\rm d}y
+\int_0^\infty (f_u^{(i)},f_h^{(i)})\cdot(u_x,h_x) x^{\frac{3}{2}-2\sigma_i}{\rm d}y\nonumber\\
&+(\frac{3}{4}-\sigma_i)\int_0^\infty |(\p_y u,\p_y h)|^2 x^{\frac{1}{2}-2\sigma_i}{\rm d}y
  :=K_1+K_2+K_3+K_4+K_5.\label{u_p^i_estimate_3.0}
\end{align}

Similar to that in the second step, we obtain
\begin{align}
\label{u_p^i_estimate_3.1}
K_1\leq&\mathcal{O}(\delta,\sigma)\Vert (u_x,h_x)x^{\frac{3}{4}-\sigma_i}\Vert_{L_y^2}^2,\\
\label{u_p^i_estimate_3.2}
K_2\leq& \delta_0|(\p_y u,\p_y h)(x,0)x^{\frac{1}{4}-\sigma_i}|^2+C|x^{-1-\sigma'}|^2,\\
\label{u_p^i_estimate_3.3}
K_3\leq& \mathcal{O}(\delta,\sigma)\Vert (u_x,h_x)x^{\frac{3}{4}-\sigma_i}\Vert_{L_y^2}^2
  +\mathcal{O}(\delta,\sigma)\Vert (\p_y u,\p_y h)x^{\frac{1}{4}-\sigma_i}\Vert_{L_y^2}^2\nonumber\\
  & +\mathcal{O}(\delta,\sigma)\Vert (u,h)x^{-\frac{1}{4}-\sigma_i}\Vert_{L_y^2}^2,\\
\label{u_p^i_estimate_3.5}
K_4+&K_5\leq \delta_0\|(u_x,h_x)x^{\frac{3}{4}-\sigma_i}\|_{L^2_y}^2+Cx^{-1-\sigma'}+C\|(\p_y u,\p_y h)x^{\frac{1}{4}-\sigma_i}\|_{L^2_y}^2.
\end{align}

Plugging the estimates \eqref{u_p^i_estimate_3.1}-\eqref{u_p^i_estimate_3.5} to \eqref{u_p^i_estimate_3.0}, and integrating the result on $[1,x)$,
we deduce that
\begin{align}\label{u_p^i_estimate3}
\begin{split}
  &\sup_{x\geq 1}\int_0^\infty |(u_y,h_y)|^2x^{\frac{3}{2}-2\sigma_i}{\rm d}y
  +\int_1^\infty\int_0^\infty |(u_x,h_x)|^2x^{\frac{3}{2}-2\sigma_i}{\rm d}y{\rm d}x\\
  \leq& C+\Vert (\p_y u,\p_y h)x^{\frac{1}{4}-\sigma_i}\Vert_{L_x^2L_y^2}^2+\mathcal{O}(\delta,\sigma)\Vert (u,h)x^{-\frac{1}{4}-\sigma_i}\Vert_{L_x^2L_y^2}^2\\
  & +\delta_0\Vert (\p_y^2u,\p_y^2h)x^{\frac{3}{4}-\sigma_i}\Vert_{L_x^2L_y^2}^2.
\end{split}
\end{align}

Finally, due to the equations \eqref{u_p^i_system}, one has
\begin{align}\label{u_p^i_estimate4}
\begin{split}
  &\Vert (\p_y^2u,\p_y^2h)x^{\frac{3}{4}-\sigma_i}\Vert_{L_x^2L_y^2}^2\\
  \leq& \left(\Vert 1+u_p^0\Vert_{L_y^\infty}+\Vert \sigma+h_p^0\Vert_{L_y^\infty}\right)\Vert (\p_xu,\p_xh)x^{\frac{3}{4}-\sigma_i}\Vert_{L_x^2L_y^2}^2+C+\mathcal{O}(\delta,\sigma)\\
  &\cdot\left[\Vert (u,h)x^{-\frac{1}{4}-\sigma_i}\Vert_{L_x^2L_y^2}^2+\Vert (u_y,h_y)x^{\frac{1}{4}-\sigma_i}\Vert_{L_x^2L_y^2}^2+
  \Vert (u_x,h_x)x^{\frac{1}{4}-\sigma_i}\Vert_{L_x^2L_y^2}^2\right].
\end{split}
\end{align}

Therefore, the estimate \eqref{u_p^i_estimate_P} follows from summing up the estimates \eqref{u_p^i_estimate1}, \eqref{u_p^i_estimate2}, \eqref{u_p^i_estimate3} and \eqref{u_p^i_estimate4}. And similarly, we can derive the estimate \eqref{u_p^i_estimate_Pk_weighted} as well, because the above calculations with any positive weight $z^{2m}$ and higher-order derivatives do not bring extra difficulty, here we omit details.
\end{proof}

\subsection{The final boundary layer $(u_p^n,v_p^n,h_p^n,g_p^n)$}\label{sec4.3}
There are some differences between the final boundary layer and the intermediate layers, which are resulted from the zero boundary conditions for $(v_p^n,g_p^n)$ on $\{y=0\}$ compared with \eqref{u_p^i_system}$_6$.

The initial boundary value problem for $\eps^{\frac{n}{2}}$-order boundary layer corrector reads as
\begin{equation}\label{u_p^n_system}
\begin{cases}
  -u^n_{pyy}+(1+u_p^0)u^n_{px}+p_{px}^n=(\sigma+h_p^0)h^n_{px}-\mathcal{P}_u^{(n)}+f_u^{(n)},\\
  -h^n_{pyy}+(1+u_p^0)h^n_{px}=(\sigma+h_p^0)u^n_{px}-\mathcal{P}_h^{(n)}+f_h^{(n)},\\
  p_{py}^n=0,\\
  u^n_{px}+v^n_{py}=h^n_{px}+g^n_{py}=0,\\
  (u^n_p,v^n_p,h^n_p,g^n_p )(x,0)=(-\overline u_e^n(x),0,-\overline h_e^n(x),0),\\
  (u^n_p,h^n_p)(x,\infty)=(0,0),\quad (u^n_p,h^n_p)(1,y)=(u_0^n,h_0^n)(y),
\end{cases}
\end{equation}
in which
\begin{equation}\label{P_un}
\begin{cases}
  \mathcal{P}_u^{(n)}=u^n_pu_{sx}^{(n)}+v^n_pu_{py}^0+v_s^{(n)}u^n_{py}-h^n_ph_{sx}^{(n)}-g^n_ph_{py}^0-g_s^{(n)}h_{py}^n,\\
  \mathcal{P}_h^{(n)}=u^n_ph_{sx}^{(n)}+v^n_ph_{py}^0+v_s^{(n)}h^n_{py}-h_p^2u_{sx}^{(n)}-g^n_pu_{py}^0-g_s^{(n)}u^n_{py},\\
  f_u^{(n)}=-\eps^{-\frac{n}{2}}(R^{u,n-1}+R_E^{u,n}+\eps^{\frac{n+1}{2}} p_{px}^{n,a}),\\
  f_h^{(n)}=-\eps^{-\frac{n}{2}}(R^{h,n-1}+R_E^{h,n}).
\end{cases}
\end{equation}

Evaluating the equation \eqref{u_p^1_system}$_1$ at $y=\infty$, and using the identity \eqref{u_p^n_system}$_3$, one can deduce that $p_p^n=C$. Without loss of generality, we take $p_p^n=0$. By the divergence free conditions and the boundary conditions \eqref{u_p^n_system}$_5$ on $\{y=0\}$, we define the vertical components $(v^n_p,g^n_p)$ by
\begin{align}\label{v_p}
  (v^n_p,g^n_p)(x,y)=-\int_0^y \p_x(u^n_p,h^n_p)(x,\theta){\rm d}\theta.
\end{align}

The final boundary layer correctors $(u_p^n,v_p^n,h_p^n,g_p^n)$ will be defined by cutting off the profiles $(u_p,v_p,h_p,g_p)$, for the sake of preserving the vanishing boundary conditions as $y\rightarrow \infty$.
And the remainder terms are expressed as follows
\begin{align}
\label{R^u_n}
  R^{u,n}
  =&\eps^{\frac{n}{2}}[-\eps u_{pxx}^n+ u_{px}^n\sum_{j=1}^n \eps^{\frac{j}{2}}(u_e^j+u_p^j)- h_{px}^n \sum_{j=1}^n \eps^{\frac{j}{2}}(h_e^j+h_p^j)\nonumber\\
  &+v_p^n \sum_{j=1}^n \eps^{\frac{j}{2}}(u_{py}^j+\sqrt\eps u_{eY}^j)-g_p^n \sum_{j=1}^n \eps^{\frac{j}{2}}(h_{py}^j+\sqrt\eps h_{eY}^j)
  +\mathcal{E}_u^{(n)}],\\
\label{R^v_n}
  R^{v,n}
  =&\eps^{\frac{n}{2}}[-\Delta_\eps v_p^n+\overline u_s^{(n)}v_{px}^n+u_p^nv_{sx}^{(n)}+v_s^{(n)}v_{py}^n+v_p^n\overline v_{sy}^{(n)}\nonumber\\
  &\qquad -\overline h_s^{(n)}g_{px}^n-h_p^ng_{sx}^{(n)}-g_s^{(n)}g_{py}^n-g_p^n\overline g_{sy}^{(n)}],\\
\label{R^h_n}
  R^{h,n}
  =&\eps^{\frac{n}{2}}[-\eps h_{pxx}^n+ h_{px}^n \sum_{j=1}^n \eps^{\frac{j}{2}}(u_e^j+u_p^j)- u_{px}^n \sum_{j=1}^n \eps^{\frac{j}{2}}(h_e^j+h_p^j)\nonumber\\
  &+v_p^n \sum_{j=1}^n \eps^{\frac{j}{2}}(h_{py}^j+\sqrt\eps h_{eY}^j)-g_p^n \sum_{j=1}^n \eps^{\frac{j}{2}}(u_{py}^j+\sqrt\eps u_{eY}^j)
  +\mathcal{E}_h^{(n)}],\\
\label{R^g_n}
  R^{g,n}
  =&\eps^{\frac{n}{2}}[-\Delta_\eps g_p^n+\overline u_s^{(n)}g_{px}^n+u_p^ng_{sx}^{(n)}+v_s^{(n)}g_{py}^n+v_p^n\overline g_{sy}^{(n)}\nonumber\\
  &\qquad -\overline h_s^{(n)}v_{px}^n-h_p^nv_{sx}^{(n)}-g_s^{(n)}v_{py}^n-g_p^n\overline v_{sy}^{(n)}]+R^{g,n-1}+R_E^{g,n},
\end{align}
where the error terms $\mathcal{E}_u^{(n)},\mathcal{E}_h^{(n)}$ will be prescribed later in \eqref{eps_u^n}-\eqref{eps_h^n}, which are created by the cut-off arguments, and $R^{g,n-1}$ is defined as in \eqref{R^g_i} for taking $i=n-1$, $R_E^{g,n}$ is expressed by
\begin{align}\label{R^g,n_E}
&R_E^{g,n}
=\eps^{\frac{n}{2}} \left[u_e^n \sum_{j=0}^{n-1} \eps^{\frac{j}{2}}g_{px}^j-h_e^n \sum_{j=0}^{n-1} \eps^{\frac{j}{2}}v_{px}^j
  + g_{eY}^n \sum_{j=0}^{n-1} \eps^{\frac{j}{2}}v_p^j-v_{eY}^n \sum_{j=0}^{n-1} \eps^{\frac{j}{2}}g_p^j\right]\nonumber\\
  &\qquad +\eps^{\frac{n-1}{2}} \left[v_e^n\sum_{j=1}^{n-1}\eps^{\frac{j}{2}}g_{py}^j-g_e^n\sum_{j=1}^{n-1}\eps^{\frac{j}{2}}v_{py}^j
  +g_{ex}^n\sum_{j=0}^{n-1} \eps^{\frac{j}{2}}u_p^j-v_{ex}^n\sum_{j=0}^{n-1} \eps^{\frac{j}{2}}h_p^j\right].
\end{align}

Performing similar argument as that for the intermediate layers, we can get
\begin{proposition}\label{u_p_prop}
Suppose that $(u^n_p,h^n_p)$ solve \eqref{u_p^n_system}, then for any $m,k\in\mathbb{N}$ and $n\geq 2$, there holds that
\begin{align}
\label{u_p_estimate_P}
  &\Vert (u^n_p,h^n_p)x^{\frac{1}{4}}\Vert_{\mathcal{P}(\sigma_n)}\leq C(\delta,\sigma,\sigma_i),\\
\label{u_p_estimate_Pk_weighted}
  &\Vert z^m(u^n_p,h^n_p)x^{\frac{1}{4}}\Vert_{\mathcal{P}_k(\sigma_n)}\leq C(\delta,\sigma,\sigma_i;k,m),
\end{align}
where the norms $\mathcal{P}(\sigma_n)$,$\mathcal{P}_k(\sigma_n)$ are defined as in \eqref{u_p^1_norm},\eqref{u_p^1_norm_k} with $\sigma_1$ replaced by $\sigma_n$. The sufficiently small constant $\sigma_n$ is prescribed by $\sigma_n=\frac{1}{1000}$.
\end{proposition}

With Proposition \ref{u_p_prop} at hands, one can also obtain the following theorem.
\begin{theorem}\label{u_p_theorem}
For any $m,k,j\in\mathbb{N}$, there exists solution $(u^n_p,v^n_p,h^n_p,g^n_p)$ to the linear boundary layer problem \eqref{u_p^n_system} in the domain $\Omega=[1,\infty)\times \mathbb{R}_+$, and the solution satisfies the estimates
\begin{align}
\label{uh_p_estimate}
  &x^{\frac{1}{2}}\Vert z^m\p_x^k\p_y^j (u^n_p,h^n_p)\Vert_{L_y^\infty}+x^{\frac{1}{4}}\Vert z^m\p_x^k\p_y^j (u^n_p,h^n_p)\Vert_{L_y^2}\nonumber\\
  &\qquad\qquad\qquad \leq C(\delta,\sigma,\sigma_i;m,k,j)x^{-k-\frac{j}{2}+\sigma_n},\\
\label{vg_p_estimate}
  &x^{\frac{1}{4}}\Vert \p_x^k(v^n_p,g^n_p)\Vert_{L_y^\infty}\leq C(\delta,\sigma,\sigma_i;m,k,j)x^{-k-\frac{j}{2}-\frac{3}{4}+\sigma_n}.
\end{align}
\end{theorem}
\begin{proof}
The first estimate is obtained by the similar arguments as in the proof of Theorem \ref{u_p^1_theorem}. For the second estimate, since $(v^n_p,g^n_p)$ vanishes on $\{y=0\}$, for $k=0$ and any small $\kappa>0$, we have
\begin{align}\label{v_p_estimate1}
\begin{split}
  &|v^n_{p}|^2=-\int_0^y \p_y|v^n_p|^2\leq \int_0^\infty 2|v^n_p v^n_{py}|
  \leq 2\left\Vert v^n_{p}y^{-\frac{1}{2}-\kappa}\right\Vert_{L_y^2} \Vert v^n_{py}y^{\frac{1}{2}+\kappa}\Vert_{L_y^2}\\
  =& 2\left\Vert y^{-\frac{1}{2}-\kappa}\int_0^yv^n_{py}\right\Vert_{L_y^2} \Vert v^n_{py}y^{\frac{1}{2}+\kappa}\Vert_{L_y^2}
  \lesssim \Vert v^n_{py}y^{\frac{1}{2}-\kappa}\Vert_{L_y^2}\Vert v^n_{py}y^{\frac{1}{2}+\kappa}\Vert_{L_y^2}\\
  =& \Vert v^n_{py}z^{\frac{1}{2}-\kappa}x^{\frac{1}{4}-\frac{\kappa}{2}}\Vert_{L_y^2}
  \Vert v^n_{py}z^{\frac{1}{2}+\kappa}x^{\frac{1}{4}+\frac{\kappa}{2}}\Vert_{L_y^2}
  \lesssim x^{\frac{1}{2}}\cdot x^{-1-\frac{1}{4}+\sigma_n}\cdot x^{-1-\frac{1}{4}+\sigma_n}\\
  \leq& C(\delta,\sigma,\sigma_1;m,k)x^{-2+2\sigma_n},
\end{split}
\end{align}
where we have used Lemma \ref{Hardy-type-0} with $p=2,\alpha=1-2\kappa$. Similarly, the case of $k\geq 1$ can be also deduced.
\end{proof}

\begin{remark}\label{u_p_remark}
The weighted estimates are not valid for $(v^n_p,g^n_p)$, due to the limitation $\alpha<p-1$ in the Hardy inequality \eqref{Hardy2}. Compared with profiles $(v_p^i,g_p^i)$, since $(v_p^i,g_p^i)$ vanish as $y\to \infty$, the Hardy inequality with $\alpha>p-1$ is valid and thus the weighted estimates can be derived.
\end{remark}

In what follows, we will construct the final boundary layer via cutting off the profiles $(u^n_p,v^n_p,h^n_p,g^n_p)$. Introduce the cut-off function $\chi(\sqrt\eps z)$ supported in $[0,1]$ and define
\begin{align}\label{u_p^n_def}
\begin{split}
  \tilde{v_p^n}:=\chi(\sqrt\eps z)v^n_p,\  \tilde{u_p^n}:=\int_x^\infty \p_y\left[\chi(\sqrt\eps\frac{y}{\sqrt\theta})v^n_p(\theta,y))\right]{\rm d}\theta,
  \end{split}
\end{align}
  \begin{align}\label{h_p^n_def}
\begin{split}
  \tilde{g_p^n}:=\chi(\sqrt\eps z)g^n_p,\  \tilde{h_p^n}:=\int_x^\infty \p_y\left[\chi(\sqrt\eps\frac{y}{\sqrt\theta})g^n_p(\theta,y))\right]{\rm d}\theta.
\end{split}
\end{align}
It is clear that $(\tilde{u_p^n},\tilde{v_p^n},\tilde{h_p^n},\tilde{g_p^n})$ satisfy divergence-free conditions
\begin{equation*}
  \p_x \tilde{u_p^n}+\p_y \tilde{v_p^n}=0,\ \p_x \tilde{h_p^n}+\p_y\tilde{ g_p^n}=0,
\end{equation*}
and the following boundary conditions
\begin{align}\label{u_p^n_bc}
&  (\tilde{u_p^n},\tilde{v_p^n},\tilde{h_p^n},\tilde{g_p^n})|_{y\rightarrow\infty}=(0,0,0,0),\\
& (\tilde{u_p^n},\tilde{v_p^n},\tilde{h_p^n},\tilde{g_p^n})|_{y=0}=(-\overline u_e^n(x),0,-\overline h_e^n(x),0).
\end{align}

By \eqref{u_p^n_def}, one would derive the same decay rates as those in Theorem \ref{u_p_theorem}. First, we derive the following $L^2$ estimates.

\begin{theorem}\label{u_p^n_theorem_L^2}
For $(\tilde{u}_p^n,\tilde{h}_p^n)$, it holds that
\begin{align}
\label{uh_p^n_estimateL2}
  &\Vert \p_x^k\p_y^j (\tilde u_p^n,\tilde h_p^n)\Vert_{L_y^2}\leq C(\delta,\sigma,\sigma_i;k,j)x^{-k-\frac{j}{2}-\frac{1}{4}+\sigma_n}.
\end{align}
\end{theorem}
\begin{proof}
Let us first recall that the cut off function is supported in $z\leq\frac{1}{\sqrt\eps}$ and thus satisfies
\begin{equation}\label{u_p^n_tradeoff}
  \textit{o}(\chi)|\eps^{\frac{1}{4}+\frac{\kappa}{2}}\langle y \rangle^{\frac{1}{2}+\kappa}x^{-\frac{1}{4}-\frac{\kappa}{2}}|\lesssim 1,
\end{equation}
where the natation $\textit{o}(\chi)$ is used to denote $\chi$ or any derivative of $\chi$. With this, and using the estimate \eqref{uh_p_estimate}, we have
\begin{align}\label{u_p^n_L2_1}
\begin{split}
  \Vert \tilde h_p^n\Vert_{L_y^2}
  \leq& \int_x^\infty \left( \left\Vert\sqrt\eps\theta^{-\frac{1}{2}} y\chi'\cdot\frac{ g_p^n}{y}\right\Vert_{L_y^2}\right){\rm d}\theta
  +\int_x^\infty \Vert \chi g_{py}^n\Vert_{L_y^2}{\rm d}\theta\\
  \lesssim& \int_x^\infty \Vert g_{py}^n\Vert_{L_y^2} {\rm d}\theta
  \lesssim  x^{-\frac{1}{4}+\sigma_n}.
\end{split}
\end{align}

Taking $x$-derivative, one gets
\begin{align}\label{u_p^n_L2_2}
\begin{split}
  \Vert \tilde h_{px}^n\Vert_{L_y^2}\leq \left\Vert\sqrt\eps x^{-\frac{1}{2}} y\chi'\cdot\frac{g_p^n}{y}\right\Vert_{L_y^2}+\Vert\chi g_{py}^n\Vert_{L_y^2}
  \lesssim \Vert g_{py}^n\Vert_{L_y^2} \lesssim x^{-\frac{3}{2}+\sigma_n}.
\end{split}
\end{align}

For the $y$-derivative, we have
\begin{align}\label{u_p^n_L2_3}
\begin{split}
  \Vert \tilde h_{py}^n\Vert_{L_y^2}
  \leq& \int_x^\infty \left(\left\Vert \eps\theta^{-1} y\chi''\cdot\frac{g_p^n}{y}\right\Vert_{L_y^2}
  +\left\Vert\frac{\sqrt\eps}{\sqrt\theta}\chi'g_{py}^n\right\Vert_{L_y^2}+\Vert \chi g_{pyy}^n\Vert_{L_y^2}\right){\rm d}\theta\\
  \lesssim& \int_x^\infty \left(\theta^{-\frac{1}{2}}\Vert g_{py}^n\Vert_{L_y^2}+\Vert g_{pyy}^n\Vert_{L_y^2}\right) {\rm d}\theta
  \lesssim  x^{-\frac{3}{4}+\sigma_n}.
\end{split}
\end{align}

In addition, the above calculations can be also applied to derive the estimates of higher derivatives for the velocity field and we omit the details here.
\end{proof}
Moreover, since the final boundary layer are of $\eps^{\frac{n}{2}}$-order ($n\geq 2$), it ensures us to pay factors of $\eps$ to absorb the weights of $z$. Precisely, it is obtained by the definition that the cut-off function $\chi(\sqrt\eps z)$ is supported in $\{z\leq\frac{1}{\sqrt\eps}\}$. Then, one shall obtain the following weighted $L^\infty$ estimates.
\begin{theorem}\label{u_p^n_theorem_weighted}
For $(\tilde u_p^n,\tilde v_p^n,\tilde h_p^n,\tilde g_p^n)$, there holds that
\begin{align}
\label{uh_p^n_estimate_weighted}
  &\Vert z^m\p_x^k\p_y^j (\tilde u_p^n,\tilde h_p^n)\Vert_{L_y^\infty}\leq \eps^{-\frac{m}{2}}C(\delta,\sigma,\sigma_i;m,k,j)x^{-k-\frac{j}{2}-\frac{1}{2}+\sigma_n},\\
\label{vg_p^n_estimate_weighted}
  &\Vert z^m\p_x^k\p_y^j (\tilde v_p^n,\tilde g_p^n)\Vert_{L_y^\infty}\leq \eps^{-\frac{m}{2}}C(\delta,\sigma,\sigma_i;m,k,j)x^{-k-\frac{j}{2}-1+\sigma_n}.
\end{align}
\end{theorem}

\begin{proof}
We begin with the estimate for $\tilde g_p^n$. By the definition, it is direct to infer that
\begin{align}\label{v_p^n_estimate1}
\begin{split}
  z^m|\tilde g_p^n|\leq z^m|\chi g_{p}^n|\lesssim \eps^{-\frac{m}{2}}x^{-1+\sigma_n}.
\end{split}
\end{align}

Taking derivatives, we get
\begin{align}
\label{v_p^n_estimate2}
  z^m|\p_x\tilde g_p^n|\leq z^m|\chi'\frac{\sqrt\eps}{x}zg_p^n|+z^m|\chi g_{px}^n|\lesssim \eps^{-\frac{m}{2}}x^{-2+\sigma_n},\\
\label{v_p^n_estimate3}
  z^m|\p_y\tilde g_p^n|\leq z^m|\chi'\frac{\sqrt\eps}{\sqrt x}g_p^n|+z^m|\chi g_{py}^n|\lesssim \eps^{-\frac{m}{2}}x^{-\frac{3}{2}+\sigma_n}.
\end{align}

For the tangential components, we have
\begin{align}\label{u_p^n_estimate1}
\begin{split}
  z^m|\tilde h_p^n|\leq& |o(\chi)z^m|\int_x^\infty \left(\frac{\sqrt\eps}{\sqrt\theta}|g_p^n|+|g_{py}^n|\right){\rm d}\theta\\
  \lesssim& \eps^{-\frac{m}{2}}\int_x^\infty \theta^{-\frac{3}{2}+\sigma_n}{\rm d}\theta\lesssim \eps^{-\frac{m}{2}}x^{-\frac{1}{2}+\sigma_n}.
\end{split}
\end{align}

And the term with derivative can be bounded by
\begin{align}\label{u_p^n_estimate2}
\begin{split}
  z^m|\tilde h_{py}^n|\leq& z^m\int_x^\infty \left|\frac{\eps}{\theta}\chi''g_p^n+\frac{\sqrt\eps}{\sqrt\theta}\chi'g_{py}^n+\chi g_{pyy}^n\right|{\rm d}\theta\\
  \lesssim& |o(\chi)z^m|\int_x^\infty \theta^{-2+\sigma_n}{\rm d}\theta\lesssim \eps^{-\frac{m}{2}}x^{-1+\sigma_n}.
\end{split}
\end{align}

The above calculations can be applied for higher derivatives and the velocity field as well. Here we omit the details.
\end{proof}

\subsection{The estimates of the error terms}\label{sec5.2}
The main goal of this subsection is to derive the estimates for the remainder terms $(R^{u,n},R^{v,n},R^{h,n},R^{g,n})$ defined as in \eqref{R^u_n}-\eqref{R^g_n}, which is stated as follows.
\begin{theorem}\label{R^un_theorem}
For any $k\geq 0,\gamma\in [0,\frac{1}{4})$, sufficiently small constant $\kappa>0$, and $n\geq 2$, $\sigma_n$ defined as in Proposition \ref{u_p_prop}, then the error terms $R^{u,n},R^{v,n},R^{h,n},R^{g,n}$ satisfy
\begin{align}
\label{R^unterm_estimate}
  &\eps^{-\frac{n}{2}-\gamma}|(\p_x^k R^{u,n},\p_x^k R^{h,n})|\leq C(n,\kappa)\eps^{\frac{1}{4}-\gamma-\kappa}x^{-k-\frac{3}{2}+2\sigma_n},\\
\label{R^vnterm_estimate}
  &\eps^{-\frac{n}{2}-\gamma}|\sqrt\eps(\p_x^k R^{v,n},\p_x^k R^{g,n})|\leq C(n,\kappa)\eps^{\frac{1}{4}-\gamma-\kappa}x^{-k-\frac{3}{2}+2\sigma_n},\\
  &\eps^{-\frac{n}{2}-\gamma}\Vert (\p_x^k R^{u,n},\sqrt\eps\p_x^k R^{v,n},\p_x^k R^{h,n},\sqrt\eps\p_x^k R^{g,n})\Vert_{L_y^2}\nonumber\\
\label{R^unterm_estimate_L2}
  &\qquad\qquad \leq C(n,\kappa)\eps^{\frac{1}{4}-\gamma-\kappa}x^{-k-\frac{5}{4}+2\sigma_n+\kappa}.
\end{align}
\end{theorem}

To prove Theorem \ref{R^un_theorem}, let us first specify the error terms $\mathcal{E}_u^{(n)},\mathcal{E}_h^{(n)}$ in the expressions of \eqref{R^u_n} and \eqref{R^h_n} respectively, which are created by the cutting off procedures. Thanks to the definition of $(u_p^n,v_p^n,h_p^n,g_p^n)$ in \eqref{u_p^n_def}, we have
\begin{align}\label{eps_u^n}
\begin{split}
  \mathcal{E}_u^{(n)}
  =&(\chi-1)f_u^{(n)}+\mathcal{\tilde E}_u^{(n)}
\end{split}
\end{align}
and
\begin{align}\label{eps_h^n}
\begin{split}
  \mathcal{E}_h^{(n)}
  =&(\chi-1)f_h^{(n)}+\mathcal{\tilde E}_h^{(n)},
\end{split}
\end{align}
where
\begin{align}\label{E_u^n}
\begin{split}
\mathcal{\tilde E}_u^{(n)}=
  &-(1+u_p^0)\sqrt{\frac{\eps}{x}}\chi'v_p^n+u_{sx}^{(n)}\int_x^\infty \left[\sqrt{\frac{\eps}{\theta}}\chi'v_p^n+\chi'\frac{\sqrt\eps z}{\theta}u_p^n\right]{\rm d}\theta\\
  &+v_s^{(n)}\int_x^\infty \left[\frac{\eps}{\theta}\chi''v_p^n+2\sqrt{\frac{\eps}{\theta}}\chi'v_{py}^n-\sqrt{\frac{\eps}{\theta}}\chi' zu_{py}^n\right]{\rm d}\theta\\
  &+(\sigma+h_p^0)\sqrt{\frac{\eps}{x}}\chi'g_p^n-h_{sx}^{(n)}\int_x^\infty \left[\sqrt{\frac{\eps}{\theta}}\chi'g_p^n+\chi'\frac{\sqrt\eps z}{\theta}h_p^n\right]{\rm d}\theta\\
  &-g_s^{(n)}\int_x^\infty \left[\frac{\eps}{\theta}\chi''g_p^n+2\sqrt{\frac{\eps}{\theta}}\chi'g_{py}^n-\sqrt{\frac{\eps}{\theta}}\chi' zh_{py}^n\right]{\rm d}\theta\\
  &-\int_x^\infty \left[\frac{\sqrt\eps}{\theta}\chi'\frac{z}{2}u_{pyy}^n-3\frac{\eps}{\theta}\chi''v_{py}^n-3\sqrt{\frac{\eps}{\theta}}\chi'v_{pyy}^n
  -\left(\frac{\eps}{\theta}\right)^{\frac{3}{2}}\chi'''v_p^n\right]{\rm d}\theta
\end{split}
\end{align}
and
\begin{align}\label{E_h^n}
\begin{split}
\mathcal{\tilde E}_h^{(n)}=
  &-(1+u_p^0)\sqrt{\frac{\eps}{x}}\chi'g_p^n+h_{sx}^{(n)}\int_x^\infty \left[\sqrt{\frac{\eps}{\theta}}\chi'v_p^n+\chi'\frac{\sqrt\eps z}{\theta}u_p^n\right]{\rm d}\theta\\
  &+v_s^{(n)}\int_x^\infty \left[\frac{\eps}{\theta}\chi''g_p^n+2\sqrt{\frac{\eps}{\theta}}\chi'g_{py}^n-\sqrt{\frac{\eps}{\theta}}\chi' zh_{py}^n\right]{\rm d}\theta\\
  &+(\sigma+h_p^0)\sqrt{\frac{\eps}{x}}\chi'v_p^n-u_{sx}^{(n)}\int_x^\infty \left[\sqrt{\frac{\eps}{\theta}}\chi'g_p^n+\chi'\frac{\sqrt\eps z}{\theta}h_p^n\right]{\rm d}\theta\\
  &-g_s^{(n)}\int_x^\infty \left[\frac{\eps}{\theta}\chi''v_p^n+2\sqrt{\frac{\eps}{\theta}}\chi'v_{py}^n-\sqrt{\frac{\eps}{\theta}}\chi' zu_{py}^n\right]{\rm d}\theta\\
  &-\int_x^\infty \left[\frac{\sqrt\eps}{\theta}\chi'\frac{z}{2}h_{pyy}^n-3\frac{\eps}{\theta}\chi''g_{py}^n-3\sqrt{\frac{\eps}{\theta}}\chi'g_{pyy}^n
  -\left(\frac{\eps}{\theta}\right)^{\frac{3}{2}}\chi'''g_p^n\right]{\rm d}\theta.
\end{split}
\end{align}

Then one can obtain the following lemma for the error terms.
\begin{lemma}\label{eps_uh^n_lemma}
For sufficiently small constant $\kappa>0$ and $n\geq 2$, $\sigma_n$ defined as in Proposition \ref{u_p_prop}, the error terms $\mathcal{ E}_u^{(n)},\mathcal{ E}_h^{(n)}$ satisfy the estimates
\begin{align}
\label{eps_uh^n_estimate1}
  &|\p_x^k\mathcal{ E}_u^{(n)},\p_x^k\mathcal{ E}_h^{(n)}|\leq C(n,\kappa)\eps^{\frac{1}{4}-\kappa}x^{-k-\frac{3}{2}+\sigma_n+\kappa},\\
\label{eps_uh^n_estimate1}
  &\Vert \p_x^k\mathcal{ E}_u^{(n)},\p_x^k\mathcal{ E}_h^{(n)}\Vert_{L_y^2}\leq C(n,\kappa)\eps^{\frac{1}{4}-\kappa}x^{-k-\frac{5}{4}+\sigma_n+\kappa}.
\end{align}
\end{lemma}
\begin{proof}
The terms in $\mathcal{ E}_u^{(n)},\mathcal{ E}_h^{(n)}$ can be estimated via similar arguments as in Lemma 5.26 in \cite{Iyerglobal1}. Here we omit the proof.
\end{proof}

\begin{remark}\label{remark_n>2}
Let us mention that the factor $x^{-\frac{5}{4}}$ in the decay rate of \eqref{eps_uh^n_estimate1} is very important for the global-in-$x$ stability of the Prandtl expansion, which is used to control the forcing term of the remainder profile $(u,v,h,g)$ (see Remark \ref{remark_n>2_S6}). It is the reason why we should expand the approximate solutions to $\eps^{\frac{n}{2}}$-order for $n\geq 2$, so as to obtain the enhanced decay rate. See also Remark \ref{remark_enhance} and Theorem \ref{u_p^i_theorem} for details.\end{remark}

 Now we prove Theorem \ref{R^un_theorem}.
\begin{proof}[Proof of Theorem \ref{R^un_theorem}]
Theorem \ref{R^un_theorem} follows directly from Lemma \ref{eps_uh^n_lemma} and the estimates for the profiles $(u^i_p,v^i_p,h^i_p,g^i_p)$ and $(u^i_e,v^i_e,h^i_e,g^i_e)$ constructed before. Note that the remaining terms in $(R^{u,n},R^{v,n},R^{h,n},R^{g,n})$ can be estimated via similar arguments of Lemma 5.27 in \cite{Iyerglobal1}.
 Here we omit the details.
\end{proof}

From now on, for simplicity of the presentation, we still use the notations $(u^n_p,v^n_p,h^n_p,g^n_p)$ instead of $  (\tilde{u_p^n},\tilde{v_p^n},\tilde{h_p^n},\tilde{g_p^n})$.

\section{Construction of the approximate solutions}\label{sec5}

Based on the construction in previous sections, we can summarize the property for each profile. Before that, we introduce the following notation for simplification
\begin{equation}\label{notation1}
\begin{cases}
  u_{s}^{\mathrm{P}}:=\sum\limits_{j=0}^{n} \eps^{\frac{j}{2}} u_{p}^{j},
  \quad u_{s}^{\mathrm{E}}:=1+\sum\limits_{j=1}^{n} \eps^{\frac{j}{2}} u_{e}^{j},
  \quad u_{s}:=\bar{u}_{s}^{(n)}=u_{s}^{\mathrm{P}}+u_{s}^{\mathrm{E}},\\
  u_{s}^{\mathrm{P}, n-1}=\sum\limits_{j=0}^{n-1} \eps^{\frac{j}{2}} u_{p}^{j},
  \quad \text { so that } u_{s}^{\mathrm{P}}=u_{s}^{\mathrm{P}, n-1}+\eps^{\frac{n}{2}} u_{p}^{n}.
\end{cases}
\end{equation}

And similarly, for the vertical components, we denote
\begin{equation}\label{notation2}
\begin{cases}
  v_{s}^{\mathrm{P}}:=\sum\limits_{j=0}^{n} \eps^{\frac{j}{2}} v_{p}^{j},
  \quad v_{s}^{\mathrm{E}}:=\sum\limits_{j=1}^{n} \eps^{\frac{j-1}{2}} v_{e}^{j},
  \quad v_{s}:=\bar{v}_{s}^{(n)}=v_{s}^{\mathrm{P}}+v_{s}^{\mathrm{E}},\\
  v_{s}^{\mathrm{P}, n-1}=\sum\limits_{j=0}^{n-1} \eps^{\frac{j}{2}} v_{p}^{j},
  \quad \text { so that } v_{s}^{\mathrm{P}}=v_{s}^{\mathrm{P}, n-1}+\eps^{\frac{n}{2}} v_{p}^{n}.
\end{cases}
\end{equation}

We also use similar notations for the magnetic fields $(h_p^j,g_p^j)$.

In conclusion, for the approximate solutions, we finally get
\begin{theorem}\label{summary}
Assume $n\geq 2$, the positive constants $\delta,\sigma,\eps$ are sufficiently small relative to universal constants and $\eps\ll\eps_*$. The boundary and in-flow values are specified as \eqref{bc_0_prandtl}-\eqref{cp_2}. Then,

(1) there exist boundary layer profiles $(u_p^i,v_p^i,h_p^i,g_p^i)$ $(i=0,1,\cdots,n-1)$, such that

(i) in the case of $k=0$, the solutions satisfy:
\begin{align}
\label{prantl_estimate1}
  &\Vert \p_y^j (u_s^{P,n-1},h_s^{P,n-1})z^m x^{\frac{j}{2}}\Vert_{L^\infty}\leq \mathcal{O}(\delta,\sigma;m,n,j)~{\rm for}~0\leq j\leq 2,\\
  &\Vert \p_y^j (u_s^{P,n-1},h_s^{P,n-1})z^m x^{\frac{j}{2}}\Vert_{L^\infty}\leq C(m,n,j)~{\rm for}~j> 2,\\
  &\Vert \p_y^j (v_s^{P,n-1},g_s^{P,n-1})z^m x^{\frac{j}{2}+\frac{1}{2}}\Vert_{L^\infty}\leq \mathcal{O}(\delta,\sigma;m,n,j)~{\rm for}~j=0,1,\\
  &\Vert \p_y^j (v_s^{P,n-1},g_s^{P,n-1})z^m x^{\frac{j}{2}+\frac{1}{2}}\Vert_{L^\infty}\leq C(m,n,j)~{\rm for}~j\geq 2,
\end{align}

(ii) in the case of $k\geq 1$, the solutions satisfy:
\begin{align}
  &\Vert \p_x (u_s^{P,n-1},h_s^{P,n-1}) z^mx\Vert_{L^\infty}\leq \mathcal{O}(\delta,\sigma;m,n)~{\rm for}~j=0,\\
  &\Vert \p_x\p_y^j (u_s^{P,n-1},h_s^{P,n-1}) z^mx^{1+\frac{j}{2}}\Vert_{L^\infty}\leq C(m,n,j)~{\rm for}~j\geq 1,\\
  &\Vert \p_x^k\p_y^j (u_s^{P,n-1},h_s^{P,n-1}) z^mx^{k+\frac{j}{2}}\Vert_{L^\infty}\leq C(m,n,k,j)~{\rm for}~k>1,~j\geq 0,\\
  &\Vert \p_x^k\p_y^j (v_s^{P,n-1},g_s^{P,n-1})z^m x^{k+\frac{j}{2}+\frac{1}{2}}\Vert_{L^\infty}\leq C(m,n,k,j)~{\rm for}~k\geq 1,~j\geq 0,
\end{align}

(2) final boundary layer profiles $(u_p^n,v_p^n,h_p^n,g_p^n)$ satisfy
\begin{align}
  &\Vert z^m\p_x^k\p_y^j (u_p^n,h_p^n)x^{k+\frac{j}{2}+\frac{1}{2}-\sigma_n}\Vert_{L^\infty}\leq \eps^{-\frac{m}{2}}C(\delta,\sigma,;m,k,j),
  {\rm for~} m,k,j\geq 0,\\
  \label{prantl_estimate10}
  &\Vert z^m\p_x^k\p_y^j (v_p^n,g_p^n)x^{k+\frac{j}{2}+1-\sigma_n}\Vert_{L^\infty}\leq \eps^{-\frac{m}{2}}C(\delta,\sigma;m,k,j),{\rm for~any~} m,k,j\geq 0,
\end{align}

(3) ideal MHD profiles $(u_e^i,v_e^i,h_e^i,g_e^i)$ $(i=1,\cdots,n)$ satisfy
\begin{align}
\label{ideal_estimate1}
  &\Vert (u_s^E-1,h_s^E)x^{\frac{1}{2}},(u_{sx}^E,h_{sx}^E) x^{\frac{3}{2}}\Vert_{L^\infty}\leq \mathcal{O}(\delta,\sigma),\\
  &\Vert (v_s^E,g_s^E)x^{\frac{1}{2}},(v_{sY}^E,g_{sY}^E) x^{\frac{3}{2}}\Vert_{L^\infty}\leq \mathcal{O}(\delta,\sigma),\\
  &\Vert \p_x^k(v_s^E,g_s^E) x^{k-\frac{1}{2}}Y\Vert_{L^\infty}\leq C(k,j)~{\rm for}~ k\geq 1,~j=0,\\
  &\Vert \p_x^k\p_Y^j (v_s^E,g_s^E)x^{k+j+\frac{1}{2}}\Vert_{L^\infty}\leq C(k,j)~{\rm for}~k+j>0,\\
\label{ideal_estimate5}
  &\Vert \p_x^k\p_Y^j (u_s^E,h_s^E)x^{k+j+\frac{1}{2}}\Vert_{L^\infty}\leq \sqrt\eps C(k,j)~{\rm for}~k+j>0.
\end{align}
\end{theorem}
\begin{proof}
This theorem is a direct consequence of the estimates obtained in all above sections, and we omit the details here.
\end{proof}

\section{The global-in-$x$ stability for Prandtl expansions: Proof of Theorem \ref{maintheorem}}\label{sec6}
This section is to show the global-in-$x$ stability for the Prandtl expansion, i.e., our main result, Theorem \ref{maintheorem}.
With the approximate solutions constructed in the above sections, we obtain the following system for remainder profiles $(u,v,h,g,p)$
\begin{align}\label{u_epssystem}
\begin{cases}
  -\Delta_\eps u+S_u(u,v,h,g)+p_x=F_u(u,v,h,g),\\
  -\Delta_\eps v+S_v(u,v,h,g)+\frac{p_y}{\eps }=F_v(u,v,h,g),\\
  -\Delta_\eps h+S_h(u,v,h,g)=F_h(u,v,h,g),\\
  -\Delta_\eps g+S_g(u,v,h,g)=F_g(u,v,h,g),\\
  \p_x u+\p_y v=\p_x h+\p_y g=0,
\end{cases}
\end{align}
where
\begin{align}\label{S_u}
\begin{cases}
S_u=&u_s\p_x u+u\p_x u_s+v_s\p_y u+v\p_y u_s\\
   & -(h_s\p_x h+h\p_x h_s+g_s\p_y h+g\p_y h_s),\\
S_v=&u_s\p_x v+u\p_x v_s+v_s\p_y v+v\p_y v_s\\
   & -(h_s\p_x g+h\p_x g_s+g_s\p_y g+g\p_y g_s),\\
S_h=&u_s\p_x h+u\p_x h_s+v_s\p_y h+v\p_y h_s\\
   & -(h_s\p_x u+h\p_x u_s+g_s\p_y u+g\p_y u_s),\\
S_g=&u_s\p_x g+u\p_x g_s+v_s\p_y g+v\p_y g_s\\
   &-(h_s\p_x v+h\p_x v_s+g_s\p_y v+g\p_y v_s),
\end{cases}
\end{align}
in which the source terms $F_u,F_v,F_h,F_g$ are denoted as
\begin{equation}\label{F_u}
  (F_u,F_v,F_h,F_g)=\eps^{-\frac{n}{2}-\gamma}(R^{u,n},R^{v,n},R^{h,n},R^{g,n})-(N^u,N^v,N^h,N^g)
\end{equation}
and the terms $R^{u,n},R^{v,n},R^{h,n},R^{g,n}$ are given by \eqref{R^u_n}-\eqref{R^g_n}. In addition, the nonlinear terms $N^u,N^v,N^h,N^g$ are defined by
\begin{align}\label{N^u}
\begin{cases}
   N^u=\eps^{\frac{n}{2}+\gamma}(u\p_x u +v\p_y u -h\p_x h -g\p_y h),\\
   N^v=\eps^{\frac{n}{2}+\gamma}(u\p_x v +v\p_y v -h\p_x g -g\p_y g),\\
   N^h=\eps^{\frac{n}{2}+\gamma}(u\p_x h +v\p_y h -h\p_x u -g\p_y u),\\
   N^g=\eps^{\frac{n}{2}+\gamma}(u\p_x g +v\p_y g -h\p_x v -g\p_y v).
\end{cases}
\end{align}
Boundary conditions for the remainder profiles are given as follows
\begin{align}\label{u_epsboundary}
  \begin{cases}
  (u ,v ,h ,g)|_{y=0}=(u ,v ,h ,g)|_{y\rightarrow\infty}=(0,0,0,0),\\
  (u ,v ,h ,g)|_{x=1}=(u ,v ,h ,g)|_{x\rightarrow\infty}=(0,0,0,0).
\end{cases}
\end{align}

To solve the problem, let us first localize \eqref{u_epssystem} on domain $\Omega^\beta=\{(x,y)|x\geq 1,0\leq y<\beta\}$ with the following artificial boundary conditions
\begin{align}\label{u_epsboundary_omega_N}
  \begin{cases}
  (u ,v ,h ,g)|_{y=0}=(u ,v ,h ,g)|_{y=\beta}=(0,0,0,0),\\
  (u ,v ,h ,g)|_{x=1}=(u ,v ,h ,g)|_{x\rightarrow\infty}=(0,0,0,0).
\end{cases}
\end{align}
All estimates will be derived independent of $\beta$. And then,  let $\beta\to \infty$ to obtain the solutions on domain $\Omega$ (see Subsection \ref{sec6.2}).

\begin{remark}\label{remark_sec6}
Note that the argument presented in \cite{Iyerglobal2,Iyerglobal3} can be also applied to our problem with slightly modification, for the sake of completeness, we sketch the framework and emphasize key points here, see \cite{Iyerglobal2,Iyerglobal3} for more details.
\end{remark}

\subsection{The linear stability estimates}\label{sec6.1}
In this subsection, our goal is to derive the linear stability estimates for the following linearized equations together with boundary conditions \eqref{u_epsboundary_omega_N}
\begin{align}\label{u_eps_linearized}
\begin{cases}
  -\Delta_\eps u+S_u(u ,v ,h ,g)+p_x=f_u(u,h;\bar u,\bar v,\bar h,\bar g),\\
  -\Delta_\eps v+S_v(u ,v ,h ,g)+\frac{p_y}{\eps }=f_v(\bar u,\bar v,\bar h,\bar g),\\
  -\Delta_\eps h+S_h(u ,v ,h ,g)=f_h(u,h;\bar u,\bar v,\bar h,\bar g),\\
  -\Delta_\eps g+S_g(u ,v ,h ,g)=f_g(\bar u,\bar v,\bar h,\bar g),\\
  \p_x u+\p_y v=\p_x h+\p_y g=0,
\end{cases}
\end{align}
where the linearized source terms $f_u,f_v,f_h,f_g$ are denoted as
\begin{align}\label{f^u_linear}
\begin{split}
  (f_u,f_v,f_h,f_g)=&\eps^{-\frac{n}{2}-\gamma}(R^{u,n},R^{v,n},R^{h,n},R^{g,n})\\
  &-(N^u,N^v,N^h,N^g)(u,h;\bar u,\bar v,\bar h,\bar g)
\end{split}
\end{align}
with terms $(N^u,N^v,N^h,N^g)(u,h;\bar u,\bar v,\bar h,\bar g)$ defined by
\begin{align}\label{N^u_linear}
\begin{cases}
   N^u(u,h;\bar u,\bar v,\bar h,\bar g)
   =\eps^{\frac{n}{2}+\gamma}(\bar u\p_x\bar u +\bar v\p_y u -\bar h\p_x\bar h -\bar g\p_y h),\\
   N^v(\bar u,\bar v,\bar h,\bar g)
   =\eps^{\frac{n}{2}+\gamma}(\bar u\p_x \bar v +\bar v\p_y\bar v -\bar h\p_x\bar g -\bar g\p_y\bar g),\\
   N^h(u,h;\bar u,\bar v,\bar h,\bar g)
   =\eps^{\frac{n}{2}+\gamma}(\bar u\p_x\bar h +\bar v\p_y h -\bar h\p_x\bar u -\bar g\p_y u),\\
   N^g(\bar u,\bar v,\bar h,\bar g)
   =\eps^{\frac{n}{2}+\gamma}(\bar u\p_x\bar g +\bar v\p_y\bar g -\bar h\p_x\bar v -\bar g\p_y\bar v).
\end{cases}
\end{align}

\begin{remark}\label{remark_linear_vu_y_pre}
We remark that some terms are in form of $\bar v u_y, \bar g h_y, \bar v h_y, \bar g u_y$ instead of $\bar v \bar u_y, \bar g \bar h_y, \bar v \bar h_y, \bar g \bar u_y$ for technical reasons, see also Remark \ref{remark_linear_vu_y} for details.
\end{remark}

Let us introduce the $\mathcal{Z}$ norm
\begin{align}\label{norm}
\begin{split}
  \|(u ,v ,h ,g)\|_{ \mathcal{Z}}
  :=&\Vert (u ,v ,h ,g)\Vert_{X_1\cap X_2\cap X_3}+\eps^{N_1}\Vert (u ,v ,h ,g)\Vert_{Y_1}\\
  &+\eps^{N_2}\Vert (u ,v ,h ,g)\Vert_{Y_2}+\|(u ,v ,h ,g)\|_{ \mathcal{U}},
\end{split}
\end{align}
where we have used the notations as follows
\begin{align}
\label{X1norm}
  \|(u ,v ,h ,g)\|_{X_{1}}^{2}
  :=&\left\|\p_y(u,h)\right\|_{L^{2}}^{2}+\left\|\sqrt{\eps}\p_x(v,g) x^{\frac{1}{2}}\right\|_{L^{2}}^{2}+\left\|\p_y(v,g) x^{\frac{1}{2}}\right\|_{L^{2}}^{2},\\
\label{X2norm}
  \|(u ,v ,h ,g)\|_{X_{2}}^{2}
  :=&\left\|\p_{x y}(u,h) \cdot \rho_{2} x\right\|_{L^{2}}^{2}+\left\|\sqrt{\eps} \p_{x x}(v,g)\cdot\left(\rho_{2} x\right)^{\frac{3}{2}}\right\|_{L^{2}}^{2}\nonumber\\
  &+\left\|\p_{x y}(v,g)\cdot\left(\rho_{2} x\right)^{\frac{3}{2}}\right\|_{L^{2}}^{2},\\
\label{X3norm}
  \|(u ,v ,h ,g)\|_{X_{3}}^{2}
  :=&\left\|\p_{x x y}(u,h) \cdot\left(\rho_{3} x\right)^{2}\right\|_{L^{2}}^{2}+\left\|\sqrt{\eps} \p_{x x x}(v,g)\cdot\left(\rho_{3} x\right)^{\frac{5}{2}}\right\|_{L^{2}}^{2}\nonumber\\
  &+\left\|\p_{x x y}(v,g) \cdot\left(\rho_{3} x\right)^{\frac{5}{2}}\right\|_{L^{2}}^{2},\\
\label{Y2norm}
  \|(u ,v ,h ,g)\|_{Y_{1}}^{2}
  :=&\left\|\p_{x y}(u,h)\cdot x\right\|_{L^{2}}^{2}+\left\|\sqrt{\eps}\p_{x x}(v,g)\cdot x^{\frac{3}{2}}\right\|_{L^{2}}^{2}\nonumber\\
  &+\left\|\p_{x y}(v,g)\cdot x^{\frac{3}{2}}\right\|_{L^{2}}^{2}+\left\|\p_{y y}(u,h)\right\|_{L^{2}(x \leq 2000)}^{2},\\
\label{Y3norm}
  \|(u ,v ,h ,g)\|_{Y_{2}}^{2}
  :=&\left\|\p_{x x y}(u,h) \cdot \zeta_{3} x^{2}\right\|_{L^{2}}^{2}+\left\|\sqrt{\eps}\p_{x x x}(v,g) \cdot \zeta_{3} x^{\frac{5}{2}}\right\|_{L^{2}}^{2}\nonumber\\
  &+\left\|\p_{x x y}(v,g) \cdot \zeta_{3} x^{\frac{5}{2}}\right\|_{L^{2}}^{2},
\end{align}
in which the $Y_i$ norm is to descirbe the behavior near the boundary $x=1$. The cut-off functions $\rho_k(x)$ corresponding to $X_k(k=2,3)$ and $\zeta_3(x)$ are given by
\begin{align}\label{rho_k}
\rho_{k}(x)=\left\{\begin{array}{ll}
0 & \text { for } 1 \leq x \leq 50+50(k-2), \\
1 & \text { for } x \geq 60+50(k-2),
\end{array}\right.
\end{align}
and
\begin{align}\label{zeta}
\zeta_{3}(x)=\left\{\begin{array}{ll}
0 & \text { for } 1 \leq x \leq \frac{3}{2},\\
1 & \text { for } x \geq 2.
\end{array}\right.
\end{align}
In addition, define the following uniform norm
\begin{align}\label{uniform_norm}
\begin{split}
  \|(u ,v ,h ,g)\|_{ \mathcal{U}}
  :=&\eps^{N_3}\Vert (u,h)x^{\frac{1}{4}}+\sqrt\eps (v,g)x^{\frac{1}{2}}\Vert_{L^\infty}\\
  &+\eps^{N_4}\sup_{x\geq 20}\Vert (u_x,h_x)x^{\frac{5}{4}}+\sqrt\eps(v_x,g_x)x^{\frac{3}{2}}\Vert_{L^\infty}\\
  &+\eps^{N_5}\sup_{x\geq 20}\Vert (u_y,h_y)x^{\frac{1}{2}}\Vert_{L_y^2}\\
  &+\eps^{N_6}\left[\int_{20}^\infty x^4\Vert \sqrt\eps (v_{xx},g_{xx})\Vert_{L_y^\infty}^2 {\rm d}x\right]^{\frac{1}{2}}.
\end{split}
\end{align}
In above definitions, the constants $N_1,\cdots,N_6$ will be determined latter in \eqref{constants}.

For reader's convenience, let us first explain the structure of $\mathcal{Z}(\Omega^\beta)$ norm (refer to Section 2 in \cite{Iyerglobal2} for more details). Recall that the main goal in this section is to derive $L^\infty$ norm for $(u ,v ,h ,g)$. To this end, we should derive the estimates for $\Vert (u,h)x^{\frac{1}{4}}\Vert_{L^\infty}$ and $\Vert (v,g)x^{\frac{1}{2}}\Vert_{L^\infty}$, which are generated by terms $\bar u \bar u_{x}+\bar h \bar h_{x}, \bar u \bar h_{x}+\bar h \bar u_{x}$ and $\bar v u_{y}+\bar g h_{y},\bar v h_{y}+\bar g u_{y}$ in $N^u,N^h$ respectively. For example, in the basic (first-order) energy estimate, applying multiplier $u$  to equation \eqref{u_epssystem}$_1$, the term $\bar u \bar u_{x}+\bar h \bar h_{x}$ is estimated as
\begin{align}\label{vg_Linfty_origin}
\begin{split}
 \iint \eps^{\frac{n}{2}+\gamma} (\bar u \bar u_{x}+\bar h \bar h_{x}) u
 \leq\eps^{\frac{n}{2}+\gamma}\left\|(\bar u,\bar h) x^{\frac{1}{4}}\right\|_{L^\infty}
  \left\|(\bar u_{x},\bar h_{x}) x^{\frac{1}{2}}\right\|_{L^2}\left\| u x^{-\frac{3}{4}}\right\|_{L^2},
\end{split}
\end{align}
and in the first-order positivity estimate derived by multiplier $v_y x$, we get
\begin{align}\label{vg_Linfty_origin}
\begin{split}
  \iint \eps^{\frac{n}{2}+\gamma} (\bar v u_{y}+\bar g h_{y}) v_{y} x
  \leq\eps^{\frac{n}{2}+\gamma}\left\|(\bar v,\bar g) x^{\frac{1}{2}}\right\|_{L^\infty}
  \left\|(u_{y},h_{y})\right\|_{L^2}\left\|v_{y} x^{\frac{1}{2}}\right\|_{L^2}.
\end{split}\end{align}
To close the estimates in $L^\infty$ sense, the following lemma about the short-$x$ and long-$x$ estimates is needed.
\begin{lemma}\label{short-long-x}
Let $(u,v,h,g)$ be the solution to \eqref{u_eps_linearized} with \eqref{u_epsboundary_omega_N}, then exists large enough $M_1$ independent of $n,\eps,\gamma,\beta$ and $N_i(i=1,\cdots,6)$ such that

(i) short-$x$ estimate:
\begin{align}\label{uvhg_infty}
\begin{split}
  \sup_{x\leq 2000}\Vert (u,v,h,g)\Vert_{L_y^\infty}
  \lesssim& \eps^{-M_1}\left[C+\Vert (\bar u,\bar v,\bar h,\bar g)\Vert_{X_1}
  +\eps^{\frac{n}{2}+\gamma}\Vert (\bar u,\bar v,\bar h,\bar g)\Vert_{L^\infty}\right.\\
  &\left.\qquad \cdot\left(\Vert (\bar u,\bar v,\bar h,\bar g)\Vert_{X_1}+\Vert(u, v,h,g)\Vert_{X_1}\right)\right].
\end{split}
\end{align}

(ii) long-$x$ estimates:
\begin{align}\label{uh_Linfty}
\begin{split}
  \sup_{x\geq 20}\left\|(u,h) x^{\frac{1}{4}}\right\|_{L_y^{\infty}}
  \leq \sup_{x\geq 20}\left\|(u_x,h_x) x^{\frac{5}{4}}\right\|_{L_y^{\infty}}
  \lesssim \|(u ,v ,h ,g)\|_{X_{1} \cap Y_{1} \cap Y_{2}}
\end{split}
\end{align}
and
\begin{align}\label{vg_Linfty}
\begin{split}
  \sup_{x\geq 20}\left\|\sqrt{\eps} (v,g) x^{\frac{1}{2}}\right\|_{L_y^{\infty}}
  \leq \sup_{x\geq 20}\left\|\sqrt{\eps} (v_x,g_x) x^{\frac{3}{2}}\right\|_{L_y^{\infty}}
  \lesssim\|(u ,v ,h ,g)\|_{X_{1} \cap Y_{1} \cap Y_{2}}.
\end{split}
\end{align}
\end{lemma}
\begin{proof}
The proof is similar to Lemma 2.11, Lemma 2.18 and Lemma 2.19 in \cite{Iyerglobal2} and we omit it here. See \cite{Iyerglobal2} for details.
\end{proof}
The above lemma implies that the norms $Y_1,Y_2$ should be contained in the $\mathcal{Z}$ norm, which can be controlled by the following lemma.
\begin{lemma}\label{Y_i-estimate}
Suppose that $\|(\bar u,\bar v,\bar h,\bar g)\|_{ \mathcal{Z}(\Omega^\beta)}\leq 1$. Let $(u,v,h,g)$ be the solutions to \eqref{u_eps_linearized} with \eqref{u_epsboundary_omega_N}, and $M_1$ be the constant determined in Lemma \ref{short-long-x}, then there exists large enough $M_2$ independent of $n,\eps,\gamma,\beta$ and $N_i(i=1,\cdots,6)$ such that
\begin{align}\label{Y2_inequality}
\begin{split}
  &\eps^{N_{1}}\|(u,v,h,g)\|_{Y_{1}}\\
  \lesssim & \eps^{N_{1}-M_{1}}+\eps^{\frac{n}{2}+\gamma+N_{1}-M_{1}-N_{3}}\|(\bar u,\bar v,\bar h,\bar g)\|_{\mathcal{Z}}^{2} \\
   &+\left(\eps^{N_{1}-M_{1}}+\eps^{\frac{n}{2}+\gamma+N_{1}-M_{1}-N_{3}}\right)\|(u,v,h,g)\|_{X_{1}}\\
   &+\eps^{N_{1}}\|(u,v,h,g)\|_{X_{2}}
\end{split}
\end{align}
and
\begin{align}\label{Y3_inequality}
\begin{split}
  &\eps^{N_{2}}\|(u,v,h,g)\|_{Y_{2}}\\
  \lesssim& \eps^{N_{2}-M_{2}}\left[C+\|(u,v,h,g)\|_{X_{1} \cap Y_{1}}\right]+\eps^{N_{2}}\|(u,v,h,g)\|_{X_{3}}\\
  &+\eps^{\frac{n}{2}+\gamma-N_{1}-N_{3}-M_{2}+N_{2}}\|(u,v,h,g)\|_{X_{1} \cap Y_{1}}\\
  &+\eps^{N_{2}-M_{2}-2 N_{1}-2 N_{3}+\frac{n}{2}+\gamma}\|(\bar u,\bar v,\bar h,\bar g)\|_{\mathcal{Z}}^{2}.
\end{split}
\end{align}
\end{lemma}
\begin{proof}
This lemma can be obtained via the similar arguments as those for Corallary 2.12 and Corallary 2.14 in \cite{Iyerglobal2}, here we omit the proof here.
\end{proof}
The above Lemma \ref{Y_i-estimate} implies that we should add the second and third order norms $X_2,X_3$ to the $\mathcal{Z}$ norm. Thanks to the cut-off and almost linear properties of weight functions introduced in $X_2,X_3$, the arguments  are significantly simplified, compared with $Y_1,Y_2$.

Now we can determine the constant $n$ in (\ref{expansion}) in the following way. Let $M_1,M_2$ be the constants determined in \eqref{Y2_inequality} and \eqref{Y3_inequality} respectively and
\begin{align}\label{constants}
\begin{cases}
N_1\ {\rm is~ chosen~with~}N_1-M_1=100,\\
N_2\ {\rm is~ chosen~with~}N_2-M_2=2N_1,\\
N_k=N_2, \ k=3,\cdots,6,
\end{cases}
\end{align}
then we choose $n$ such that
\begin{equation}\label{n-choice}
\frac{n}{2}+\gamma>2N_1-N_2+N_3+M_2+100.
\end{equation}
Obviously, one has $\frac{n}{2}+\gamma>N_3$.

The last two terms in uniform norm \eqref{uniform_norm} are needed for some terms of $\p_{xx} (N^u,N^h)$, which are parts of the third order estimates (see also \eqref{W3_3})
\begin{align}\label{infty_vxx}
\begin{split}
  & \eps^{\frac{n}{2}+\gamma}\iint\left(\bar{v}_{x x} u_{y}-\bar{g}_{x x} h_{y}\right)\cdot\left( u_{x x} \rho_{3}^{4} x^{4}+ v_{xxy} \rho_{3}^{5} x^{5}\right)\\
  & +\eps^{\frac{n}{2}+\gamma}\iint\left(\bar{v}_{x x} h_y-\bar{g}_{x x} u_{y}\right) \cdot\left( h_{x x} \rho_{3}^{4} x^{4}+ g_{xxy} \rho_{3}^{5} x^{5}\right)\\
  \leq& \eps^{\frac{n}{2}+\gamma}\left[\left\|(u_{x x},h_{x x})x^{\frac{3}{2}}\right\|_{L^{2}}
  +\left\|\left(v_{xxy},g_{xxy}\right)x^{\frac{5}{2}}\right\|_{L^{2}}\right]\\
  & \cdot\sup_{x\geq 20}\left\|\left(u_{y}, h_{y}\right) x^{\frac{1}{2}}\right\|_{L_{y}^{2}} \left( \int_{20}^{\infty} x^{4}\left\|\left(\bar v_{x x}, \bar{g}_{x x}\right)\right\|_{L^\infty}^{2}\right)^{\frac{1}{2}}.
\end{split}
\end{align}
With this, following the similar arguments for Corollary 2.21 in \cite{Iyerglobal2}, we have
\begin{lemma}\label{estimate-uniform-norm}
Let $N_i(i=1,\cdots,6)$ be \eqref{constants}, then it holds that
\begin{align}\label{embedding_inequality}
\begin{split}
  \|(u ,v ,h ,g)\|_{ \mathcal{U}}
  \lesssim \eps^{\max \left\{N_{1}, N_{2}\right\}}\|(u ,v ,h ,g)\|_{X_{1} \cap Y_{1} \cap Y_{2}}.
\end{split}
\end{align}
\end{lemma}
With the above preparations at hands, one can derive the following embedding lemma.
\begin{lemma}\label{Z-embeding}
Suppose that $\|(\bar u,\bar v,\bar h,\bar g)\|_{ \mathcal{Z}(\Omega^\beta)}\leq 1$ and $N_i(i=1,\cdots,6)$ are determined in \eqref{constants}, then there exists $\omega(N_{i})$ independent of $\eps,\gamma,\beta$ such that
\begin{align}\label{Z_inequality}
\begin{split}
\|(u,v,h,g)\|_{\mathcal{Z}(\Omega^{\beta})}
  \lesssim& \eps^{100}+\|(u,v,h,g)\|_{X_{1} \cap X_{2} \cap X_{3}(\Omega^{\beta})}\\
  &+\eps^{\frac{n}{2}+\gamma-\omega(N_{i})}\|(\bar u,\bar v,\bar h,\bar g)\|_{\mathcal{Z}(\Omega^{\beta})}^{2}.
\end{split}
\end{align}
\end{lemma}
\begin{proof}
The proof is similar to that for Theorem 2.22 in \cite{Iyerglobal2}, here we omit the proof. See \cite{Iyerglobal2} for details.
\end{proof}

It remains to derive the estimates of $\|(u,v,h,g)\|_{X_{1} \cap X_{2} \cap X_{3}(\Omega^{\beta})}$ for the system \eqref{u_eps_linearized}. Fortunately, applying 1-st to 3-rd order energy estimates and positivity estimates (similar to the proof for Theorem 3.19 in \cite{Iyerglobal2}), we have the following lemma:

\begin{lemma}\label{lemma_Y23Z}
Assume that $\|(\bar u,\bar v,\bar h,\bar g)\|_{ \mathcal{Z}(\Omega^\beta)}\leq 1$, we have the following linear estimates
\begin{equation}\label{X123_estimate_pre}
  \|(u,v,h,g)\|_{X_{1} \cap X_{2} \cap X_{3}(\Omega^{\beta})}^{2} \lesssim \mathcal{W}_{1}+\mathcal{W}_{2}+\mathcal{W}_{3},
\end{equation}
where the terms $\mathcal{W}_{1},\mathcal{W}_{2},\mathcal{W}_{3}$ are given by
\begin{align}
\label{def_W1}
&\mathcal{W}_{1}=\mathcal{W}_{1, E}+\mathcal{W}_{1, P},\\
\label{def_W1E}
&\mathcal{W}_{1, E}
  =\left|\iint f_{u}\cdot u \right|+\iint \eps\left| f_{v}\right| \cdot \left|v \right|
  +\left|\iint f_{h}\cdot h \right|+\iint \eps\left| f_{g}\right| \cdot \left|g \right|,\\
\label{def_W1P}
&\mathcal{W}_{1, P}
  =\iint \left|f_{u}\right||v_y|x+\iint \eps\left| f_{v}\right||v_x|x
  +\iint \left|f_{h}\right||g_y|x+\iint \eps\left| f_{g}\right||g_x|x,
\end{align}
and
\begin{align}
\label{def_W2}
&\mathcal{W}_{2}=\mathcal{W}_{2, E}+\mathcal{W}_{2, P},\\
\label{def_W2E}
&\mathcal{W}_{2, E}
  =\iint|\p_{x}f_u|\cdot|u_{x} |\cdot|\rho_{2}^{2} x^{2}|+\iint \eps|\p_{x}f_v|\cdot|v_{x} |\cdot|\rho_{2}^{2}x^{2}|\nonumber\\
  &\qquad\quad +\iint|\p_{x}f_h|\cdot|h_{x} |\cdot|\rho_{2}^{2} x^{2}|+\iint \eps|\p_{x}f_g|\cdot|g_{x} |\cdot|\rho_{2}^{2}x^{2}|,\\
\label{def_W2P}
&\mathcal{W}_{2, P}
  =\iint|\p_{x}f_u|\cdot|v_{x y} |\cdot|\rho_{2}^{3} x^{3}|+\iint \eps|\p_{x}f_v|\cdot|v_{x x} |\cdot|\rho_{2}^{3} x^{3}|\nonumber\\
  &\qquad\quad +\iint|\p_{x}f_h|\cdot|g_{x y} |\cdot|\rho_{2}^{3} x^{3}|+\iint \eps|\p_{x}f_g|\cdot|g_{x x} |\cdot|\rho_{2}^{3} x^{3}|,
\end{align}
and
\begin{align}
\label{def_W3}
&\mathcal{W}_{3}=\mathcal{W}_{3, E}+\mathcal{W}_{3, P},\\
\label{def_W3E}
&\mathcal{W}_{3, E}
  =\iint|\p_{xx}f_u|\cdot|u_{x x} |\rho_{3}^{4}x^4+\iint \eps|\p_{xx}f_v|\cdot|v_{x x} |\rho_{3}^{4}x^4\nonumber\\
  &\qquad\quad +\iint|\p_{xx}f_h|\cdot|h_{x x} |\rho_{3}^{4}x^4+\iint \eps|\p_{xx}f_g|\cdot|g_{x x} |\rho_{3}^{4}x^4,\\
\label{def_W3P}
&\mathcal{W}_{3, P}
  =\iint|\p_{xx}f_u|\cdot|u_{x x x} | \rho_{3}^{5}x^5+\iint \eps|\p_{xx}f_v|\cdot|v_{x x x} |\rho_{3}^{5}x^5\nonumber\\
  &\qquad\quad +\iint|\p_{xx}f_h|\cdot|h_{x x x} | \rho_{3}^{5}x^5+\iint \eps|\p_{xx}f_g|\cdot|g_{x x x} |\rho_{3}^{5}x^5.
\end{align}
The subscript $E$ and $P$ represent the forcing terms produced by the energy and positivity estimates respectively, $f_u,f_v,f_h,f_g$ are given by \eqref{f^u_linear}, and the cut-off functions $\rho_2,\rho_3$ are defined as in \eqref{rho_k} for $k=2,3$.
\end{lemma}

\begin{remark}\label{remark_global_hardy}
The Hardy inequalities \eqref{Hardy_remainder_x} and \eqref{Hardy_remainder_y} play the key roles in the estimates for the terms in $S_u,S_v,S_h,S_g$ in \eqref{S_u}.
Let us sketch the main idea to bound some terms in $S_u$ and others are similar. For example, in the 1-st order energy estimate, thanks to Theorem \ref{summary}, $h h_{sx}$ is controlled by
\begin{align}\label{remark_hardy1}
\begin{split}
  \iint h h_{sx} \cdot u
  =&\iint \left[h_{sx}^{P,n-1}+\eps^{\frac{n}{2}} h_{p x}^{n}+h_{s x}^{E} \right]\cdot h u\\
  =&\iint y^{2} h_{sx}^{P,n-1} \cdot \frac{h}{y}\cdot \frac{u}{y}
  +\eps^{\frac{n}{2}}\iint h_{p x}^{n} y x^{1-\sigma_{n}} \cdot\frac{h}{y} \cdot \frac{u}{x^{1-\sigma_{n}} }\\
  &+\iint h_{s x}^{E}x^{\frac{3}{2}}\cdot hx^{-\frac{3}{4}} \cdot ux^{-\frac{3}{4}} \\
  \leq& \left\|z^{2} h_{s x}^{P,n-1}x\right\|_{L^\infty} \left\|(u_{y},h_{y})\right\|_{L^{2}}
  +\left\|h_{s x}^{E} x^{\frac{3}{2}} \right\|_{L^\infty} \left\|(u_{x},h_{x}) x^{\frac{1}{4}}\right\|_{L^{2}} \\
  &+\varepsilon^{\frac{n}{2}}\left\|z h_{p x}^{n} x^{\frac{3}{2}-\sigma_{n}}\right\|_{L^{\infty}} \cdot\left\|h_{y}\right\|_{L^{2}} \cdot\left\|u_{x} x^{\sigma_{n}}\right\|_{L^{2}}\\
  \leq& \left[\mathcal{O}(\delta, \sigma)+\sqrt\eps\right] \cdot\left[\left\|(u_{y},h_{y})\right\|_{L^{2}}^{2}+\left\|(u_{x},h_{x}) x^{\frac{1}{4}}\right\|_{L^{2}}^{2}\right] .
\end{split}
\end{align}
Similarly, the singular term $g h_{sy}$ can be bounded by
\begin{align}\label{remark_hardy2}
\begin{split}
  \iint g h_{sy} \cdot u
  =&\iint \left[h_{sy}^{P,n-1}+\eps^{\frac{n}{2}} h_{p y}^{n}+\sqrt\eps h_{s Y}^{E} \right]\cdot g u\\
   \leq& \left\|z^{2} h_{sy}^{P,n-1} x^{\frac{1}{2}}\right\|_{L^\infty} \cdot\left\|g_{y} x^{\frac{1}{2}}\right\|_{L^{2}} \cdot\left\|u_{y}\right\|_{L^{2}}\\
  &+\varepsilon^{\frac{n}{2}}\left\|z h_{p y}^{n} x^{1-\sigma_{n}}\right\|_{L^{\infty}} \cdot\left\|g_{y} x^{\frac{1}{2}}\right\|_{L^{2}} \cdot\left\|u_{x} x^{\sigma_{n}}\right\|_{L^{2}}\\
  &+\left\|h_{sY}^E x^{\frac{3}{2}}\right\|_{L^{\infty}} \cdot\left\|\sqrt{\varepsilon} g_{x} x^{\frac{1}{4}}\right\|_{L^{2}}\cdot\left\|u_{x} x^{\frac{1}{4}}\right\|_{L^{2}}\\
  \lesssim& \left[\mathcal{O}(\delta, \sigma)+\sqrt\eps\right] \cdot\left[\left\|u_{y}\right\|_{L^{2}}^2+\left\|(u_{x},h_{x}) x^{\frac{1}{2}}\right\|_{L^{2}}^2
  +\left\|\sqrt{\varepsilon} g_{x} x^{\frac{1}{4}}\right\|_{L^{2}}^2\right].
\end{split}
\end{align}

The estimates \eqref{remark_hardy1} and \eqref{remark_hardy2} are completely different from those in \cite{DLXmhd} (see also \cite{GN17}), where the $L^2$ norm is bounded in the following way
\begin{align*}
  \| (u,v,h,g)\|_{ L^2}^2\leq L\| (u_x,v_x,h_x,g_x)\|_{ L^2}^2.
\end{align*}
This is why only local-in-$x$ stability is obtained in \cite{DLXmhd}.
\end{remark}

The following lemma is to derive the estimates for $\mathcal{W}_{1},\mathcal{W}_{2},\mathcal{W}_{3}$, which completes the estimates.
\begin{lemma}\label{lemma_X123}
Suppose that $\|(\bar u,\bar v,\bar h,\bar g)\|_{ \mathcal{Z}(\Omega^\beta)}\leq 1$. For any small viscosity coefficient $\eps\ll\eps_*$, and large enough $n$ with \eqref{n-choice}, then it holds that
\begin{align}\label{W123_estimate}
\begin{split}
  \left|\mathcal{W}_{1}+\mathcal{W}_{2}+\mathcal{W}_{3}\right|
  \lesssim & \eps^{\frac{1}{4}-\gamma-\kappa}+\eps^{\frac{1}{4}-\gamma-\kappa}\|(u,v,h,g)\|_{X_{1} \cap X_{2} \cap X_{3}}^{2} \\
  &+\eps^{\frac{n}{2}-\omega\left(N_{i}\right)}\|(u,v,h,g)\|_{\mathcal{Z}}^{2}
  +\eps^{\frac{n}{2}-\omega\left(N_{i}\right)}\|(\bar u,\bar v,\bar h,\bar g)\|_{\mathcal{Z}}^{4},
\end{split}
\end{align}
where $\omega(N_{i})$, $N_i$ are determined in Lemma \ref{Z-embeding}, \eqref{constants} respectively and $\gamma, \kappa$ are arbitrary constants with $0\leq\gamma<\frac{1}{4}, 0<\kappa\ll 1$ and $\gamma+\kappa<\frac{1}{4}$.
\end{lemma}

\begin{proof}
The proof of \eqref{W123_estimate} will be completed in three steps, according to the definition of $\mathcal{W}_{1},\mathcal{W}_{2},\mathcal{W}_{3}$ in \eqref{def_W1}-\eqref{def_W3P}.

{\bf{Step 1: Estimates for $\mathcal{W}_{1}$ (generated by first-order estimate)}.}

Let us begin with the terms $R^{u,n},R^{v,n},R^{h,n},R^{g,n}$. By the estimate \eqref{R^unterm_estimate_L2} and Hardy inequality \eqref{Hardy_remainder_x}, we get that
\begin{align}\label{W1_1}
\begin{split}
  &\eps^{-\frac{n}{2}-\gamma}\iint\left[|R^{u,n}|\cdot \left(|u|+|v_{y}x|\right)+|R^{h,n}|\cdot \left(|h|+|g_{y}x|\right)\right]\\
  &+\eps^{1-\frac{n}{2}-\gamma}\iint\left[ |R^{v,n}|\cdot \left(|v|+|v_{x}x|\right)+ |R^{g,n}|\cdot \left(|g|+|g_{x}x|\right)\right]\\
  \leq& \eps^{-\frac{n}{2}-\gamma}\left\|\left(R^{u, n}, R^{h, n}, \sqrt{\eps} R^{v, n}, \sqrt{\eps} R^{g, n}\right) x^{\frac{1}{2}+\sigma^{\prime}}\right\|_{L^{2}}\left[\left\|\left(u, h\right) x^{-\frac{1}{2}-\sigma^{\prime}}\right\|_{L^{2}}\right.\\
  &\left. +\left\|\left(\sqrt{\eps} v, \sqrt{\eps} g\right) x^{-\frac{1}{2}-\sigma^{\prime}}\right\|_{L^{2}}+\left\|\left(v_{y}, g_{y}, \sqrt{\eps} v_{x}, \sqrt{\eps}g_{x}\right) x^{\frac{1}{2}}\right\|_{L^{2}}\right]\\
  \leq& \eps^{\frac{1}{4}-\gamma-\kappa}\left\|x^{-\frac{5}{4}+2 \sigma_{n}+\kappa+\frac{1}{2}+\sigma^{\prime}}\right\|_{L^{2}} \cdot\left[\left\|\left(u_{x}, h_{x}, \sqrt{\eps} v_{x}, \sqrt{\eps} g_{x}\right) x^{\frac{1}{2}-\sigma^{\prime}}\right\|_{L^{2}}\right.\\
  &\left.+\left\|\left(v_{y}, g_{y}, \sqrt{\eps} v_{x}, \sqrt{\eps} g_{x}\right) x^{\frac{1}{2}}\right\|_{L^{2}}\right]\\
  \leq& c(n) \eps^{\frac{1}{4}-\gamma-\kappa}\|(u, v, h, g)\|_{X_{1}}.
\end{split}
\end{align}

We continue to estimate the other terms in $\mathcal{W}_{1,E}$. Thanks to the expressions of $N^u,N^h$ and Hardy inequality \eqref{Hardy_remainder_x}, we have
\begin{align}\label{W1_2}
\begin{split}
  &\left|\eps^{\frac{n}{2}+\gamma} \iint\left(\bar{u}\bar{u}_{x}-\bar{h} \bar{h}_{x}\right) \cdot u+\left(\bar u \bar h_{x}-\bar h \bar u_{x}\right) \cdot h\right|\\
  \leq& \eps^{\frac{n}{2}+\gamma} \Vert(\bar{u}_{x}, \bar{h}_{x}) x^{\frac{1}{4}}\Vert_{L^\infty} \cdot\Vert(\bar{u}_{x}, \bar{h}_{x}) x^{\frac{1}{2}}\Vert_{L^{2}} \cdot\Vert(u_{x}, h_{x}) x^{\frac{1}{4}}\Vert_{L^{2}} \\
  \leq& \eps^{\frac{n}{2}+\gamma-\omega\left(N_{i}\right)} \Vert(\bar{u}, \bar{v}, \bar{h}, \bar{g})\Vert_{\mathcal{Z}}^{2} \cdot\|(u, v, h, g)\|_{\mathcal{Z}}.
\end{split}
\end{align}
On the other hand, by Hardy inequality \eqref{Hardy_remainder_x} and the definition of $\mathcal{Z}$, one has
\begin{align}\label{W1_3}
\begin{split}
  &\left|\eps^{\frac{n}{2}+\gamma} \iint(\bar{v} u_{y}-\bar{g} h_{y})\cdot u+(\bar{v} h_{y}-\bar{g} u_{y}) \cdot h\right|\\
  =& \left|\eps^{\frac{n}{2}+\gamma} \iint\frac{1}{2}\bar{v}\p_y |u|^{2}+\frac{1}{2}\bar v\p_y |h|^{2}-\bar{g} \p_y(u h)\right|\\
  =&\left|\eps^{\frac{n}{2}+\gamma} \iint \frac{1}{2} \bar{v}_{y}\left(u^{2}+h^{2}\right)-\bar{g}_{y}\cdot u h\right|\\
  \leq& \eps^{\frac{n}{2}+\gamma} \left\|(u, h) x^{\frac{1}{4}}\right\|_{L^{\infty}}\cdot\left\|\left(\bar{v}_{y}, \bar{g}_{y}\right) x^{\frac{1}{2}}\right\|_{L^{2}}\cdot\left\|\left(u_{x}, h_{x}\right) x^{\frac{1}{4}}\right\|_{L^{2}}\\
  \leq& \eps^{\frac{n}{2}+\gamma-\omega\left(N_{i}\right)}\|(\bar{u}, \bar{v}, \bar{h}, \bar{g})\|_{\mathcal{Z}}\cdot\|(u, v, h, g)\|_{\mathcal{Z}}^{2} .
\end{split}
\end{align}

As for $N^v,N^g$, there holds that
\begin{align}
  &\eps^{1+\frac{n}{2}+\gamma}\iint\left[\left|\bar{u} \bar{v}_{x}+\bar{v} \bar{v}_{y}-\bar{h} \bar{g}_{x}-\bar{g} \bar{g}_{y}\right|\cdot |v|
  +\left|\bar{u} \bar{g}_{x}+\bar{v} \bar{g}_{y}-\bar{h} \bar{v}_{x}-\bar{g} \bar{v}_{y}\right|\cdot |g|\right] \nonumber\\
  \leq& \eps^{\frac{n}{2}+\gamma}\left\|(\bar{u}, \bar{h}) x^{\frac{1}{4}}\right\|_{L^\infty} \cdot\left\|\sqrt\eps\left(\bar{v}_{x}, \bar{g}_{x}\right) x^{\frac{1}{2}}\right\|_{L^{2}} \cdot\left\|\sqrt\eps (v, g) x^{-\frac{3}{4}}\right\|_{L^{2}}\nonumber\\
  &+\eps^{\frac{n}{2}+\gamma}\left\|\sqrt{\eps}(\bar{v}, \bar{g}) x^{\frac{1}{2}}\right\|_{L^{\infty}}\cdot \left\|\left(\bar{v}_{y}, \bar{g}_{y}\right) x^{\frac{1}{2}}\right\|_{L^{2}} \cdot \|\left.\sqrt{\varepsilon}(v, g) x^{-\frac{3}{4}}\right\|_{L^{2}}\nonumber\\
  \leq& \eps^{\frac{n}{2}+\gamma-\omega\left(N_{i}\right)}\|(\bar{u}, \bar{v}, \bar{h}, \bar{g})\|_{\mathcal{Z}}^{2}\cdot\|(u, v, h, g)\|_{\mathcal{Z}}\label{W1_4}
,
\end{align}
in which Hardy inequality \eqref{Hardy_remainder_x} has been used for the last term.

And then, for the term $\mathcal{W}_{1,P}$ from positivity estimate procedures, with respect to $N^u,N^h$, it is direct to get that
\begin{align}\label{W1_5}
\begin{split}
  &\iint \left|N^u|\cdot |v_{y} x\right|+\iint\left|N^h|\cdot |g_{y} x\right|\\
  \leq& \eps^{\frac{n}{2}+\gamma}\left\|(\bar{u}, \bar{h})\right\|_{L^{\infty}}\left\|\left(\bar{u}_{x}, \bar{h}_{x}\right) x^{\frac{1}{2}}\right\|_{L^{2}}\left\|\left(v_{y}, g_{y}\right) x^{\frac{1}{2}}\right\|_{L^{2}}\\
  &+\eps^{\frac{n}{2}+\gamma}\left\|(\bar{v}, \bar{g}) x^{\frac{1}{2}}\right\|_{L^{\infty}}\left\|\left(u_y, h_y\right)\right\|_{L^{2}}
  \cdot\left\|\left(v_{y}, g_{y}\right) x^{\frac{1}{2}}\right\|_{L^{2}},\\
  \leq& \eps^{\frac{n}{2}+\gamma-\omega\left(N_{i}\right)}\|(\bar{u}, \bar{v}, \bar{h}, \bar{g})\|_{\mathcal{Z}}^{2}\cdot\|(u, v, h, g)\|_{\mathcal{Z}}\\
  &+\eps^{\frac{n}{2}+\gamma-\omega\left(N_{i}\right)}\|(\bar{u}, \bar{v}, \bar{h}, \bar{g})\|_{\mathcal{Z}} \cdot\|(u, v, h, g)\|_{\mathcal{Z}}^{2}.
\end{split}
\end{align}
For the terms $N^v,N^g$, it follows that
\begin{align}\label{W1_6}
\begin{split}
  &\eps\iint\left|N^v\right|\cdot |v_{x} x|+\eps\iint\left|N^g\right| \cdot |g_{x} x| \\
  \leq &\eps^{\frac{n}{2}+\gamma}\left\|(\bar{u},\bar h) x^{\frac{1}{4}}\right\|_{L^{\infty}} \cdot \left\|\sqrt\eps (\bar{v}_{x}, \bar{g}_{x}) x^{\frac{1}{2}}\right\|_{L^{2}}
  \cdot\left\| \sqrt\eps(v_{x}, g_{x}) x^{\frac{1}{2}} \right\|_{L^{2}} \\
  &+\eps^{\frac{n}{2}+\gamma}\left\|\sqrt{\varepsilon}(\bar{v}, \bar{g}) x^{\frac{1}{2}}\right\|_{L^{\infty}} \cdot\left\|\left(\bar{v}_{y}, \bar{g}_{y}\right) x^{\frac{1}{2}}\right\|_{L^{2}} \cdot\left\|\sqrt{\varepsilon}\left(v_{x}, g_{x}\right) x^{\frac{1}{2}}\right\|_{L^{2}}\\
    \leq& \eps^{\frac{n}{2}+\gamma-\omega\left(N_{i}\right)}\|(\bar{u}, \bar{v}, \bar{h}, \bar{g})\|_{\mathcal{Z}}^{2}\cdot\|(u, v, h, g)\|_{\mathcal{Z}}.
\end{split}
\end{align}

{\bf{Step 2: Estimates for $\mathcal{W}_{2}$ (generated by second-order estimate).}}

Similar to the first step, for linear terms $\p_x(R^{u,n},R^{v,n},R^{h,n},R^{g,n})$, we have
\begin{align}\label{W2_1}
  &\eps^{-\frac{n}{2}-\gamma}\iint|\p_x R^{u,n}| \left(|u_x|\rho_2^2x^2+|v_{xy}|\rho_2^3x^3\right)
  +|\p_x R^{h,n}| \left(|h_x|\rho_2^2x^2+|g_{xy}|\rho_2^3x^3\right)\nonumber\\
  &+\eps^{1-\frac{n}{2}-\gamma}\iint|\p_x R^{v,n}| \left(|v_x|\rho_2^2x^2+|v_{xx}|\rho_2^3x^3\right)
  + |\p_x R^{g,n}| \left(|g_x|\rho_2^2x^2+|g_{xx}|\rho_2^3x^3\right)\nonumber\\
  \leq& \eps^{-\frac{n}{2}-\gamma}\left\|\p_x\left(R^{u, n}, R^{h, n}, \sqrt{\eps} R^{v, n}, \sqrt{\eps} R^{g, n}\right) x^{\frac{3}{2}+\sigma^{\prime}}\right\|_{L^{2}}\left[\left\|\left(u_x, h_x,\sqrt\eps v_x,\sqrt\eps g_x\right) x^{\frac{1}{2}}\right\|_{L^{2}}\right.\nonumber\\
  &\left. +\left\|\left(v_{xy}, g_{xy}\right) (\rho_2x)^{\frac{3}{2}}\right\|_{L^{2}}
  +\left\|\sqrt\eps\left(v_{xx}, g_{xx}\right) (\rho_2x)^{\frac{3}{2}}\right\|_{L^{2}}\right]\nonumber\\
  \leq& \eps^{\frac{1}{4}-\gamma-\kappa}\left\|x^{-1-\frac{5}{4}+2 \sigma_{n}+\kappa+\frac{3}{2}+\sigma^{\prime}}\right\|_{L^{2}}
  \|(u, v, h, g)\|_{X_{1}\cap X_{2}}\nonumber\\
  \leq& c(n) \eps^{\frac{1}{4}-\gamma-\kappa}\|(u, v, h, g)\|_{X_{1}\cap X_{2}},
\end{align}

Taking $x$ derivative to $N^{u},N^{h}$, we get
\begin{align}
\label{N^u_partial_x}
  \eps^{-(\frac{n}{2}+\gamma)}\p_x N^{u}=
  &\left|\bar{u}_{x}\right|^{2}-\left|\bar{h}_{x}\right|^{2}+\bar{v}_{x} u_{y}-\bar{g}_{x} h_{y}\nonumber\\
  &
  +\bar{u} \bar{u}_{x x}-\bar h\bar{h}_{x x}+\bar{v} u_{y x}-\bar{g} h_{y x}, \\
\label{N^h_partial_x}
  \eps^{-(\frac{n}{2}+\gamma)}\p_{x} N^{h}=
  &\bar{v}_{x} h_{y}-\bar{g}_{x} u_{y}+\bar{u} \bar{h}_{x x}-\bar{h} \bar{u}_{x x}+\bar{v} h_{y x}-\bar{g} u_{y x},
\end{align}
which give
\begin{align}\label{W2_2}
\begin{split}
  &\eps^{\frac{n}{2}+\gamma}\iint\left|\left(\bar{u}_{x}\right)^{2}-\left(\bar{h}_{x}\right)^{2}+\bar{u} \bar{u}_{x x}-\bar{h}\bar h_{x x}\right|\cdot
  (|u_{x}|\rho_{2}^{2}x^{2}+|v_{xy}|\rho_{2}^{3}x^{3}) \\
  &+\eps^{\frac{n}{2}+\gamma}\iint \left|\bar{u} \bar{h}_{x x}-\bar{h} \bar{u}_{x x}\right|\cdot (|h_{x}|\rho_{2}^{2}x^{2}+|g_{xy}|\rho_{2}^{3}x^{3})\\
  \leq& \eps^{\frac{n}{2}+\gamma} \left\|(\bar u_{x},\bar h_{x}) x^{\frac{5}{4}}\right\|_{L^{\infty}(x \geq 20)}\cdot \left\|(\bar{u}_x, \bar{h}_x) x^{\frac{1}{2}}\right\|_{L^{2}}\cdot \left(\left\|{u}_x x^{\frac{1}{2}}\right\|_{L^{2}}+\left\|{v}_{xy} x^{\frac{3}{2}}\right\|_{L^{2}}\right) \\
  &+\eps^{\frac{n}{2}+\gamma}\|(\bar{u}, \bar{h})\|_{L^{\infty}}\left\|\left(\bar{u}_{x x}, \bar{h}_{x x}\right) x^{\frac{3}{2}}\right\|_{L^{2}} \left(\left\|\left(u_{x}, h_{x}\right) x^{\frac{1}{2}}\right\|_{L^{2}}+\left\|\left(v_{xy}, g_{xy}\right) x^{\frac{3}{2}}\right\|_{L^{2}}\right).
\end{split}
\end{align}

And the remainder terms in $\p_x(N^u,N^h)$ can be handled by
\begin{equation}\label{W2_3}
\begin{aligned}
  &\eps^{\frac{n}{2}+\gamma}\iint\left|\bar{v}_{x} u_{y}-\bar{g}_{x} h_{y}+\bar{v} u_{yx}-\bar{g} h_{yx}\right|\cdot\left(|u_{x}| \rho_{2}^{2} x^{2}+|v_{xy}| \rho_{2}^{3} x^{3}\right) \\
  &+\eps^{\frac{n}{2}+\gamma} \iint\left|\bar{v}_{x} h_{y}-\bar{g}_{x} u_{y}+\bar{v} h_{yx}-\bar{g} u_{yx}\right|\cdot\left(|h_{x}| \rho_{2}^{2} x^{2}+|g_{xy}|
  \rho_2^3 x^3\right) \\
  \leq&\eps^{\frac{n}{2}+\gamma} \left[\left\|\left(u_{x},h_{x}\right) x^{\frac{1}{2}}\right\|_{L^{2}}+\left\|\left(v_{xy},g_{xy}\right) x^{\frac{3}{2}}\right\|_{L^{2}}\right]\cdot \Bigg[\|(u_{y},h_{y})\|_{L^{2}}\\
  &\times\left\|\left(\bar{v}_{x}, \bar{g}_{x}\right) x^{\frac{3}{2}}\right\|_{L^\infty(x\geq 20)}+ \left\|(\bar{v},\bar{g}) x^{\frac{1}{2}}\right\|_{L^{\infty}}\cdot\left\|\left(u_{xy}, h_{xy}\right) x\right\|_{L^2}\Bigg].
\end{aligned}
\end{equation}

Taking $x$ derivative to $N^{v},N^{g}$, one can obtain
\begin{align}
\label{N^v_partial_x}
  \eps^{-(\frac{n}{2}+\gamma)}\p_x N^{v}
  =\bar{u}_{x} \bar{v}_{x}+\bar{v}_{x} \bar{v}_{y}-\bar{h}_{x} \bar{g}_{x}-\bar{g}_{x} \bar{g}_{y}+\bar{u} \bar{v}_{x x}+\bar{v} \bar{v}_{x y}-\bar{h} \bar{g}_{xx}-\bar{g} \bar{g}_{xy}, \\
\label{N^g_partial_x}
  \eps^{-(\frac{n}{2}+\gamma)}\p_{x} N^{g}
  =\bar{u}_{x} \bar{g}_{x}+\bar{v}_{x} \bar{g}_{y}-\bar{h}_{x} \bar{v}_{x}-\bar{g}_{x} \bar{v}_{y}+\bar{u} \bar{g}_{x x}+\bar{v} \bar{g}_{x y}-\bar{h} \bar{v}_{x x}-\bar{g} \bar{v}_{x y},
\end{align}
which yield that
\begin{align}\label{W2_4}
\begin{split}
  &\eps\iint\left[|\p_x N^{v}|\cdot(|v_x|\rho_{2}^{2} x^{2}+|v_{xx}|\rho_{2}^{3} x^{3})+|\p_x N^{g}|\cdot(|g_x|\rho_{2}^{2} x^{2}+|g_{xx}|\rho_{2}^{3} x^{3})\right]\\
  \leq& \eps^{\frac{n}{2}+\gamma}\left[\left\|\sqrt{\eps}\left(v_{x},g_{x}\right) x^{\frac{1}{2}} \right\|_{L^{2}}
  +\left\|\sqrt{\eps}\left(v_{xx},g_{xx}\right) x^{\frac{3}{2}} \right\|_{L^{2}}\right]
  \cdot\left[\left\|\left(\bar{u}_{x},\bar{h}_{x}\right) x^{\frac{5}{4}}\right\|_{L^{2}(x\geq 20)} \right.\\
  &\left. \cdot\left\|\sqrt\eps(\bar{v}_{x},\bar{g}_{x})x^{\frac{1}{2}}\right\|_{L^{2}}
  +\left\|\sqrt\eps\left(\bar{v}_{x},\bar{g}_{x}\right) x^{\frac{3}{2}}\right\|_{L^\infty}
  \cdot\left\|\left(\bar{v}_{y},\bar{g}_{y}\right) x^{\frac{1}{2}}\right\|_{L^{2}}
  +\left\|(\bar{u}, \bar{h}) x^{\frac{1}{4}}\right\|_{L^{\infty}} \right.\\
  &\left.\cdot\left\|\sqrt\eps\left(\bar{v}_{x x}, \bar{g}_{x x}\right) x^{\frac{3}{2}}\right\|_{L^{2}}
   +\left\| \sqrt{(\bar{v}}, \bar{g}) x^{\frac{1}{2}}\right\|_{L^{\infty}}
  \cdot\left\|\left(\bar{v}_{x y}, \bar{g}_{x y}\right) \gamma^{\frac{3}{2}}\right\|_{L^{2}}\right].
\end{split}
\end{align}

{\bf{Step 3: Estimates for $\mathcal{W}_{3}$ (generated by third-order estimate).}}

For the linear terms $\p_{xx}(R^{u,n},R^{v,n},R^{h,n},R^{g,n})$, we have
\begin{align}\label{W3_1}
\begin{split}
  &\eps^{-\frac{n}{2}-\gamma}\iint |\p_x^2 R^{u,n}| \left(|u_{xx}|\rho_3^4x^4+|v_{xxy}|\rho_3^5x^5\right)\\
  &+\eps^{-\frac{n}{2}-\gamma}\iint |\p_x^2 R^{h,n}| \left(|h_{xx}|\rho_3^4x^4+|g_{xxy}|\rho_3^5x^5\right)\\
  &+\eps^{1-\frac{n}{2}-\gamma}\iint |\p_x^2 R^{v,n}| \left(|v_{xx}|\rho_3^4x^4+|v_{xxx}|\rho_3^5x^5\right)\\
  &+\eps^{1-\frac{n}{2}-\gamma}\iint |\p_x^2 R^{g,n}| \left(|g_{xx}|\rho_3^4x^4+|g_{xxx}|\rho_3^5x^5\right)\\
  \leq& \eps^{-\frac{n}{2}-\gamma}\left\|\p_x^2\left(R^{u, n}, R^{h, n}, \sqrt{\eps} R^{v, n}, \sqrt{\eps} R^{g, n}\right) x^{\frac{5}{2}+\sigma^{\prime}}\right\|_{L^{2}}\\
  &\cdot\left[\left\|\left(u_{xx}, h_{xx},\sqrt\eps v_{xx},\sqrt\eps g_{xx}\right) (\rho_3x)^{\frac{5}{2}}\right\|_{L^{2}}\right.\\
  &\left. +\left\|\left(v_{xxy}, g_{xxy}\right) (\rho_3x)^{\frac{3}{2}}\right\|_{L^{2}}
  +\left\|\sqrt\eps\left(v_{xxx}, g_{xxx}\right) (\rho_3x)^{\frac{5}{2}}\right\|_{L^{2}}\right]\\
  \leq& \eps^{\frac{1}{4}-\gamma-\kappa}\left\|x^{-2-\frac{5}{4}+2 \sigma_{n}+\kappa+\frac{5}{2}+\sigma^{\prime}}\right\|_{L^{2}}
  \|(u, v, h, g)\|_{X_{1}\cap X_{2}\cap X_{3}}\\
  \leq& c(n) \eps^{\frac{1}{4}-\gamma-\kappa}\|(u, v, h, g)\|_{X_{1}\cap X_{2}\cap X_{3}}.
\end{split}
\end{align}

Taking $x$ derivative to $\p_xN^{u},\p_xN^{h}$, we get
\begin{align}
\label{N^u_partial_xx}
  \eps^{-(\frac{n}{2}+\gamma)}\p_{xx} N^{u}=
  &3 \bar{u}_{x} \bar{u}_{x x}-3 \bar{h}_{x} \bar{h}_{x x}+\bar{u} \bar{u}_{x x x}-\bar{h} \bar{h}_{x x x} \nonumber\\
  &+2 \bar{v}_{x} u_{x y}-2 \bar{g}_{x} h_{x y}+\bar{v}_{x x} u_{y}-\bar{g}_{x x} h_{y}+\bar{v} u_{x x y}-\bar{g} h_{x x y}, \\
\label{N^h_partial_xx}
  \eps^{-(\frac{n}{2}+\gamma)}\p_{xx} N^{h}=
  &\bar{u}_{x} \bar{h}_{x x}-\bar{h}_{x} \bar{u}_{x x}+\bar{u} \bar{h}_{x x x}-\bar{h} \bar{u}_{x x x} \nonumber\\
  &+2 \bar{v}_{x} h_{x y}-2 \bar{g}_{x} u_{x y}+\bar{v}_{x x} h_{y}-\bar{g}_{x x} u_{y}+\bar{v} h_{x x y}-\bar{g} u_{x x y}.
\end{align}

Then we have
\begin{align}\label{W3_2}
\begin{split}
  & \eps^{\frac{n}{2}+\gamma}\iint\left|3 \bar{u}_{x} \bar{u}_{x x}-3 \bar{h}_{x} \bar{h}_{x x}+\bar{u} \bar{u}_{x x x}-\bar{h} \bar{h}_{x x x}\right|
  \cdot \left(|u_{x x}| \rho_{3}^{4} x^{4}+|v_{xxy}| \rho_{3}^{5} x^{5}\right) \\
  &+\eps^{\frac{n}{2}+\gamma}\iint \left|\bar{u}_{x} \bar{h}_{x x}-\bar{h}_{x} \bar{u}_{x x}+\bar{u} \bar{h}_{x x x}-\bar{h} \bar{u}_{x x x}\right|
  \cdot \left(|h_{x x}| \rho_{3}^{4} x^{4}+|g_{xxy}| \rho_{3}^{5} x^{5}\right)  \\
  \leq& \eps^{\frac{n}{2}+\gamma}\left[\left\|\left(u_{xx}, h_{xx}\right) x^{\frac{3}{2}}\right\|_{L^{2}}+\left\|\left(v_{xxy}, g_{xxy}\right) x^{\frac{5}{2}}\right\|_{L^{2}}\right]\cdot
  \left[\left\|\left(\bar{u}_{x}, \bar{h}_{x}\right) x^{\frac{5}{4}}\right\|_{L^\infty} \right.\\
  &\left. \cdot\left\|\left(\bar{u}_{xx}, \bar{h}_{xx}\right) x^{\frac{3}{2}}\right\|_{L^{2}}
  +\left\|(\bar{u}, \bar{h}) x^{\frac{1}{4}}\right\|_{L^\infty}\left\|\left(\bar{u}_{xxx}, \bar{h}_{xxx}\right) x^{\frac{5}{2}}\right\|_{L^{2}}\right].
\end{split}
\end{align}

For the remaining terms in $\p_x^2(N^u,N^h)$, there holds that
\begin{align}\label{W3_3}
\begin{split}
  & \eps^{\frac{n}{2}+\gamma}\iint\left|2 \bar{v}_{x} u_{x y}-2 \bar{g}_{x} h_{x y}+\bar{v}_{x x} u_{y}-\bar{g}_{x x} h_{y}\right|\cdot
  \left( |u_{x x}| \rho_{3}^{4} x^{4}+ |v_{xxy}| \rho_{3}^{5} x^{5}\right)\\
  &+ \eps^{\frac{n}{2}+\gamma}\iint\left|\bar{v} u_{x x y}-\bar{g} h_{x x y}\right|\cdot
  \left( |u_{x x}| \rho_{3}^{4} x^{4}+ |v_{xxy}| \rho_{3}^{5} x^{5}\right)\\
  &+ \eps^{\frac{n}{2}+\gamma}\iint\left|-2 \bar{v}_{x} h_{x y}-2 \bar{g}_{x} u_{x y}+\bar{v}_{x x} h_y-\bar{g}_{x x} u_{y}\right| \cdot
  \left( |h_{x x}| \rho_{3}^{4} x^{4}+ |g_{xxy}| \rho_{3}^{5} x^{5}\right)\\
  &+ \eps^{\frac{n}{2}+\gamma}\iint\left|\bar{v} h_{x x y}-\bar{g} u_{x x y}\right| \cdot
  \left( |h_{x x}| \rho_{3}^{4} x^{4}+ |g_{xxy}| \rho_{3}^{5} x^{5}\right)\\
  \leq& \eps^{\frac{n}{2}+\gamma}\left[\left\|\left(u_{x x},h_{x x}\right) x^{\frac{3}{2}}\right\|_{L^{2}}+\left\|\left(v_{xxy},g_{xxy}\right)x^{\frac{5}{2}}\right\|_{L^{2}}\right]
  \cdot\left[\left\|(\bar v_x, \bar g_x) x^{\frac{3}{2}}\right\|_{L^{\infty}}\right.\\
  &\left. \cdot\left\|\left(h_{xy}, u_{xy}\right) x\right\|_{L^{2}}+\left\|(\bar{v}, \bar{g}) x^{\frac{1}{2}}\right\|_{L^{\infty}} \cdot\left\|\left(u_{xxy}, h_{xxy}\right) x^{2}\right\|_{L^2(x\geq 20)}\right.\\
  &\left. +\sup_{x\geq 20}\left\|\left(u_{y}, h_{y}\right) x^{\frac{1}{2}}\right\|_{L_{y}^{2}} \left( \int_{20}^{\infty} x^{4}\left\|\left(\bar v_{x x}, \bar{g}_{x x}\right)\right\|_{L^\infty}^{2}\right)^{\frac{1}{2}} \right].
\end{split}
\end{align}

Taking $x$ derivative to $\p_xN^{v},\p_xN^{g}$, we get
\begin{align}
\label{N^v_partial_xx}
  \eps^{-(\frac{n}{2}+\gamma)}\p_{xx} N^{v}=
  &\left(\bar{u}_{x} \bar{v}_{x x}+\bar{v}_{x} \bar{v}_{x y}-\bar{h}_{x} \bar{g}_{x x}-\bar{g}_{x} \bar{g}_{x y}\right)\nonumber\\
  &+\bar{u} \bar{v}_{x x x}+\bar{v} \bar{v}_{x x y}-\bar{h} \bar{g}_{x x x}-\bar{g} \bar{g}_{x x y},\\
\label{N^g_partial_xx}
  \eps^{-(\frac{n}{2}+\gamma)}\p_{xx} N^{g}=
  &3\left(\bar{u}_{x} \bar{g}_{x x}+\bar{v}_{x} \bar{g}_{x y}-\bar{h}_{x} \bar{v}_{x x}-\bar{g}_{x} \bar{v}_{x y}\right)\nonumber\\
  &+\bar{u} \bar{g}_{x x x}+\bar{v} \bar{g}_{x x y}-\bar{h} \bar{v}_{x x x}-\bar{g} \bar{v}_{x x y}.
\end{align}

Then it follows that
\begin{align}\label{W3_4}
\begin{split}
  &\iint \left[|\p_{xx}N^{v} |\cdot \left(|v_{xx}| \rho_{3}^{4} x^{4}+|v_{xxx}| \rho_{3}^{5} x^{5}\right)
  +|\p_{xx} N^{g}| \cdot \left(|g_{xx}| \rho_{3}^{4} x^{4}+|g_{xxx}| \rho_{3}^{5} x^{5}\right)\right]\\
  \leq& \eps^{\frac{n}{2}+\gamma} \left[\left\|\sqrt\eps\left(v_{xx}, g_{xx}\right) x^{\frac{3}{2}}\right\|_{L^{2}}
  +\left\|\sqrt\eps\left(v_{xxx}, g_{xxx}\right) x^{\frac{5}{2}}\right\|_{L^{2}}\right]
  \cdot\left[\left\|\sqrt{\eps}\left(\bar v_{x}, \bar{g}_{x}\right) x^{\frac{3}{2}}\right\|_{L^\infty}\right.\\
  &\left. \left\|\left(\bar{u}_{x x}, \bar{h}_{x x}\right) x^{\frac{3}{2}}\right\|_{L^{2}}
  +\left\|\left(\bar{u}_{x}, \bar{h}_{x}\right) x^{\frac{5}{4}}\right\|_{L^{\infty}} \cdot\left\|\sqrt{\varepsilon}\left(\bar{v}_{x x}, \bar{g}_{x x}\right) x^{\frac{3}{2}}\right\|_{L^{2}}+\left\|\left(\bar{u},\bar{h}\right) x^{\frac{1}{4}}\right\|_{L^\infty} \right.\\
  &\left. \cdot\left\|\sqrt{\eps}\left(v_{xxx},{g}_{xxx}\right) x^{\frac{5}{2}}\right\|_{L^{2}}+\left\|\sqrt{\eps}(\bar{v}, \bar{g}) x^{\frac{1}{2}}\right\|_{L^{\infty}} \cdot\left\|\left(\bar{v}_{xxy} , \bar{g}_{xxy}\right) x^{\frac{5}{2}}\right\|_{L^{2}}\right].
\end{split}
\end{align}

Summing up, one concludes that the estimate \eqref{W123_estimate} follows from the above estimates.
\end{proof}

\begin{remark}\label{remark_n>2_S6}
It is emphasized that the estimate \eqref{R^unterm_estimate_L2} of $(R^{u, n}, R^{h, n}, R^{v, n},R^{g, n})$ has played the key roles in the analysis. Precisely,  observe that in  \eqref{W1_1},\eqref{W2_1},\eqref{W3_1}, the decay rates with respect to $x$ variable ensure that the first term is integrable in $L_x^2$ norm. Actually, it is realized by expanding the approximate solutions to $\eps^{\frac{n}{2}}$-order for $n\geq 2$ in \eqref{expansion}.
\end{remark}

\begin{remark}\label{remark_linear_vu_y}
In estimate \eqref{W1_3}, the differential structures of $\bar v\p_y|u|^2$, $\bar v\p_y|h|^2$, $\bar g\p_y(u h)$ are important in the proof, which come from the linearized terms as stated in Remark \ref{remark_linear_vu_y_pre}. Otherwise, if considering the linearized terms with structures of $\bar v \bar u_y, \bar g \bar h_y, \bar v \bar h_y, \bar g \bar u_y$, for example, we do not know how to control the last term as follows
\begin{align*}
  \iint \bar v\bar u_y\cdot u
  \leq \Vert \bar v x^{\frac{1}{2}}\Vert_{L^\infty}\cdot \Vert \bar u_y\Vert_{L^2}\cdot \Vert u x^{-\frac{1}{2}}\Vert_{L^2},
\end{align*}
in which Hardy inequality \eqref{Hardy_remainder_x} is invalid.
\end{remark}

Finally, putting the above arguments together, we get the {\it a priori estimate} of this subsection in the following theorem.
\begin{theorem}\label{theorem_priori_estimate}
 Let $\omega(N_{i})$, $N_i$ be the constants determined in Lemma \ref{Z-embeding}, \eqref{constants} respectively. For any fixed positive number $\beta>0$ and $n\in \mathbb{N}$ satisfying \eqref{n-choice}, assume that $\|(\bar u,\bar v,\bar h,\bar g)\|_{ \mathcal{Z}(\Omega^\beta)}\leq 1$, then for small viscosity coefficient $\eps\ll \eps_*$, the solution $(u,v,h,g)\in \mathcal{Z}(\Omega^\beta)$ to the problem \eqref{u_eps_linearized} with \eqref{u_epsboundary_omega_N} satisfies the following estimate
  \begin{equation}\label{priori_estimate}
  \|(u,v,h,g)\|_{ \mathcal{Z}(\Omega^\beta)}^2\lesssim \eps^{\frac{1}{4}-\gamma-\kappa}+\eps^{\frac{n}{2}-\omega(N_i)}\|(\bar u,\bar v,\bar h,\bar g)\|_{ \mathcal{Z}(\Omega^\beta)}^4,
\end{equation}
in which $\gamma, \kappa$ are arbitrary constants with $0\leq\gamma<\frac{1}{4}, 0<\kappa\ll 1$ and $\gamma+\kappa<\frac{1}{4}$.
\end{theorem}

\subsection{Nonlinear stability: Proof of Theorem \ref{maintheorem}}\label{sec6.2} Based on the estimate \eqref{priori_estimate} obtained in the last subsection, we are ready to prove Theorem \ref{maintheorem}. Precisely, we will prove global-in-$x$ theory for nonlinear system \eqref{u_epssystem} with \eqref{u_epsboundary}.

First, the analysis will be centered on the following vorticity formulation of the system \eqref{u_epssystem}
\begin{equation}\label{u_epssystem_stream}
\begin{cases}
  \Delta_\eps^2\phi+\p_y S_u-\eps \p_x S_v=\p_y F_u-\eps\p_x F_v,\\
  \Delta_\eps^2\psi+\p_y S_h-\eps \p_x S_g=\p_y F_h-\eps\p_x F_g,
\end{cases}
\end{equation}
by introducing the stream functions $\phi,\psi$ satisfying
\begin{equation}\label{stream}
  (u,v,h,g)=(-\phi_y,\phi_x,-\psi_y,\psi_x).
\end{equation}
In addition, one can deduce the boundary conditions for the stream functions
\begin{equation}\label{u_epsboundary_stream}
\begin{cases}
  (\phi,\psi)|_{y=0}=(\phi,\psi)|_{y\rightarrow\infty}=(\phi,\psi)|_{x=1}=(\phi,\psi)|_{x\rightarrow\infty}=(0,0),\\
  (\phi_x,\psi_x)|_{x=1}=(\phi_x,\psi_x)|_{x\rightarrow\infty}=(0,0).
\end{cases}
\end{equation}

Inspired by Iyer \cite{Iyerglobal3}, we will first construct the approximate solution $(u^{\alpha,\beta}$, $v^{\alpha,\beta}$, $h^{\alpha,\beta}$, $g^{\alpha,\beta})$ to system \eqref{u_epssystem} on fixed domain $\Omega^\beta:=\{(x,y)|x\geq 1,0\leq y<\beta\}$ by solving an auxiliary system for the stream functions $\phi$ and $\psi$ with positive terms $\alpha W(\phi), \alpha W(\psi)$ uniform in $\alpha$ as follows. And then, solution $(u,v,h,g)$ to system \eqref{u_epssystem} can be obtained by taking the limits $\alpha\rightarrow 0$ and $\beta\rightarrow\infty$ up to subsequences. Note that the desired compactness is obtained from the positive terms $W(\phi),W(\psi)$.

Now we introduce the following auxiliary system for stream functions and construct its solution
\begin{equation}\label{u_epssystem_auxiliary}
\begin{cases}
  \Delta_\eps^2\phi+\alpha W(\phi)+S_1(\phi,\psi; v, g)=\p_y\tilde F_u-\eps\p_xF_v,\\
  \Delta_\eps^2\psi+\alpha W(\psi)+S_2(\phi,\psi; v, g)=\p_y\tilde F_h-\eps\p_xF_g,\\
  (\phi,\psi)|_{y=0}=(\phi,\psi)|_{y=\beta}=(0,0),\\
  (\phi,\psi)|_{x=1}=(\phi,\psi)|_{x\rightarrow\infty}=(0,0),\\
  (\phi_x,\psi_x)|_{x=1}=(\phi_x,\psi_x)|_{x\rightarrow\infty}=(0,0),
\end{cases}
\end{equation}
where
\begin{align}\label{W_phi}
\begin{split}
  W(\phi)=&\phi x^{2m}-\p_y^2\phi x^{2m+2}-\p_x(\phi_xx^{2m+2})+\p_y^4\phi x^{2m+4}\\
  &+\p_x(\p_y^2\phi_x x^{2m+4})+\p_x^2(\p_x^2\phi x^{2m+4})
\end{split}
\end{align}
and
\begin{align}\label{W_psi}
\begin{split}
  W(\psi)=&\psi x^{2m}-\p_y^2\psi x^{2m+2}-\p_x(\psi_xx^{2m+2})+\p_y^4\psi x^{2m+4}\\
  &+\p_x(\p_y^2\psi_x x^{2m+4})+\p_x^2(\p_x^2\psi x^{2m+4})
\end{split}
\end{align}
with the uniform coefficient $\alpha>0$. Here, the linear terms $S_1(\phi,\psi; v, g),S_2(\phi,\psi; v, g)$ are denoted as
\begin{align}
\label{S_1}
  S_1(\phi,\psi; v, g)
  =&\p_y\left[-u_s\phi_{xy}-\phi_y\p_x u_s-(v_s-\eps^{\frac{n}{2}+\gamma} v)\phi_{yy}+\phi_x\p_y u_s\right]\nonumber\\
  &+\p_y\left[h_s\psi_{xy}+\psi_y\p_x h_s+(g_s-\eps^{\frac{n}{2}+\gamma} g)\psi_{yy}-\psi_x\p_y h_s\right]\nonumber\\
  &-\eps\p_x\left[u_s\phi_{xx}-\phi_y\p_x v_s+v_s\phi_{xy}+\phi_y\p_y v_s\right]\nonumber\\
  &+\eps\p_x\left[h_s\psi_{xx}-\psi_y\p_x g_s+g_s\psi_{xy}+\psi_y\p_y g_s\right],\\
\label{S_2}
  S_2(\phi,\psi; v, g)
  =&\p_y\left[-u_s\psi_{xy}-\phi_y\p_x h_s-(v_s-\eps^{\frac{n}{2}+\gamma} v)\psi_{yy}+\phi_x\p_y h_s\right]\nonumber\\
  &+\p_y\left[h_s\phi_{xy}+\psi_y\p_x u_s+(g_s-\eps^{\frac{n}{2}+\gamma} g)\phi_{yy}-\psi_x\p_y u_s\right]\nonumber\\
  &-\eps\p_x\left[u_s\psi_{xx}-\phi_y\p_x g_s+v_s\psi_{xy}+\phi_x\p_y g_s\right]\nonumber\\
  &+\eps\p_x\left[h_s\phi_{xx}-\psi_y\p_x v_s+g_s\phi_{xy}+\psi_x\p_y v_s\right].
\end{align}
In other words, we have
\begin{align*}
  &S_1(\phi,\psi; v, g)=\p_yS_u-\eps\p_xS_v+\eps^{\frac{n}{2}+\gamma}\p_y( v\phi_{yy}- g\psi_{yy}),\\
  &S_2(\phi,\psi; v, g)=\p_yS_h-\eps\p_xS_g+\eps^{\frac{n}{2}+\gamma}\p_y( v\psi_{yy}- g\phi_{yy}).
\end{align*}

The nonlinear terms $F_v,F_g$ are defined as in \eqref{F_u} and $\tilde F_u,\tilde F_h$ are given by
\begin{align*}
\begin{cases}
   \tilde F_u=\eps^{-\frac{n}{2}-\gamma}R^{u,n}-\eps^{\frac{n}{2}+\gamma}(u\p_x u-h\p_x h),\\
   \tilde F_h=\eps^{-\frac{n}{2}-\gamma}R^{h,n}-\eps^{\frac{n}{2}+\gamma}(u\p_x h-h\p_x u).
\end{cases}
\end{align*}

\begin{proposition}\label{existence_prop_pre}
There exists a solution $(u^{\alpha,\beta},v^{\alpha,\beta},h^{\alpha,\beta},g^{\alpha,\beta})$ determined by auxiliary system \eqref{u_epssystem_auxiliary} and \eqref{stream} satisfying the following estimate independent of the $\alpha$ and $\beta$
\begin{equation}\label{prop_estimate}
  \|(u^{\alpha,\beta},v^{\alpha,\beta},h^{\alpha,\beta},g^{\alpha,\beta})\|_{ \mathcal{Z}(\Omega^\beta)}\lesssim \eps^{\frac{1}{4}-\gamma-\kappa},
\end{equation}
where $\gamma, \kappa$ are arbitrary constants with $0\leq\gamma<\frac{1}{4},0<\kappa\ll1$ and $\gamma+\kappa<\frac{1}{4}$.
\end{proposition}
\begin{proof}
Inspired by \cite{Iyerglobal3}, let us consider the map $\mathscr{M}_\alpha$ in domain $\Omega^\beta$ as
\begin{align}\label{map}
 &\mathscr{M}_{\alpha,\beta}\left(\overline{u}^{\alpha,\beta},\overline{v}^{\alpha,\beta},\overline{h}^{\alpha,\beta},\overline{g}^{\alpha,\beta}\right)=\left(u^{\alpha,\beta},v^{\alpha,\beta},h^{\alpha,\beta},g^{\alpha,\beta}\right)\nonumber\\
 \Longleftrightarrow&\mathscr{L}_{\alpha,\beta}\left({u}^{\alpha,\beta},{v}^{\alpha,\beta},{h}^{\alpha,\beta},{g}^{\alpha,\beta}\right)
 =(\tilde{F}_u, F_v,\tilde{F}_h,F_g)\left(\overline{u}^{\alpha,\beta},\overline{v}^{\alpha,\beta},\overline{h}^{\alpha,\beta},\overline{g}^{\alpha,\beta}\right),
\end{align}
where problem \eqref{map} corresponds to the following problem of vorticity form
\begin{equation}\label{u_epssystem_linear}
\begin{cases}
  \Delta_\eps^2\phi^{\alpha,\beta}+\alpha W(\phi^{\alpha,\beta})+S_1(\phi^{\alpha,\beta},\psi^{\alpha,\beta};\bar{v}^{\alpha,\beta},\bar{g}^{\alpha,\beta})\\
  \qquad =\p_y \tilde{F}_u(\overline{u}^{\alpha,\beta},\overline{v}^{\alpha,\beta},\overline{h}^{\alpha,\beta},\overline{g}^{\alpha,\beta})-\eps\p_x F_v(\overline{u}^{\alpha,\beta},\overline{v}^{\alpha,\beta},\overline{h}^{\alpha,\beta},\overline{g}^{\alpha,\beta}),\\
  \Delta_\eps^2\psi^{\alpha,\beta}+\alpha W(\psi^{\alpha,\beta})+S_2(\phi^{\alpha,\beta},\psi^{\alpha,\beta};\bar{v}^{\alpha,\beta},\bar{g}^{\alpha,\beta})\\
  \qquad=\p_y\tilde{F}_h(\overline{u}^{\alpha,\beta},\overline{v}^{\alpha,\beta},\overline{h}^{\alpha,\beta},\overline{g}^{\alpha,\beta})-\eps\p_x F_g(\overline{u}^{\alpha,\beta},\overline{v}^{\alpha,\beta},\overline{h}^{\alpha,\beta},\overline{g}^{\alpha,\beta}),\\
  (\phi^{\alpha,\beta},\psi^{\alpha,\beta})|_{y=0}=(\phi^{\alpha,\beta},\psi^{\alpha,\beta})|_{y=\beta}=(0,0),\\
  (\phi^{\alpha,\beta},\psi^{\alpha,\beta})|_{x=1}=(\phi^{\alpha,\beta},\psi^{\alpha,\beta})|_{x\rightarrow\infty}=(0,0),\\
  (\phi_x^{\alpha,\beta},\psi_x^{\alpha,\beta})|_{x=1}=(\phi_x^{\alpha,\beta},\psi_x^{\alpha,\beta})|_{x\rightarrow\infty}=(0,0) ,
\end{cases}
\end{equation}
where $\p_y\tilde{F}_u-\eps\p_x F_v\in H^{-1}$ and $\p_y\tilde{F}_h-\eps\p_x F_g\in H^{-1}$. Then the fixed point of \eqref{u_epssystem_linear} yields  a solution to  \eqref{u_epssystem_auxiliary}.

The linear theory for \eqref{u_epssystem_linear}
 can be obtained by the standard arguments, see also \cite{Iyerglobal3}. For the nonlinear theory, with the help of \eqref{priori_estimate}, following the similar arguments as those in \cite{Iyerglobal3}, one can deduce that the Schaefer's fixed point theorem admits a fixed point of the map $\mathscr{M}_{\alpha,\beta}$ uniform in $\alpha,\beta$, i.e., the solution $(u^{\alpha,\beta},v^{\alpha,\beta},h^{\alpha,\beta},g^{\alpha,\beta})$ to auxiliary system \eqref{u_epssystem_auxiliary}. We do not repeat the details.
\end{proof}

Next, we have the following theorem about the solution to \eqref{u_epssystem} with \eqref{u_epsboundary}.
\begin{theorem}\label{existencetheorem}
Under the assumptions of Theorem \ref{maintheorem}, for small viscosity coefficient $\eps\ll \eps_*$, there exists unique global remainder solution $(u,v,h,g)$ for the system \eqref{u_epssystem} with boundary conditions \eqref{u_epsboundary} in the domain $\Omega=[1,\infty)\times\mathbb{R}_+$ satisfying the estimate
\begin{equation}\label{theorem_estimate}
  \|(u,v,h,g)\|_{ \mathcal{Z}(\Omega)}\lesssim \eps^{\frac{1}{4}-\gamma-\kappa},
\end{equation}
where $\gamma, \kappa$ are arbitrary constants with $0\leq\gamma<\frac{1}{4},0<\kappa\ll1$ and $\gamma+\kappa<\frac{1}{4}$.
\end{theorem}

\begin{proof}
The existence theory to system \eqref{u_epssystem} with boundary conditions \eqref{u_epsboundary} can be obtained by compactness methods, via taking limits $\alpha\rightarrow 0$ and $\beta\rightarrow\infty$ for solution $(u^{\alpha,\beta},v^{\alpha,\beta},h^{\alpha,\beta},g^{\alpha,\beta})$ constructed in Proposition \ref{existence_prop_pre}.

It remains to show the uniqueness. Suppose that $(u_1,v_1,h_1,g_1)$ and $(u_2,v_2,h_2,g_2)$ are both solutions to boundary value problem \eqref{u_epssystem}, \eqref{u_epsboundary}. Introduce the new unknowns
\begin{align}
  (\widehat{u},\widehat{v},\widehat{h},\widehat{g},\widehat{p})=(u_1-u_2,v_1-v_2,h_1-h_2,g_1-g_2,p_1-p_2),
\end{align}
which solve the following boundary value problem
\begin{align}\label{u_epssystem_uniqueness}
\begin{cases}
  -\Delta_\eps \widehat{u}+S_u(\widehat{u},\widehat{v},\widehat{h},\widehat{g})+\widehat{p}_x=\widehat{F}_u,\\
  -\Delta_\eps \widehat{v}+S_v(\widehat{u},\widehat{v},\widehat{h},\widehat{g})+\frac{\widehat{p}_y}{\eps }=\widehat{F}_v,\\
  -\Delta_\eps \widehat{h}+S_h(\widehat{u},\widehat{v},\widehat{h},\widehat{g})=\widehat{F}_h,\\
  -\Delta_\eps \widehat{g}+S_g(\widehat{u},\widehat{v},\widehat{h},\widehat{g})=\widehat{F}_g,\\
  \p_x \widehat{u}+\p_y \widehat{v}=\p_x \widehat{h}+\p_y \widehat{g}=0,\\
  (\widehat{u},\widehat{v},\widehat{h},\widehat{g})|_{y=0}=(\widehat{u},\widehat{v},\widehat{h},\widehat{g})|_{x=1}=(0,0,0,0),
\end{cases}
\end{align}
in which the source terms $\widehat{F}_u,\widehat{F}_v,\widehat{F}_h,\widehat{F}_g$ are denoted as
\begin{align}\label{F^u_uniqueness}
\begin{cases}
   \widehat{F}_u=&\eps^{\frac{n}{2}+\gamma}(u_1\p_x u_1-u_2\p_x u_2 +v_1\p_y u_1-v_2\p_y u_2\\
   &\qquad\quad -h_1\p_x h_1+h_2\p_x h_2 -g_1\p_y h_1 +g_2\p_y h_2),\\
   \widehat{F}_v=&\eps^{\frac{n}{2}+\gamma}(u_1\p_x v_1 -u_2\p_x v_2+v_1\p_y v_1 -v_2\p_y v_2 \\
   &\qquad\quad -h_1\p_x g_1 +g_1\p_y g_1-h_2\p_x g_2 +g_2\p_y g_2),\\
   \widehat{F}_h=&\eps^{\frac{n}{2}+\gamma}(u_1\p_x h_1-u_2\p_x h_2 +v_1\p_y h_1-v_2\p_y h_2\\
   &\qquad\quad -h_1\p_x u_1+h_2\p_x u_2-g_1\p_y u_1+g_2\p_y u_2),\\
   \widehat{F}_g=&\eps^{\frac{n}{2}+\gamma}(u_1\p_x g_1-u_2\p_x g_2 +v_1\p_y g_1-v_2\p_y g_2\\
   &\qquad\quad -h_1\p_x v_1+h_2\p_x v_2-g_1\p_y v_1+g_2\p_y v_2).
\end{cases}
\end{align}

Observe that the structure of system \eqref{u_epssystem_uniqueness} is similar to that of \eqref{u_epssystem}, thus the analysis in Subsection \ref{sec6.1} can be also applied here, we do not repeat the detailed arguments for simplicity.

One should pay attention to the only difference that a weaker weight $x^{-b}$ is adopted to $(\widehat{u},\widehat{v},\widehat{h},\widehat{g})$ for closing the estimate. Here, the parameter $b$ is chosen to satisfy $\delta\ll b<1$. It ensures that the Hardy-type inequality \eqref{Hardy_remainder_x} is available for controlling the different terms $v_1\p_y u_1-v_2\p_y u_2,-g_1\p_y h_1 +g_2\p_y h_2$ in $\widehat{F}_u$ and $v_1\p_y h_1-v_2\p_y h_2,-g_1\p_y u_1+g_2\p_y u_2 $ in $\widehat{F}_h$.

Thus, corresponding to the $\mathcal{Z}$ norm, we define the following $\mathcal{Z}_b$ norm for $(\widehat{u},\widehat{v},\widehat{h},\widehat{g})$
\begin{align}\label{norm_b}
\begin{split}
  \|(\widehat{u},\widehat{v},\widehat{h},\widehat{g})\|_{ \mathcal{Z}_b}
  :=&\Vert (\widehat{u},\widehat{v},\widehat{h},\widehat{g})\Vert_{X_{1,b}\cap X_{2,b}\cap X_{3,b}}
  +\eps^{N_1}\Vert (\widehat{u},\widehat{v},\widehat{h},\widehat{g})\Vert_{Y_{1,b}}\\
  &+\eps^{N_2}\Vert (\widehat{u},\widehat{v},\widehat{h},\widehat{g})\Vert_{Y_{2,b}}
  +\|(\widehat{u},\widehat{v},\widehat{h},\widehat{g})\|_{ \mathcal{U}_b},
\end{split}
\end{align}
where
\begin{align}
\label{X1norm_b}
  \|(\widehat{u},\widehat{v},\widehat{h},\widehat{g})\|_{X_{1,b}}^{2}
  :=&\left\|\p_y(\widehat u,\widehat h)x^{-b}\right\|_{L^{2}}^{2}+\left\|\sqrt{\eps}\p_x(\widehat v,\widehat g) x^{\frac{1}{2}-b}\right\|_{L^{2}}^{2}+\left\|\p_y(\widehat v,\widehat g) x^{\frac{1}{2}-b}\right\|_{L^{2}}^{2},\\
\label{X2norm_b}
  \|(\widehat{u},\widehat{v},\widehat{h},\widehat{g})\|_{X_{2,b}}^{2}
  :=&\left\|\p_{x y}(\widehat u,\widehat h) \cdot \rho_{2} x^{1-b}\right\|_{L^{2}}^{2}+\left\|\sqrt{\eps} \p_{x x}(\widehat v,\widehat g)\cdot\left(\rho_{2} x\right)^{\frac{3}{2}-b}\right\|_{L^{2}}^{2}\nonumber\\
  &+\left\|\p_{x y}(\widehat v,\widehat g)\cdot\left(\rho_{2} x\right)^{\frac{3}{2}-b}\right\|_{L^{2}}^{2},\\
\label{X3norm_b}
  \|(\widehat{u},\widehat{v},\widehat{h},\widehat{g})\|_{X_{3,b}}^{2}
  :=&\left\|\p_{x x y}(\widehat u,\widehat h) \cdot\left(\rho_{3} x\right)^{2-b}\right\|_{L^{2}}^{2}
  +\left\|\sqrt{\eps} \p_{x x x}(\widehat v,\widehat g)\cdot\left(\rho_{3} x\right)^{\frac{5}{2}-b}\right\|_{L^{2}}^{2}\nonumber\\
  &+\left\|\p_{x x y}(\widehat v,\widehat g) \cdot\left(\rho_{3} x\right)^{\frac{5}{2}-b}\right\|_{L^{2}}^{2},\\
\label{Y2norm_b}
  \|(\widehat{u},\widehat{v},\widehat{h},\widehat{g})\|_{Y_{1,b}}^{2}
  :=&\left\|\p_{x y}(\widehat u,\widehat h)\cdot x^{1-b}\right\|_{L^{2}}^{2}
  +\left\|\sqrt{\eps}\p_{x x}(\widehat v,\widehat g)\cdot x^{\frac{3}{2}-b}\right\|_{L^{2}}^{2}\nonumber\\
  &+\left\|\p_{x y}(\widehat v,\widehat g)x^{\frac{3}{2}-b}\right\|_{L^{2}}^{2}+\left\|\p_{y y}(\widehat u,\widehat h)\right\|_{L^{2}(x \leq 2000)}^{2},\\
\label{Y3norm_b}
  \|(\widehat{u},\widehat{v},\widehat{h},\widehat{g})\|_{Y_{2,b}}^{2}
  :=&\left\|\p_{x x y}(\widehat u,\widehat h) \cdot \zeta_{3} x^{2-b}\right\|_{L^{2}}^{2}
  +\left\|\sqrt{\eps}\p_{x x x}(\widehat v,\widehat g) \cdot \zeta_{3} x^{\frac{5}{2}-b}\right\|_{L^{2}}^{2}\nonumber\\
  &+\left\|\p_{x x y}(\widehat v,\widehat g) \cdot \zeta_{3} x^{\frac{5}{2}-b}\right\|_{L^{2}}^{2},\\
\label{}
  \|(\widehat{u},\widehat{v},\widehat{h},\widehat{g})\|_{ \mathcal{U}_b}
  :=&\eps^{N_3}\Vert (u,h)x^{\frac{1}{4}-b}
  +\sqrt\eps (\widehat v,\widehat g)x^{\frac{1}{2}-b}\Vert_{L^\infty}\nonumber\\
  &+\eps^{N_4}\sup_{x\geq 20}\Vert (\widehat u_x,\widehat h_x)x^{\frac{5}{4}-b}+\sqrt\eps(\widehat v_x,\widehat g_x)x^{\frac{3}{2}-b}\Vert_{L^\infty}\nonumber\\
  &+\eps^{N_5}\sup_{x\geq 20}\Vert (\widehat u_y,\widehat h_y)x^{\frac{1}{2}-b}\Vert_{L_y^2}\nonumber\\
  &+\eps^{N_6}\left[\int_{20}^\infty x^{4-b}\Vert \sqrt\eps (\widehat v_{xx},\widehat g_{xx})\Vert_{L_y^\infty}^2 {\rm d}x\right]^{\frac{1}{2}}.
\end{align}

Therefore, one obtains
\begin{equation}\label{uniqueness_esitmae}
  \Vert (\widehat{u},\widehat{v},\widehat{h},\widehat{g})\Vert_{\mathcal{Z}_b}^2
  \lesssim C(b)\eps^{\frac{n}{2}+\gamma-\omega(N_i)}\Vert (\widehat{u},\widehat{v},\widehat{h},\widehat{g})\Vert_{\mathcal{Z}_b}^2,
\end{equation}
which yields that $\Vert (\widehat{u},\widehat{v},\widehat{h},\widehat{g})\Vert_{\mathcal{Z}_b}=0$. Combining the boundary conditions \eqref{u_epssystem_uniqueness}$_{6}$, we get $(\widehat{u},\widehat{v},\widehat{h},\widehat{g})=(0,0,0,0)$.

The proof is completed.
\end{proof}
Now we turn to prove Theorem \ref{maintheorem}.
\begin{proof}[Proof of Theorem \ref{maintheorem}]
Let $n$ be selected as \eqref{n-choice}. Thanks to Theorem \ref{existencetheorem}, one can deduce that the expansion \eqref{expansion} is valid globally on $\Omega$. In addition, the estimate \eqref{maintheorem_estimate}
 is the consequence of \eqref{theorem_estimate} in Theorem \ref{existencetheorem}. Therefore Theorem \ref{maintheorem} follows.
\end{proof}

\appendix
\section{The Hardy inequality}
The Hardy-type inequalities are useful in our analysis, which are the key ingredients to obtain the global-in-$x$ estimate.
\begin{lemma}\label{Hardy-type}
For $p>1$ and $\alpha>p-1$, for any measurable non-negative function $f$, it holds that
\begin{align}\label{Hardy1}
\begin{split}
  \int_0^\infty \left(\frac{1}{y}\int_y^\infty f(t){\rm d}t\right)^p y^\alpha {\rm d}y
  \leq \left(\frac{p}{\alpha+1-p}\right)^p\int_0^\infty f(y)^p y^\alpha {\rm d}y.
\end{split}
\end{align}
\end{lemma}
\begin{lemma}\label{Hardy-type-0}
For $p>1$ and $\alpha<p-1$, for any measurable non-negative function $f$, it holds that
\begin{align}\label{Hardy2}
\begin{split}
  \int_0^\infty \left(\frac{1}{y}\int_0^y f(t){\rm d}t\right)^p y^\alpha {\rm d}y
  \leq \left(\frac{p}{p-\alpha-1}\right)^p\int_0^\infty f(y)^p y^\alpha {\rm d}y.
\end{split}
\end{align}
\end{lemma}
The proof of the above lemmas can be found in \cite{hardy} and we omit the proof here.

\begin{lemma}\label{Hardy-type-1}
It holds that

(i) \textit{Weighted Hardy inequality} in $x-$direction with small constant $\sigma'<\frac{1}{2}$:
\begin{align}\label{Hardy_remainder_x}
\begin{split}
  \left\|\frac{f}{x^{1-\sigma'}}\right\|_{L^{2}}
  \leq\left\|\left\| \frac{f}{(x-1)^{1-\sigma'}}\right\|_{L_{x}^{2}}\right\|_{L_{y}^{2}}
  \lesssim\|\| f_{x}(x-1)^{\sigma'}\left\|_{L_{x}^{2}}\right\|_{L_{y}^{2}}
  \lesssim\left\|f_{x} x^{\sigma'}\right\|_{L^{2}}
\end{split}
\end{align}
for any measurable non-negative function with condition $f|_{x=1}=0$;

(ii) \textit{Hardy inequality} in $y-$direction with arbitrary constant $\alpha$:
\begin{align}\label{Hardy_remainder_y}
\begin{split}
  \left\|\frac{f}{y} x^{\alpha}\right\|_{L^{2}}
  =\left\|\left\| \frac{f}{y}\right\|_{L_{y}^{2}} x^{\alpha}\right\|_{L_{x}^{2}}\leq\left\|f_{y} x^{\alpha}\right\|_{L^{2}}
\end{split}
\end{align}
for any measurable non-negative function $f$ with $f|_{y=0}=0$.
\end{lemma}
This lemma follows directly from Lemma \ref{Hardy-type} and \ref{Hardy-type-0}, and we omit the proof here.

\smallskip
{\bf Acknowledgment.}

Ding's research is supported by the Key Project of Natural Science Foundation of China (No. 12131010), the Natural Science Foundation of China (No. 11771155, 11871005) and by the Natural Science Foundation of Guangdong Province (No. 2021A1515010249, 2021A1515010303). Ji's research is supported by the Natural Science Foundation of Guangdong Province (No. 2021A1515010249). Lin's research is partially supported by the Natural Science Foundation of China (No. 11971009).



\bigskip

\end{document}